\documentclass[12pt]{amsart}

\usepackage{fullpage}

\usepackage{amssymb,amsfonts,amsmath}

\usepackage{graphicx}
\usepackage{amssymb}
\usepackage{epstopdf}

\usepackage[all,cmtip]{xy}

\usepackage[T1]{fontenc}
\usepackage[utf8]{inputenc}
\usepackage{mathtools} 

\usepackage{subfig}

\usepackage[bbgreekl]{mathbbol}

\usepackage{tikz}

\usepackage{color}

\renewcommand{\int}{\operatorname{int}} 
\newcommand{\INT}{\operatorname{INT}}
 
\newcommand{\Ker}{\operatorname{Ker}}

\newcommand{\Arf}{\mathrm{Arf}}

\newcommand\W{{\sf W}} 
 
\newcommand\sL{{\sf L}}
\newcommand\sD{{\sf D}}
\newcommand\sK{{\sf K}}
\newcommand{\sW}{\W}

\newcommand{\N}{\mathbb{N}} 
\newcommand{\R}{\mathbb{R}}

\newcommand{\CP}{\mathbb{CP}} 
\newcommand{\RP}{\mathbb{RP}} 
\newcommand{\km}{\operatorname{km}}

\newcommand{\cT}{\mathcal{T}} 
\newcommand{\cV}{\mathcal{V}} 
\newcommand{\cW}{\mathcal{W}}

\newcommand{\imra}{\looparrowright} 
\newcommand{\sra}{\twoheadrightarrow}

\newcommand{\iinfty}{{\mathchoice
{\begin{minipage}{.15in}\includegraphics[width=.12in]{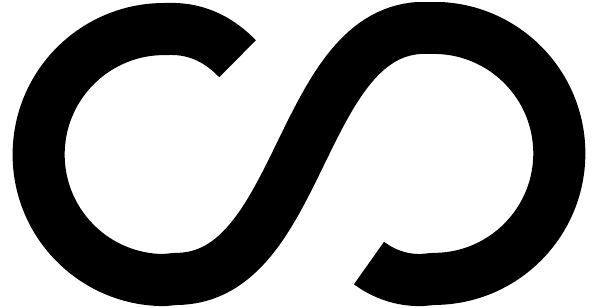}\end{minipage}}
{\begin{minipage}{.10in}\includegraphics[width=.10in]{Figures/infty2.pdf}\end{minipage}}
{\begin{minipage}{.08in}\includegraphics[width=.08in]{Figures/infty2.pdf}\end{minipage}}
{\begin{minipage}{.08in}\includegraphics[width=.08in]{Figures/infty2.pdf}\end{minipage}}
}}

\newtheorem{thm}{Theorem}[section]
      \newtheorem{lem}[thm]{Lemma}
      \newtheorem{prop}[thm]{Proposition}
      \newtheorem{cor}[thm]{Corollary}
      \newtheorem{defn}[thm]{Definition}
      
      \newtheorem{conj}[thm]{Conjecture}
      \newtheorem{rem}[thm]{Remark}
      
      \newtheorem{prob}[thm]{Open Problem}
\newtheorem{question}[thm]{Question}

\newcommand{\Z}{\mathbb{Z}}
\title{Introduction to Whitney towers}

\author[R. Schneiderman]{Rob Schneiderman}
\email{robert.schneiderman@lehman.cuny.edu}
\address{Dept. of Mathematics, Lehman College, City University of New York, Bronx, NY}

\begin{document}
\maketitle

\begin{abstract}
These introductory notes on Whitney towers in 4-manifolds, as developed in collaboration with Jim Conant and Peter Teichner, are an expansion of three expository lectures given at the Winter Braids X conference February 2020 in Pisa, Italy.
Topics presented include local manipulations of surfaces in $4$--space, fundamental definitions related to Whitney towers and their associated trees, geometric Jacobi identities,
the classification of order~$n$ twisted Whitney towers in the $4$--ball and higher-order Arf invariants, and low-order Whitney towers on $2$--spheres in $4$--manifolds and related invariants.
\end{abstract}



\begin{figure}[h]
\centerline{\includegraphics[scale=.7]{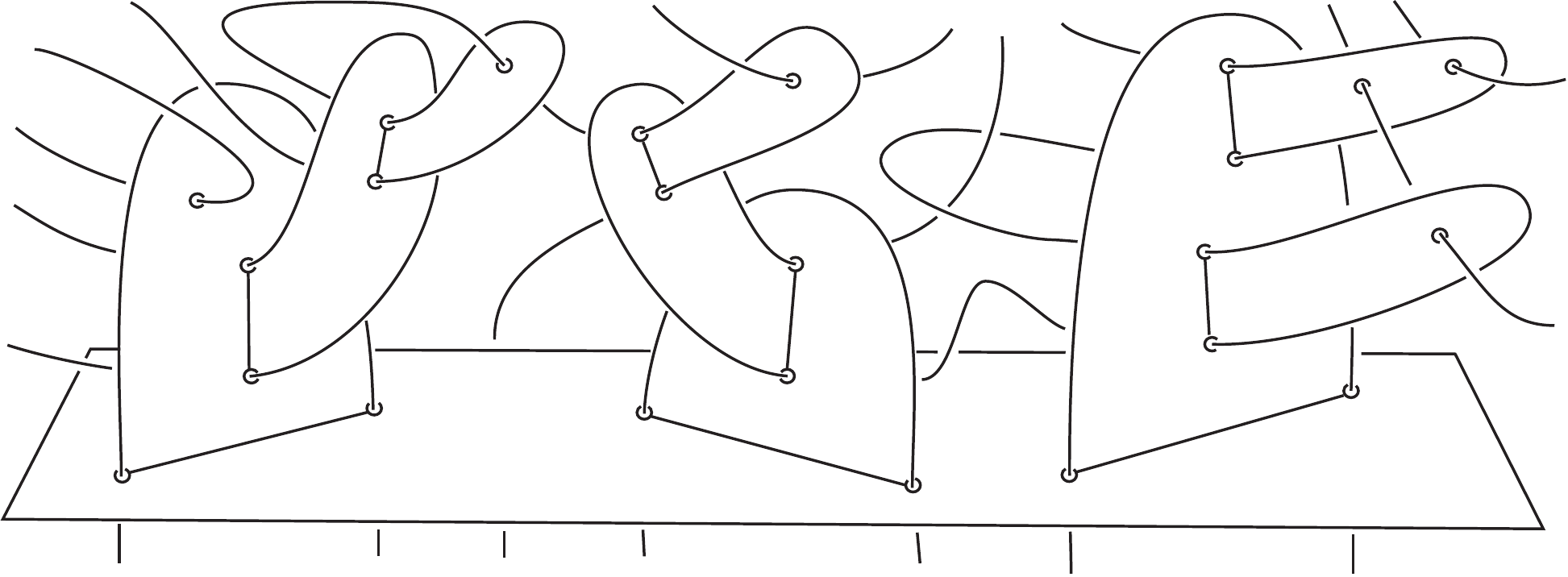}}
\end{figure}

\section*{Introduction}

These notes are an expansion of three introductory lectures given at the Winter Braids X conference February 2020 in Pisa, Italy
describing a theory of Whitney towers on immersed surfaces in $4$-manifolds, as developed in collaboration with Jim Conant and Peter Teichner. A Whitney tower is built on an immersed surface by iteratively adding Whitney disks pairing intersection points among the ``higher-order'' layers of Whitney disks and the surface (see section~\ref{subsec:w-tower-def}).
The theory's main goal is to use Whitney towers to gain topological information about the underlying immersed surface.
These notes are focussed on providing a detailed introduction to the theory while simultaneously describing some open problems.


Section~\ref{sec:disks-in-B4} covers local manipulations of surfaces and Whitney disks in $4$-space, and fundamental definitions related to Whitney towers, including the associated trivalent trees that organize Whitney towers.
This section culminates with the description of a geometric Jacobi Identity in the setting of Whitney towers.

Section~\ref{sec:twisted-order-n-classification-arf-conj} describes the classification of order $n$ twisted Whitney towers on properly immersed disks in the $4$-ball, which illustrates the order-raising intersection-obstruction theory and leads to the formulation of open problems related to certain higher-order Arf invariants which are invariants of classical link concordance.
Here the trees associated to Whitney towers are seen to represent invariants in abelian groups related to Milnor's classical link invariants.

Section~\ref{sec:2-spheres-in-4-manifolds} reviews the classical intersection and self-intersection (order 0) homotopy invariants of $2$-spheres in a $4$-manifold $X$ before introducing order~1 generalizations of these invariants in the setting of Whitney towers.
Here new subtleties coming from $\pi_1X$ and $\pi_2X$ enter the picture.
The end of this section describes open problems on the realization of the order~1 invariants when $X$ is closed and $\pi_1X$ is non-trivial. 

The appendix section~\ref{sec:appendix} provides additional material related to Section~\ref{sec:twisted-order-n-classification-arf-conj}  and Section~\ref{sec:2-spheres-in-4-manifolds} in the form of outlines and/or details of proofs of results from those sections.


Exercises appear at the end of each section.

Sections~\ref{sec:twisted-order-n-classification-arf-conj} and~\ref{sec:2-spheres-in-4-manifolds} are largely independent of each other, but both depend on Section~1.


%

\textbf{Conventions:} Manifolds and submanifolds are assumed to be smooth, with generic intersections, unless otherwise specified, and during cut-and-paste constructions corners will be assumed to be rounded.
The discussion throughout will also hold in the flat topological category via the notions of 4-dimensional topological tranversality from \cite[chap.9]{FQ}. 
Orientations will usually be assumed but suppressed unless needed.


%
%
%
%
%
%
%


{\bf Acknowledgments:} The author is supported by a Simons Foundation \emph{Collaboration Grant for Mathematicians}. Also thanks to the organizers of the \emph{Winter Braids X} conference, and of course collaborators Jim Conant and Peter Teichner.

\tableofcontents

\section{Disks in the $4$-ball}\label{sec:disks-in-B4}
This section focuses on local manipulations of surfaces in $4$-space, and fundamental definitions related to Whitney towers, including the associated trivalent trees that organize Whitney towers.
The section culminates with the description of a geometric Jacobi Identity in the setting of Whitney towers, followed by exercises related to the material covered here.

Throughout, the word ``disk'' means the $2$-disk $D^2$.
A \emph{sheet} of a surface in a $4$-manifold is an embedded disk which is the properly embedded intersection of the surface with a $4$-ball. Here a \emph{proper} map sends boundary to boundary, and sends interior to interior.
Following the tradition of Knot Theory, we will usually blur the distinction between a map and its image.
In the case that the boundary $\partial D$ of an immersed disk $D$ is contained in the interior of an immersed surface $A$, we require and assume that $\partial D$ is embedded, and also that the interior of $D$ is disjoint from $A$ near $\partial D$, i.e.~that there exists a collar in $D$ of $\partial D$ such that the intersection of this collar with $A$ is equal to $\partial D$. 
Orientations of surfaces and signs of intersection points will mostly be suppressed from notation.

\subsection{Local coordinates}

Figure~\ref{transverse-intersection-fig-1}
shows two properly embedded disks $A$ and $B$ in the $4$-ball $B^4=B^3\times I$ which have a single transverse intersection point $p=A\pitchfork B$. In this figure the $I$-parameter increases from left to right, and each $B^3$-slice $B^3\times t$ for $t\in I$ intersects each disk in an arc. Thus $A\cup B$ is the track of a homotopy of arcs in $B^3$ which fixes the vertical black arc tracing out $B$, and isotopes the horizontal blue arc tracing out $A$ from ``front to back'', realizing a crossing change at $p=A\pitchfork B$.
\begin{figure}[h]
\includegraphics[width=\textwidth]{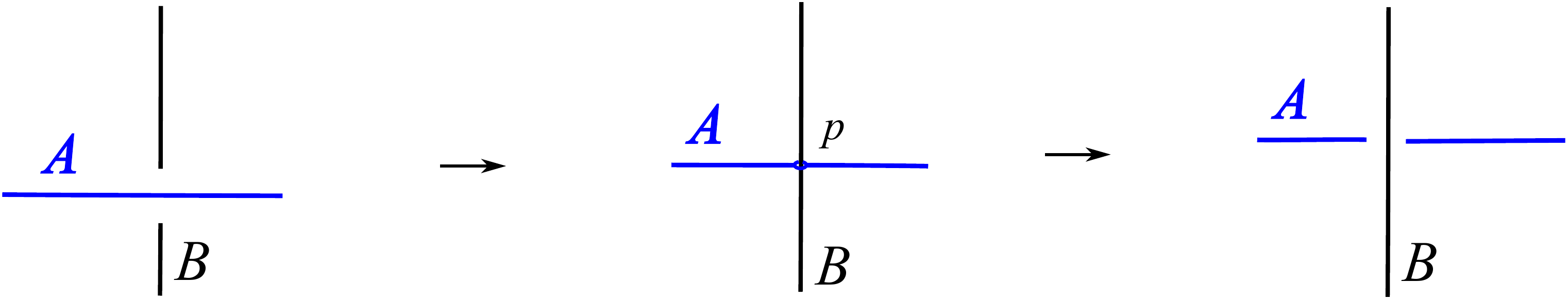}
\caption{Disks $A$ and $B$ in $B^4=B^3\times I$ with $p=A\pitchfork B$.}
\label{transverse-intersection-fig-1}
\end{figure}

In the interest of avoiding visual clutter, we do not attempt to show the boundary of $B^3$ in any slice.
So it is understood that the endpoints of the arcs lie in $\partial B^3\times t$ for all $t$, since $A$ and $B$ are each properly embedded.
Thus, to correctly interpret each picture in the movie one can either imagine the existence of the invisible
$2$-sphere $\partial B^3\times t$ containing the visible endpoints of the arcs, or one can assume that the arcs continue further outside the picture until eventually reaching $\partial B^3\times t$ which is lying outside the picture.

Thinking of the parameter $t$ as ``time'', we say that in Figure~\ref{transverse-intersection-fig-1} the transverse intersection $p=A\pitchfork B$ is contained in the \emph{present} slice $B^3\times 0\subset B^3\times I$, and that both disks extend as arcs into \emph{past} $B^3\times-\epsilon$ and \emph{future} $B^3\times+\epsilon$.
When describing surface sheets in this way we usually do not specify the endpoints of the interval $I$, which can be reparametrized as desired.

\begin{figure}[h]
\includegraphics[scale=.6]{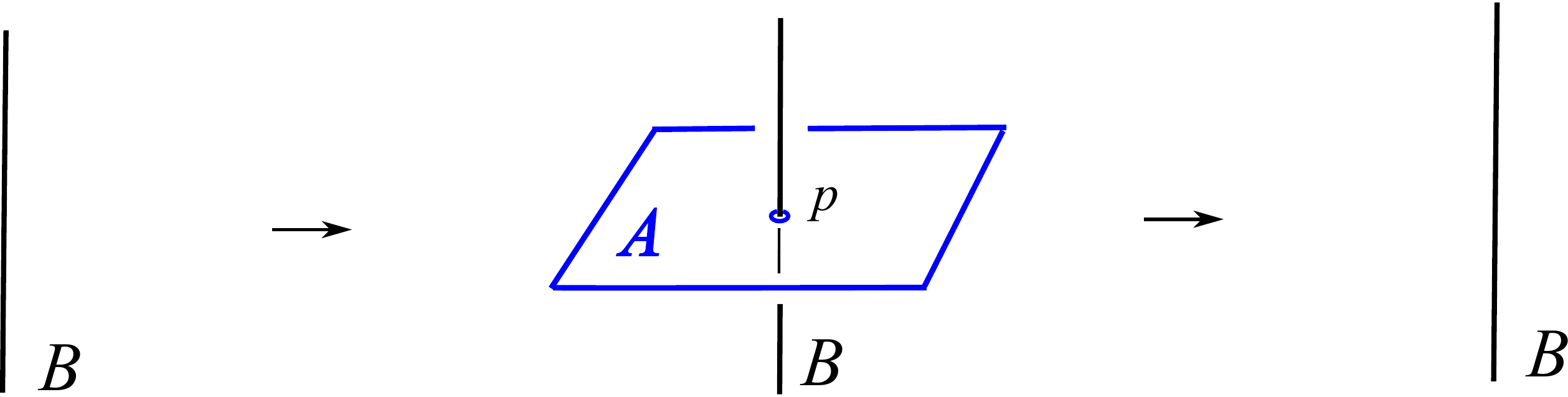}
\caption{$A$ and $B$ in $B^4=B^3\times I$ with $p=A\pitchfork B$ and $A\subset B^3\times 0$.}
\label{transverse-intersection-fig-2}
\end{figure}
 
Figure~\ref{transverse-intersection-fig-2} shows a different view of the same $4$-ball neighborhood of $p=A\pitchfork B$ by changing coordinates and/or isotopy. In this figure the disk $A$ is now contained in the present $B^3\times 0\subset B^3\times I$, while $B$ is still described by a fixed vertical arc times $I$.
This view has the advantage of being completely described by just the present slice. 

Again, $A$ and $B$ are each properly embedded in $B^4$. So the rectangular appearance of $A$ is either understood to extend further out of the picture until reaching reaching the boundary $2$-sphere of $B^3\times 0$, or we can just take each $B^3\times t$ to be a solid cube (with invisible PL  boundary sphere).

Note that the transverse intersection $p$ appears as a small ``puncture'' in the sheet $A$. This is just a visual device to help place the location of $p$ in $A$. Also, the $A$-sheet is shown here as translucent, with a sub-arc of $B$ visible ``behind'' $A$. In general, sheets in the present may appear translucent or opaque depending on which view provides a better clarification of the given configuration of sheets.


%

\subsection{Finger moves}
A \emph{finger move} is a regular homotopy supported near an arc that creates a pair of oppositely-signed intersections between two sheets. This is illustrated in Figure~\ref{finger-move-movies-fig} which shows a ``movie of movies'' starting from the disjoint sheets in the bottom and ending with the intersecting sheets in the top row. Each row shows the same $4$-ball, with the bottom row showing the disjoint sheets before the finger move. Vertically and horizontally adjacent pictures are understood to be connected by smooth interpolations of the black arcs, with the vertical parameter corresponding to the homotopy, and the horizontal parameter corresponding to the 
$I$-parameter of $B^4=B^3\times I$ as in the previous figures.
Between the middle row and the top row a single non-transverse tangential intersection will occur in the center picture at one moment of the homotopy.
\begin{figure}[h]
\includegraphics[width=\textwidth]{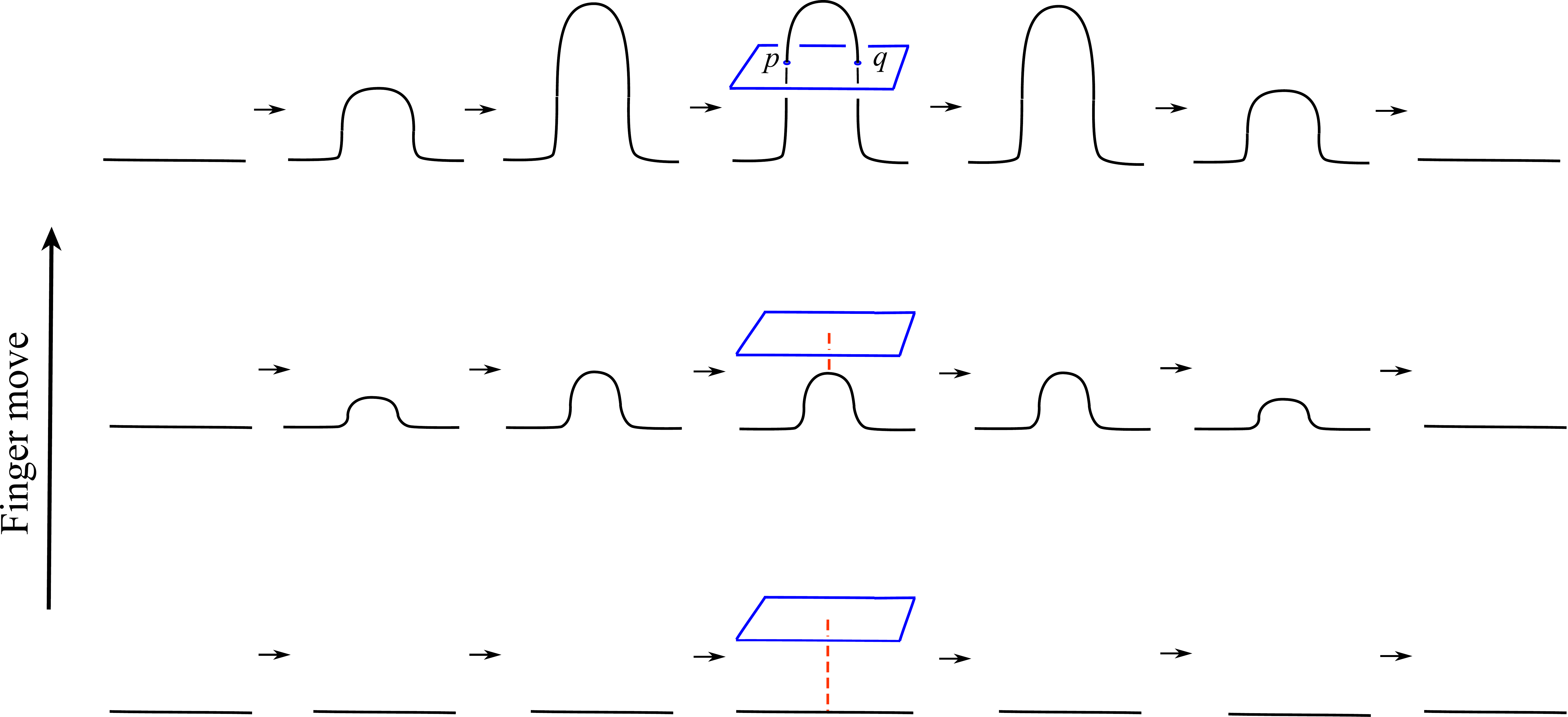}
\caption{A finger move guided by the red dashed arc.}
\label{finger-move-movies-fig}
\end{figure}

Usually only the center top and/or center bottom pictures as in Figure~\ref{finger-move-before-and-after-fig} will be shown, with the other nearby and interpolating pictures shown in Figure~\ref{finger-move-movies-fig} understood. Note that $p$ and $q$ have opposite signs for any choice of orientations on $A$, $B$ and the ambient $4$-ball (Exercise~\ref{ex:finger-move-ints-have-opposite-signs}).

A finger move is supported in a $4$-ball neighborhood of its guiding arc, so this picture also describes a finger move on two sheets in an arbitrary $4$-manifold. Since this neighborhood can be taken to be arbitrarily close to the guiding arc, it may be assumed that the support of a finger move is disjoint from any other surfaces, and hence only creates the pair of intersections. Up to isotopy, a finger move is symmetric in the two sheets (Exercise~\ref{ex:finger-move-symmetric}).

\begin{figure}[h]
\includegraphics[scale=.5]{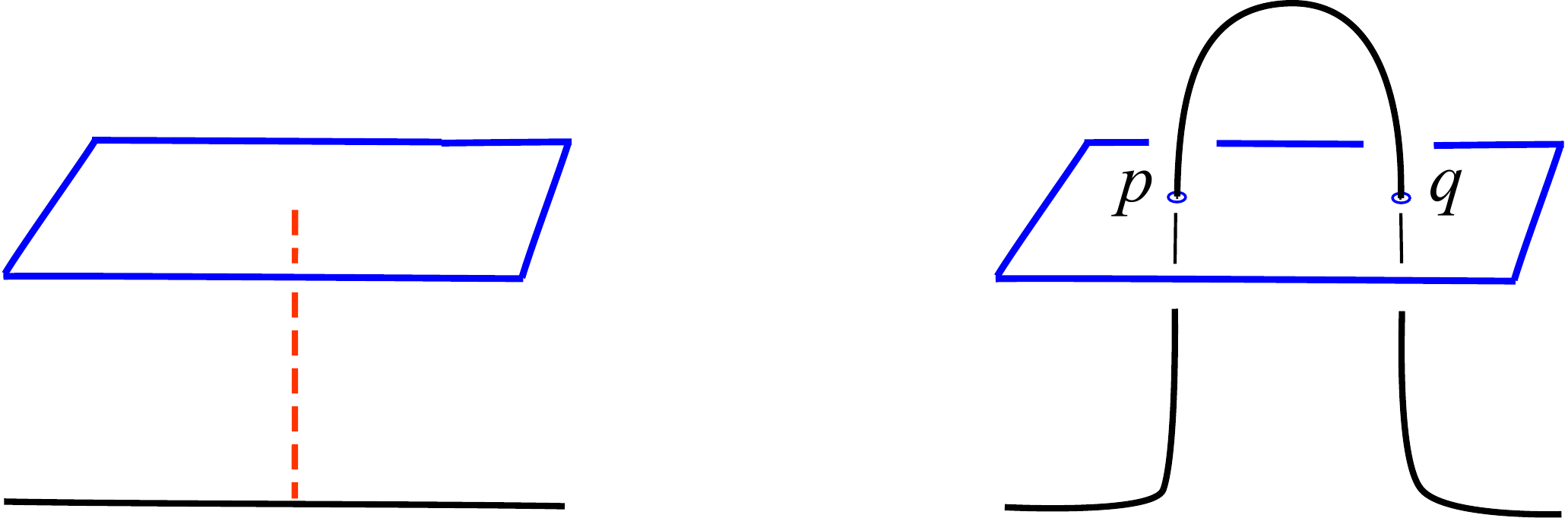}
\caption{Before (left) and after (right) a finger move: The center-bottom and center-top pictures from Figure~\ref{finger-move-movies-fig}.}
\label{finger-move-before-and-after-fig}
\end{figure}


\subsection{Whitney disks}\label{sec:whitney-disks}
Let $p$ and $q$ be oppositely-signed transverse intersections between connected properly immersed surfaces $A$ and $B$ in a $4$--manifold $X$, with $p$ and $q$ joined by embedded interior arcs $a\subset A$ and $b\subset B$ which are disjoint from all other singularities in $A$ and $B$. We allow the possibility that $A=B$, in which case the circle $a\cup b$ must be embedded and change sheets at $p$ and $q$. Any immersed disk $W$ in the interior of $X$ bounded by such a \emph{Whitney circle} $a\cup b$ is a \emph{Whitney disk} pairing $p$ and $q$. 

For any collection $W_i$ of Whitney disks pairing oppositely-signed $p_i,q_i\in A\pitchfork B$ we require that the union $\cup_i\partial W_i$ of Whitney circles is embedded. This can always be arranged by controlled isotopies of the $W_i$ (Exercise~\ref{ex:push-w-disk-boundary}).

\begin{figure}[ht!]
         \centerline{\includegraphics[width=\textwidth]{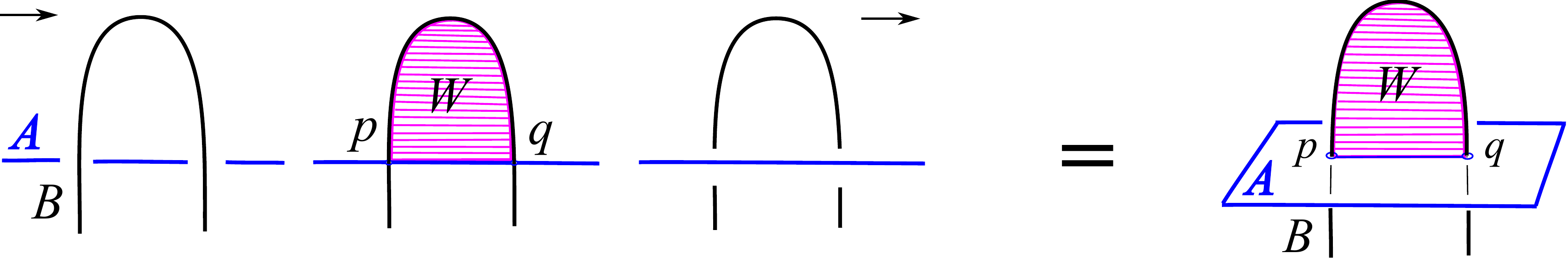}}
         \caption{Two views of a model Whitney disk $W$ pairing $p,q=A\pitchfork B$.}
         \label{Whitney-disk-pic-and-movie}
\end{figure}
Figure~\ref{Whitney-disk-pic-and-movie} shows a \emph{model} Whitney disk $W$ on sheets $A$ and $B$ in a $4$-ball.
In general, a Whitney disk in a $4$-manifold
has a neighborhood obtained by introducing \emph{plumbings} \cite[1.2]{FQ} into the model Whitney disk.
So a Whitney disk may have interior self-intersections and intersections with other surfaces, but has an embedded collar which is disjoint from all surfaces, except for its boundary arcs $a\subset A$ and $b\subset B$. 

An embedded Whitney disk whose interior is disjoint from all surfaces is called a \emph{clean} Whitney disk.




\begin{figure}[ht!]
         \centerline{\includegraphics[width=\textwidth]{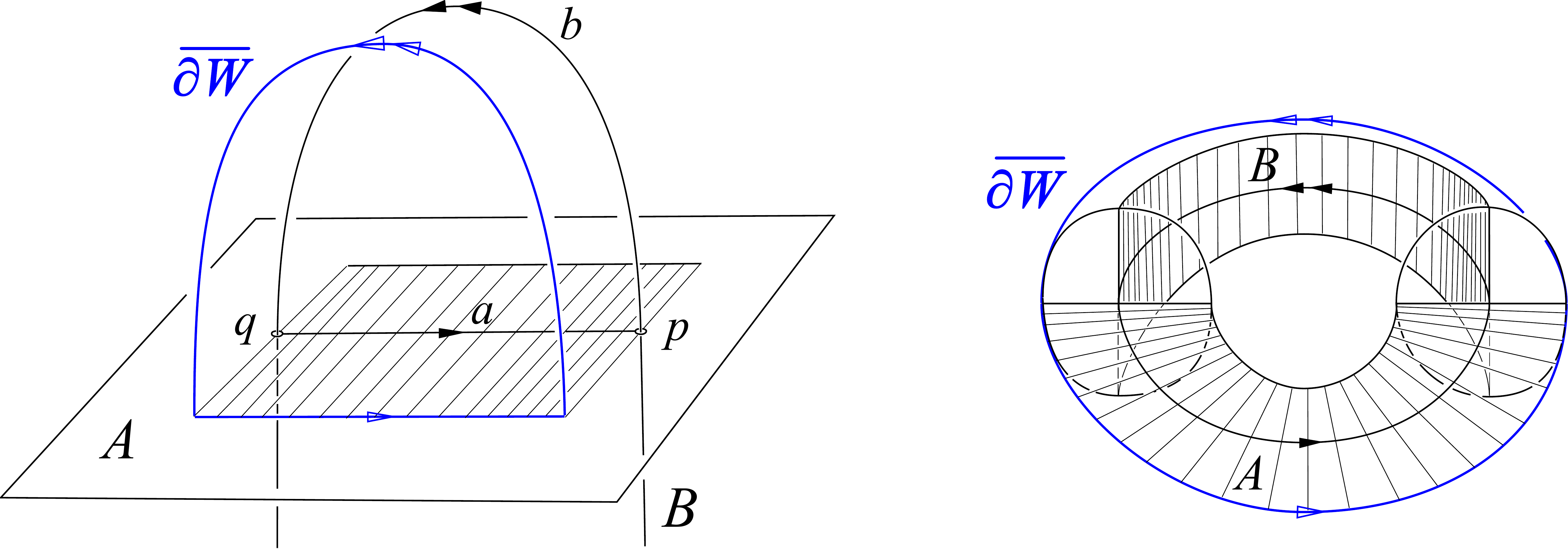}}
         \caption{Left: Near a Whitney disk $W$ pairing $p,q\in A\pitchfork B$, with $\partial W=a\cup b$, a Whitney section $\overline{\partial W}$ is shown in blue. This picture is accurate near $\partial W$; in general, the interior of the evident (but not explicitly indicated) embedded $W$ bounded by $a\cup b$ may have self-intersections as well as intersections with other surfaces. The line segments transverse to $a$ in $A$ indicate the correspondence with the right-hand picture of the normal disk-bundle $\nu_XW|_{\partial W}$.
         Right: The blue Whitney section $\overline{\partial W}$ is shown inside an embedding into $3$--space of $\nu_XW|_{\partial W}\cong S^1\times D^2$, with the sheets $A$ and $B$ indicated by line segments transverse to $\partial W$. The $A$-sheet cuts the front solid torus horizontally, while the $B$-sheet cuts the back of the solid torus vertically. The line segments transverse to $b$ in $B$ extend into past and future in the left-hand picture.}
         \label{W-subspaces}
\end{figure}

\subsection{Framed and twisted Whitney disks.}\label{subsec:framed-w-disks}
Let $W$ be a Whitney disk pairing intersections $p,q\in A\pitchfork B$, 
with boundary 
$\partial W= a\cup b$, for embedded arcs $a\subset A$ and $b\subset B$. 
Denote by $\nu_XW|_{\partial W}$ the restriction to $\partial W$ of the normal disk-bundle $\nu_XW$ of $W$ in $X$. Since $p$ and $q$ have opposite signs, $\nu_XW|_{\partial W}$ admits a nowhere-vanishing \emph{Whitney section} 
$\overline{\partial W}$ defined by taking vectors tangent to $A$ over $a$, and extending over $b$ by vectors which are normal to $B$, as shown in the left of Figure~\ref{W-subspaces}. 

 The right side of Figure~\ref{W-subspaces} shows $\overline{\partial W}$ inside an embedding into $3$--space of 
 $\nu_XW|_{\partial W}\cong S^1\times D^2$.
 Although this choice of embedding has $\overline{\partial W}$ corresponding to the $0$-framing of $D^2\times S^1\subset \R^3$, the section of 
 $\nu_XW|_{\partial W}$ determined (up to homotopy) by the canonical framing of $\nu_XW$ will in general differ by full twists relative to $\overline{\partial W}$. (If $p$ and $q$ had the same sign then there would have to be a half-twist in the sheets of $A$ and $B$, so the (continuous) Whitney section $\overline{\partial W}$ could not exist.)

If $\overline{\partial W}$ extends to a nowhere-vanishing section $\overline{W}$ of $\nu_XW$, then $W$ is said to be \emph{framed} (since a disk-bundle over a disk 
has a canonical framing, and a nowhere-vanishing normal section over an oriented surface in an oriented $4$--manifold determines a framing up to homotopy). In general, the obstruction to extending $\overline{\partial W}$ to a nowhere vanishing section of $\nu_XW$ is the relative Euler number $\chi(\nu_XW,\overline{\partial W})\in\Z$, called the \emph{twisting} of $W$ and denoted $\omega(W)$, so $W$ is framed if and only if $\omega(W)=0$. The twisting $\omega(W)$ can be computed by taking the intersection number of the zero section $W$ with any extension $\overline{W}$ of $\overline{\partial W}$ over $W$, so it does not depend on an orientation choice for $W$ (since switching the orientation on $W$ also switches the orientation of $\overline{W}$). And $\omega(W)$ is also unchanged by switching the roles of $a$ and $b$ in the construction of $\overline{\partial W}$, since interchanging the ``tangent to...'' and ``normal to...'' parts in the construction yields an isotopic section in $\nu_XW|_{\partial W}$ (isotopic through non-vanishing sections).

A Whitney disk $W$ with $\omega(W)\neq 0$ is called a \emph{twisted} Whitney disk.

\subsection{Parallel Whitney circles and disks}\label{subsec:general-whitney-section}
The twisting $\omega(W)$ of $W$, which is the element of $\pi_1(\mathrm{SO}(2))\cong\Z$ determined by a Whitney section as described above, 
can also be computed using any section $\overline{\partial W}$ of $\nu_XW|_{\partial W}$ such that $\overline{\partial W}$ is in the complement of the tangent spaces of both $A$ and $B$, since such $\overline{\partial W}$ will have the same number of rotations as a Whitney section (relative to the longitude determined by the canonical framing of $\nu_W$);
see Figure~\ref{W-subspaces}.

In these notes a \emph{parallel Whitney circle} will refer to either these ``nowhere tangent'' sections, or the standard sections in section~\ref{subsec:framed-w-disks}.

A disk is \emph{parallel} to a Whitney disk if it is the extension of a parallel Whitney circle over the Whitney disk.

\subsection{Whitney moves}
Figure~\ref{fig:model-W-move} describes a \emph{model Whitney move}: In a $4$--ball neighborhood of a framed embedded Whitney disk $W$, a pair of oppositely-signed transverse intersections $p$ and $q$ between surface sheets $A$ and $B$ is eliminated by a regular homotopy which isotopes one of the sheets across the clean framed $W$. 

Combinatorially, the result of this Whitney move is constructed by deleting a regular neighborhood (in its sheet) of one arc of $\partial W$ and replacing this neighborhood with a \emph{Whitney bubble} over that arc. This Whitney bubble is formed from two parallel copies of $W$ connected by a rectangular strip which is normal to a neighborhood (in its sheet) of the other arc. Figure~\ref{fig:model-W-move} shows the result of
adding a Whitney bubble to the sheet $A$. Although both these descriptions of the Whitney move involve a choice of arc of $\partial W$, up to isotopy the result is independent of this choice (Exercise~\ref{ex:W-move-symmetric}). 

A Whitney move on a clean framed $W$ in a $4$-manifold $X$ is described by any embedding of this model into $X$ which preserves the product structures and transversality of the sheets and $W$.

\begin{figure}[h]
\includegraphics[scale=.6]{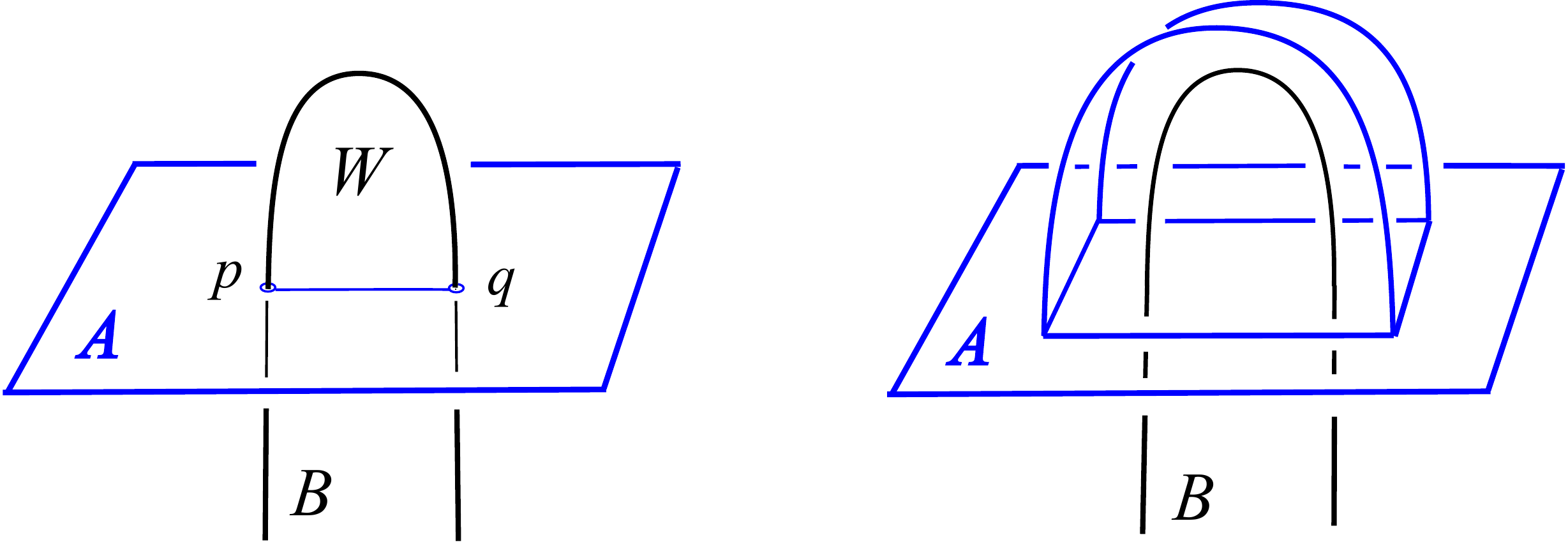}
\caption{Left: $W$ pairing $p,q=A\cap B$. Right: A Whitney move guided by $W$.}
\label{fig:model-W-move}
\end{figure}

In general, a \emph{Whitney move} on an arbitrary Whitney disk $W$ is described using the combinatorial description of adding a Whitney bubble to one sheet: The Whitney bubble is still formed from two parallel copies of $W$ together with a rectangular strip, and the intersections paired by $W$ will still be eliminated by the Whitney move, but the following new intersections will be created. Each intersection $r\in W\pitchfork C$ between the interior of $W$ and a sheet $C$ of another surface will give rise to two new intersections, as shown in 
Figure~\ref{fig:W-disk-int-and-W-move-color}; each self-intersection of $W$ will give rise to four new self-intersections; and if $W$ is twisted, then $|\omega(W)|$-many new self-intersections will be created corresponding to intersections between the two parallel copies of $W$ in the Whitney bubble (see Exercise~\ref{ex:count-w-move-ints}).
\begin{figure}[h]
\includegraphics[scale=.6]{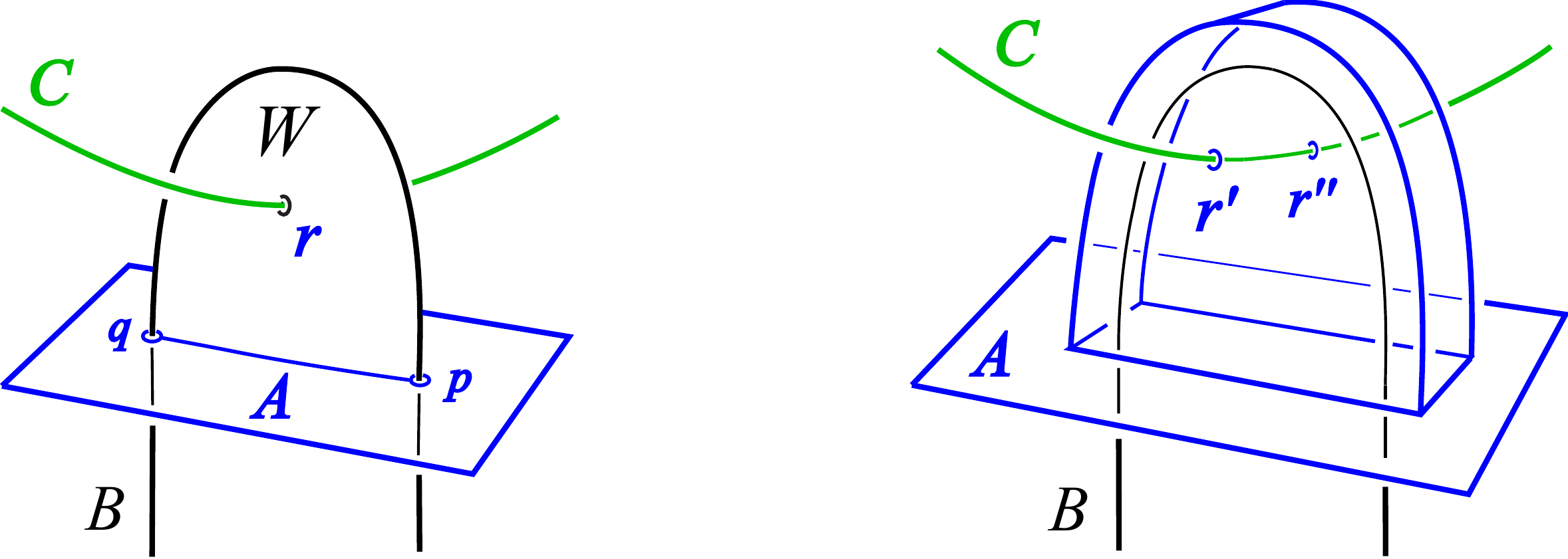}
\caption{
}
\label{fig:W-disk-int-and-W-move-color}
\end{figure}

In our settings we will generally be able to arrange that Whitney disks are embedded and framed at the cost of creating interior intersections with sheets of other surfaces, so the Whitney move shown in Figure~\ref{fig:W-disk-int-and-W-move-color} will be the most relevant.

%
%
%
%
%

%
%
%
%
%




\subsection{Whitney towers}\label{subsec:w-tower-def} 
%

A ``successful'' Whitney move eliminates the pair of intersections paired by the Whitney disk with out creating any new intersections.
Since this requires a clean framed Whitney disk, it is natural to try to somehow count interior intersections and twistings in an attempt to measure the obstructions to the existence of a homotopy of an immersed surface to an embedding.

Whenever these ``higher-order intersections'' can be paired by ``higher-order Whitney disks'' there is hope that they might be eliminated by  successful ``higher-order Whitney moves'', so one is led to construct a tower of Whitney disks by iteratively pairing up as many intersections as possible. As new Whitney disks are chosen, new intersections and twistings can appear which in turn need to be counted. We will see that invariants of the resulting Whitney tower can yield information about the bottom level immersed surface on which the tower is built. 
\begin{figure}[h]
\includegraphics[scale=.6]{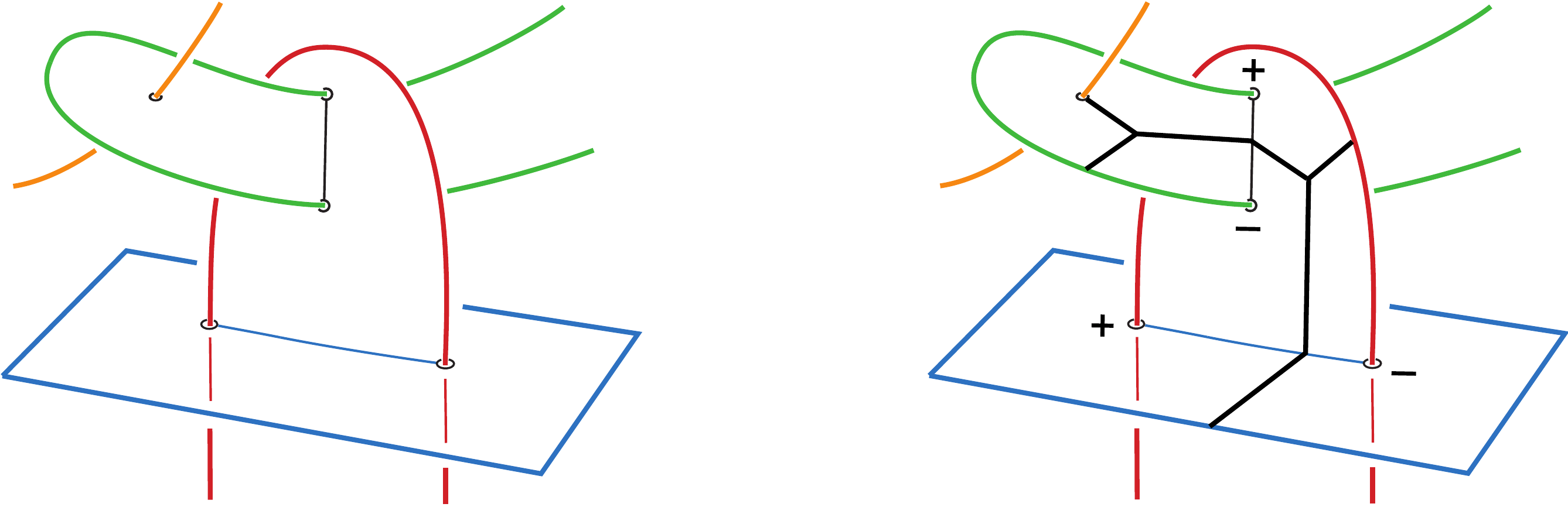}
\caption{Left: Four sheets and two Whitney disks in a $4$-ball. Right: With the tree associated to the unpaired intersection.}
\label{higher-order-intersection-color-and-with-tree}
\end{figure}

As a basic example, the left side of Figure~\ref{higher-order-intersection-color-and-with-tree} shows four sheets in $B^4$ supporting a Whitney tower: The two intersections between red and blue are paired by a Whitney disk which itself has two interior intersections with green that are paired by another Whitney disk. This second Whitney disk has a single interior intersection with yellow that appears to be a ``robust'' obstruction to successfully homotoping these sheets rel boundary to be disjointly embedded in this $4$-ball. In fact, we will see in Section~\ref{sec:twisted-order-n-classification-arf-conj} that this unpaired higher-order intersection does indeed represent such an obstruction (Exercise~\ref{ex:tau-of-bing-hopf-not-zero}).  

%

We can now state the general definition of a Whitney tower:
  
\begin{defn}\label{def:w-towers}
A \emph{Whitney tower} on $A^2\imra X^4$ is defined by:

\begin{enumerate}
\item
$A$ itself is a Whitney tower.

\item
If $\cW$ is a Whitney tower and $W$ is a Whitney disk pairing intersections in $\cW$, then
the union $\cW\cup W$ is a Whitney tower.
\end{enumerate}
\end{defn}
In a Whitney tower the Whitney disk boundaries are required to be pairwise disjointly embedded (section~\ref{sec:whitney-disks}).

Having constructed a Whitney tower by adding some number of Whitney disks as per the definition, we call the properly immersed surface $A$ the \emph{underlying surface} of the Whitney tower, and say that $A$ \emph{supports} the Whitney tower. The ``boundary of $\cW$'' is the boundary of $A$. 
We assume that $A$ comes with a given orientation, and that $\cW$ is \emph{oriented} by choosing (arbitrarily) and fixing orientations of all the Whitney disks in $\cW$. Specific orientations and signs will generally be suppressed from notation and discussion until needed, and at a first reading can be safely ignored while absorbing the combinatorics of Whitney towers.

Our underlying surfaces $A$ will always be collections of disks or annuli or spheres, with components $A_i$ indexed by $i\in\{1,2,\ldots,m\}$. 

Definition~\ref{def:w-towers} describes a very general notion of Whitney tower, and all of the figures so far give very basic examples that satisfy this definition. Structure will emerge upon the introduction of the unitrivalent trees that organize Whitney towers (Figure~\ref{higher-order-intersection-color-and-with-tree} right, and Figure~\ref{split-w-tower-with-trees-fig} bottom-right), 
and we will see that Whitney towers are interesting in their own right, as well as providing information about their underlying surfaces.


\begin{figure}[h]
\centerline{\includegraphics[width=\textwidth]{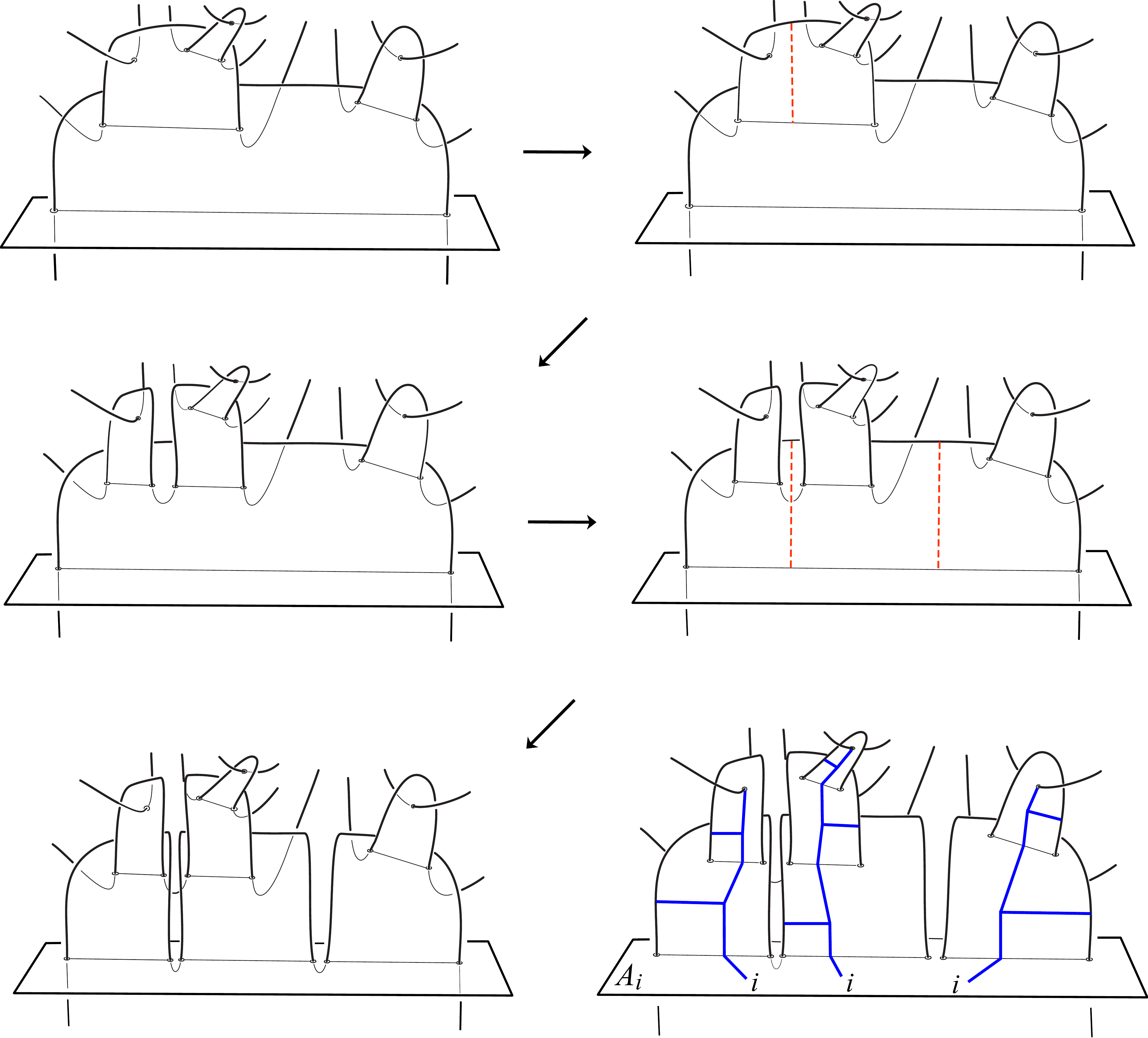}}
\caption{Zig-zagging down from top-left to bottom-left: Splitting a Whitney tower by finger moves. Bottom-right: The trees (blue) associated to unpaired intersections.}
\label{split-w-tower-with-trees-fig}
\end{figure}

As motivation for introducing formalism to organize Whitney towers, Figure~\ref {split-w-tower-with-trees-fig} shows how after ``splitting'' a Whitney tower by finger moves it can be arranged that all of its singularities are contained in regular neighborhoods of unitrivalent trees, with each Whitney disk containing only one ``problem'' (un-paired intersection or Whitney disk boundary-arc). In section~\ref{subsec:split-w-towers}
this notion of splitting is described precisely and extended to include splitting of twisted Whitney disks.



\subsection{Trees}\label{subsec:trees}
A \emph{tree} will always refer to a finite oriented unitrivalent tree, where the orientation of a tree is given by cyclic orderings of the adjacent edges around each trivalent vertex. 
Univalent vertices will usually be labeled from the index set $\{1, 2, 3, \ldots, m\}$, which will always correspond to the connected components of a properly immersed surface $A=A_1\cup A_2\cup A_3\cup\cdots\cup A_m\imra (X,\partial X)$ supporting a Whitney tower. Trees are considered up to isomorphisms which preserve labels and orientations.

We identify formal non-associative iterated bracketings 
of elements from
the index set $\{1,2,3,\ldots,m\}$ with
\emph{rooted} trees, which have each univalent vertex labeled by an element from the index set, except
for one \emph{root} univalent vertex which is left unlabeled.  
So for instance the bracket $(i,j)$ denotes the rooted tree $^{\,\,i}_{\,\,j}> \!\!\!\!-\!\!\!\!\!-$, for $i,j\in\{1,2,3,\ldots,m\}$; and inductively
$(I,J)$ denotes $^{\,\,I}_{\,\,J}> \!\!\!\!-\!\!\!\!\!-$, where $I$ and $J$ denote bracketings/rooted trees. This is formalized in the first part of the following definition:

\begin{defn}\label{def:Trees}
Let $I$ and $J$ be two rooted trees.
\begin{enumerate} 
\item The \emph{rooted product} $(I,J)$ is the rooted tree gotten
by identifying the root vertices of $I$ and $J$ to a single vertex $v$ and sprouting a new rooted edge at $v$.
This operation corresponds to the formal bracket (Figure~\ref{inner-product-trees-fig} upper right). The orientation of $(I,J)$ is inherited from those of $I$ and $J$, and at $v$ by the (left-to-right) ordering of the bracket product.

\item The \emph{inner product}  $\langle I,J \rangle $ is the
(unrooted) tree gotten by identifying the roots of $I$ and $J$ to a single non-vertex point.
Note that $\langle I,J \rangle $ inherits an orientation from $I$ and $J$, and that
all the univalent vertices of $\langle I,J \rangle $ are labeled.
(Figure~\ref{inner-product-trees-fig} lower right.)

\end{enumerate}
\end{defn}

\begin{figure}[ht!]
        \centerline{\includegraphics[scale=.40]{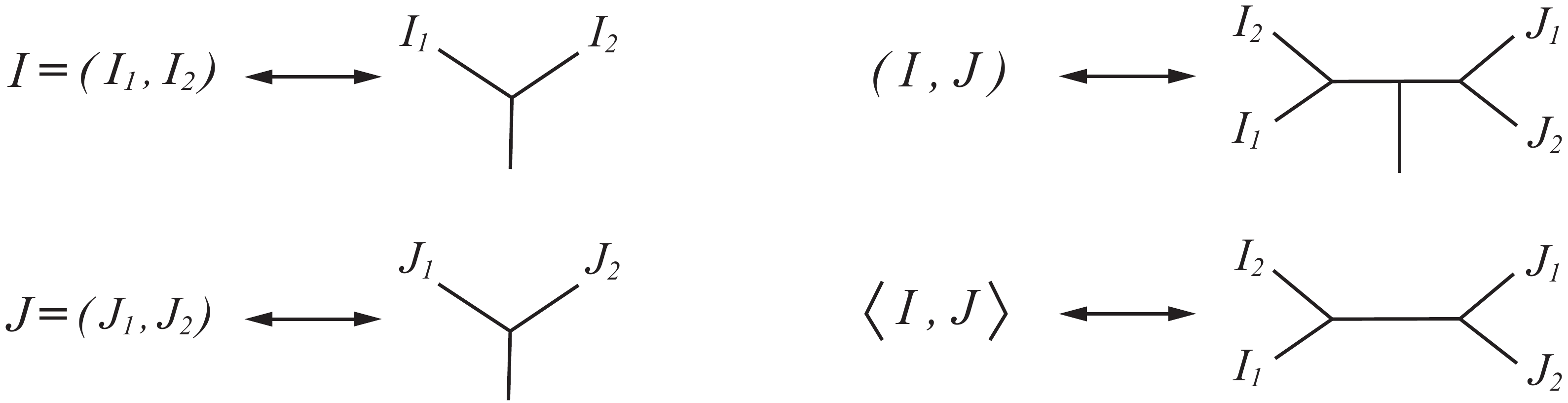}}
        \caption{The \emph{rooted product} $(I,J)$ and \emph{inner product} $\langle I,J \rangle$ of $I=(I_1,I_2)$ and $J=(J_1,J_2)$. All trivalent orientations correspond to a clockwise orientation of the plane.}
        \label{inner-product-trees-fig}
\end{figure}


\subsection{Trees for Whitey disks and intersection points}\label{subsec:trees-for-w-disks-and-ints}
Let $A=A_1,\ldots,A_m\looparrowright (X,\partial X)$ be a properly immersed surface supporting a Whitney tower $\cW$ (oriented), where the $A_i$ are the connected components of $A$. 

To each component $A_i$ is associated
the rooted tree $-\!\!\!\!-\!\!\!\!\!-\,i$ consisting of a single edge with one vertex labeled by $i$, and
to each transverse intersection $p\in A_i\pitchfork A_j$ is associated the 
tree $t_p:=\langle \,i\,,\,j\, \rangle=i\,-\!\!\!\!-\!\!\!\!\!-\,j$ consisting of an edge with vertices labeled by $i$ and $j$. Note that
for singleton brackets (rooted edges) we drop the bracket from notation, writing $i$ for $(\,i\,)$.

The rooted Y-tree $(i,j)=^{\,\,i}_{\,\,j}> \!\!\!\!-\!\!\!\!\!-$, with a single trivalent vertex and two univalent labels $i$ and $j$,
is associated to any Whitney disk $W_{(i,j)}$ pairing intersections between $A_i$ and $A_j$. This rooted tree
can be thought of as being embedded in $\cW$, with its trivalent vertex and rooted
edge sitting in $W_{(i,j)}$, and its two other edges descending into $A_i$ and $A_j$ as sheet-changing paths (Figure~\ref{Wij-Ak-fig} left). The cyclic orientation at the trivalent vertex corresponds to the orientation of $W_{(i,j)}$ via orientation conventions that will be described in section~\ref{subsec:w-tower-tree-orientations}.

Note that Figure~\ref{Wij-Ak-fig} shows the $j$-labeled univalent vertex of the tree $(i,j)$ in the present, but this $j$-labeled edge changes sheets into $A_j$ so the $j$-labeled univalent vertex is really in the past or future.

\begin{figure}[h]
\centerline{\includegraphics[scale=.5]{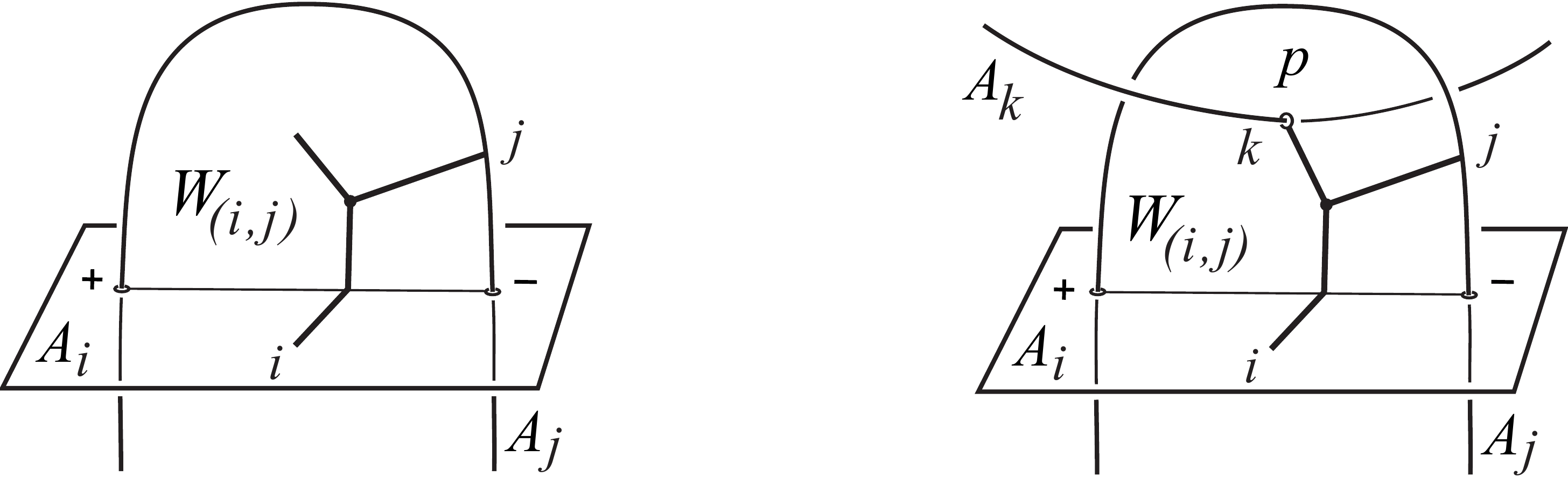}}
\caption{Left: The rooted tree $(i,j)\subset\cW$ associated to $W_{(i,j)}$ (oriented counter-clockwise from the reader's point of view). Right: The unrooted tree $t_p=\langle ((i,j),k)\rangle\subset\cW$ associated to $p\in W_{(i,j)}\pitchfork A_k$.}
\label{Wij-Ak-fig}
\end{figure}
Associated to any transverse intersection $p\in W_{(i,j)}\pitchfork A_k$ is the unrooted tree 
$t_p=\langle (i,j),k\rangle=^{\,\,i}_{\,\,j}> \!\!\!\!-\!\!\!\!\!-{\scriptstyle k}$, as illustrated in Figure~\ref{Wij-Ak-fig} right. 
We take the root vertex of each of the rooted trees associated to $W_{(i,j)}$ and $A_k$ to be at $p$, so the inner product $t_p=\langle (i,j),k\rangle$ is realized geometrically as the union $^{\,\,i}_{\,\,j}> \!\!\!-\!\!\!\!\!-\cup_p-\!\!\!\!\!-{\scriptstyle k}$ of these rooted trees in $\cW$.

Recursively, the rooted tree $(I,J)$ is associated to any Whitney disk $W_{(I,J)}$ pairing intersections
between $W_I$ and $W_J$ (see left-hand side of Figure~\ref{WdiskIJandIJKint-fig}); with the understanding that if, say, $I$ is just a singleton $i$, then $W_I=W_i$ denotes the surface component $A_i$. 
\begin{figure}[h]
\centerline{\includegraphics[scale=.5]{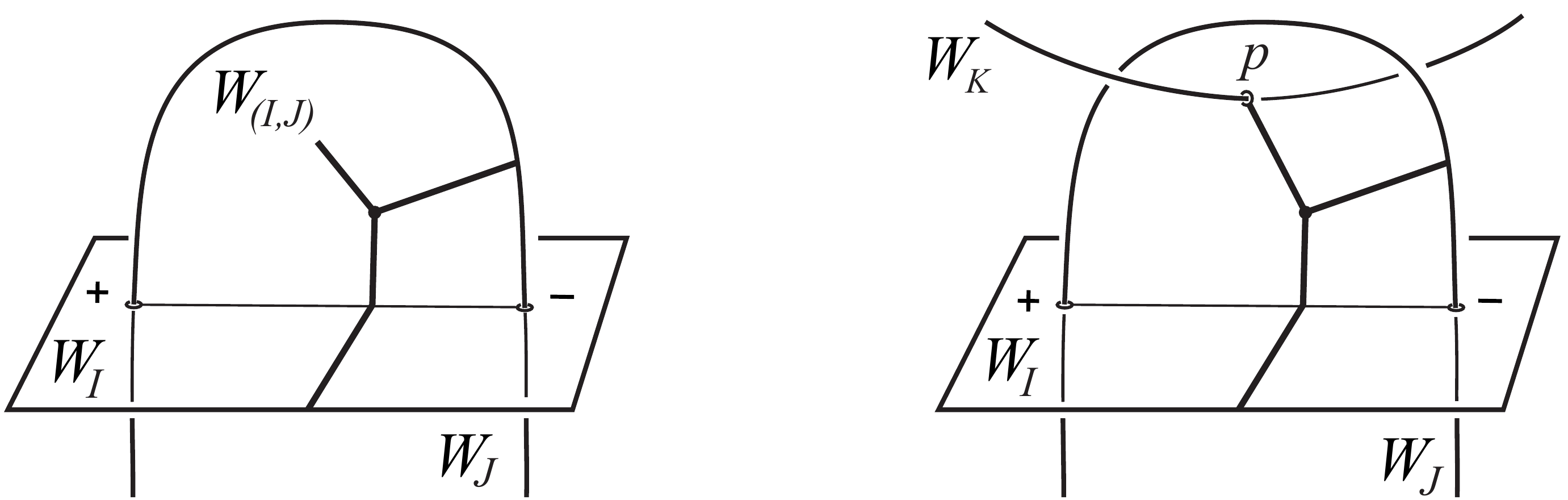}}
\caption{}
\label{WdiskIJandIJKint-fig}
\end{figure}
And to any transverse intersection $p\in W_{(I,J)}\cap W_K$ between $W_{(I,J)}$ and any
$W_K$ is associated the un-rooted tree $t_p:=\langle (I,J),K \rangle$  (see right-hand side of Figure~\ref{WdiskIJandIJKint-fig}).

The above description of the trees $t_p\subset \cW$ can always be arranged to be disjointly embedded in $\cW$, for any number of unpaired intersections in $\cW$ (Exercise~\ref{ex:embedded-trees}).

\subsection{Twisted trees for twisted Whitney disks}\label{subsec:twisted-trees}
So far we have associated rooted trees to Whitney disks, and unrooted trees to unpaired intersection points.
To keep track of the twistings associated to Whitney disks, we introduce one more type of tree:

For any rooted tree $J$, define the \emph{$\iinfty$-tree} 
$$
J^\iinfty:=J-\!\!\!-\iinfty
$$ 
by labeling the root of $J$ with the ``twist'' symbol $\iinfty$.
These $\iinfty$-trees are called \emph{twisted trees} since they are associated to twisted Whitney disks:

Recall from section~\ref{subsec:framed-w-disks} that a Whitney disk $W_J$ with non-zero twisting $\omega(W_J)\neq 0\in\Z$ is said to be \emph{twisted}. To each twisted Whitney disk $W_J\subset\cW$ we associate the twisted tree $J^\iinfty$.

\subsection{`Framed tree' terminology}\label{subsec:framed-trees}
Since twisted trees are associated to twisted Whitney disks, for clarity we will sometimes refer to the unrooted trees $t_p$ associated to intersection points as ``framed'' trees. In fact, section~\ref{subsec:split-w-towers} shows that it can always be arranged that every unpaired intersection $p\in\cW$ is an intersection between \emph{framed} Whitney disks or surface components. (Also, ``framed'' is more succinct than ``non-$\iinfty$ un-rooted'' or ``not-twisted un-rooted''.)


\subsection{Intersection forests of Whitney towers}\label{subsec:int-forest-def}
The above-described associations of trees capture the essential structure of Whitney towers, and the following key definition will be useful for both controlling constructions and defining invariants: 
\begin{defn}\label{def:intersection forests}
The \emph{intersection forest} $t(\cW)$ of a Whitney tower $\cW$ is the multiset
of
signed trees associated to unpaired intersections $p\in\cW$
and
$\Z$-coefficient $\iinfty$-trees associated to twisted Whitney disks $W_J\subset\cW$:
\[
t(\cW):=\sum \ \epsilon_p \cdot  t_p \,\, + \sum \ \omega(W_J)\cdot  J^\iinfty
\]
with $\epsilon_p\in\{+,-\}$ the usual sign of the transverse intersection point $p$, and $\omega(W_J)\in\Z$ the twisting of $W_J$. 
\end{defn}

Here the formal sums are over all unpaired $p$ and all twisted $W_J$ in $\cW$.
In fact, by splitting twisted Whitney disks (section~\ref{subsec:split-w-towers}) it can be arranged that all $\omega(W_J)=\pm 1$, just like the signs of the $p$.

In some papers $t(\cW)$ is notated as a disjoint union rather than a formal sum. Regardless of notation choice, it is helpful to think of $t(\cW)$ as being embedded in $\cW$ as per Figures~\ref{Wij-Ak-fig} and~\ref{WdiskIJandIJKint-fig} (and Exercise~\ref{ex:embedded-trees}), since
modifications of Whitney towers are described in terms of changes to $t(\cW)$, and these modifications correspond to relations in the targets of invariants represented by $t(\cW)$ that live in abelian groups generated by trees. 

\begin{rem}[Key Question]\label{rem:key-question}
Given a properly immersed surface $A$ in a $4$-manifold, constructing a Whitney tower $\cW$ on $A$ involves many possible choices of Whitney disks, and since the essential structure of $\cW$ is contained in $t(\cW)$, it is natural to ask: ``What does $t(\cW)$ tell us about $A$?''.  Various answers to this question have been found, as will be illustrated by results throughout these notes. More answers are waiting to be discovered...
\end{rem}

\begin{figure}[h]
\centerline{\includegraphics[scale=.45]{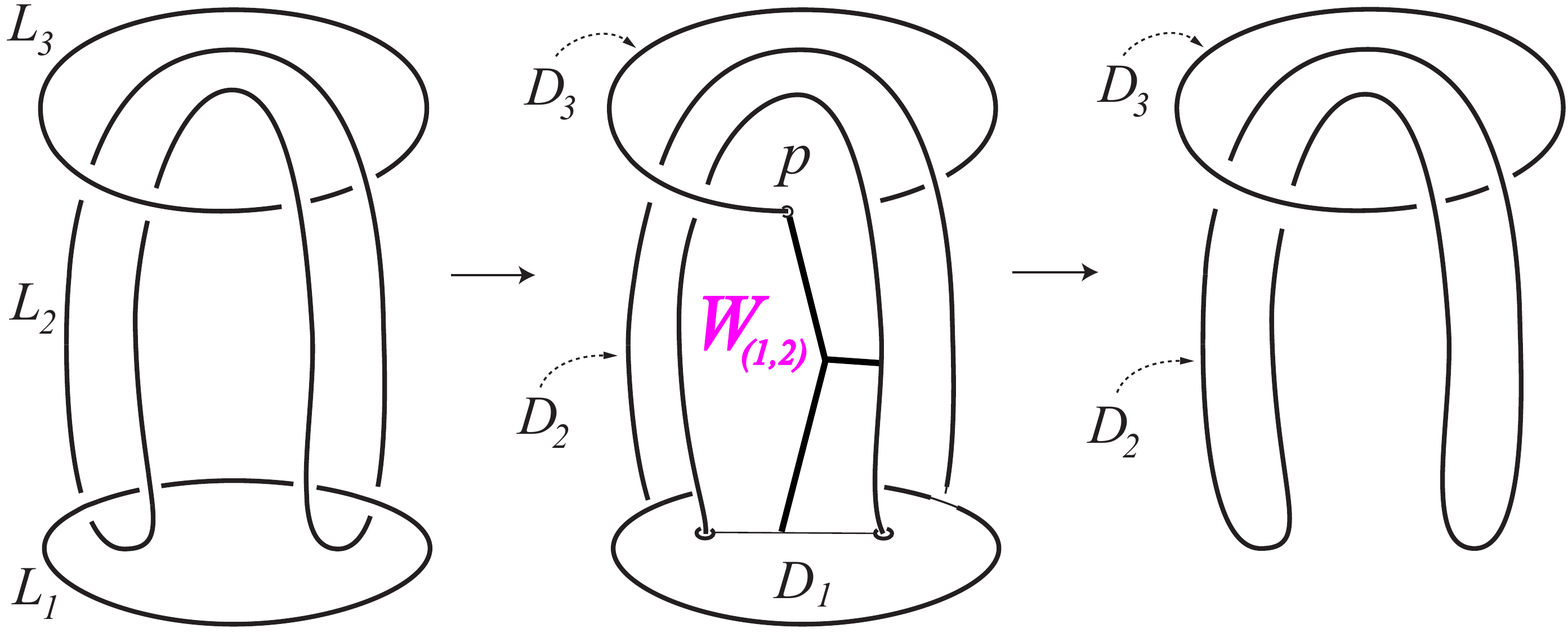}}
\caption{}
\label{fig:Borro-rings}
\end{figure}
\subsection{Examples}\label{subsec:examples-of-intersection-forests}

It will be shown in Lemma~\ref{lem:realization-of-geometric-trees} that for \emph{any} multiset of signed trees there exists a link $L\subset S^3$ bounding immersed disks into $B^4$ supporting a Whitney tower $\cW$ such that $t(\cW)$ is equal to the given multiset of signed trees. 

On the other hand, Exercise~\ref{ex:t(W)-on-S2-in-B4-restricted} will show that
there exist restrictions on the possible $t(\cW)$ for a Whitney tower $\cW$ on immersed $2$-spheres in $B^4$.

Some basic examples follow:


\subsubsection{}\label{eg:boro-rings}
Moving radially into $B^4$ from left to right, Figure~\ref{fig:Borro-rings} shows the Borromean Rings $L=L_1\cup L_2\cup L_3\subset S^3=\partial B^4$ bounding properly immersed disks $A=D_1\cup D_2\cup D_3\imra B^4$ which support $\cW=D_1\cup D_2\cup D_3\cup W_{(1,2)}$ with $t(\cW)=^{\,\,1}_{\,\,2}> \!\!\!-\!\!\!\!\!-{\scriptstyle 3}$.

The disk $D_1$ consists of the ``horizontal'' opaque disk in the lower part of the middle picture extended by an annular collar back to $L_1$ in the left picture. The disks $D_2$ and $D_3$ consist of the embedded annuli which are the products of $L_2$ and $L_3$ with the radial coordinate into $B^4$ together with embedded disks (not shown) extending further into $B^4$ that are bounded by the unlink in the right picture. The framed embedded Whitney disk $W_{(1,2)}$ which pairs $D_1\cap D_2$ is completely contained in the middle picture and has a single interior intersection point $p$ with $D_3$ such that $t_p=\langle(1,2),3\rangle=^{\,\,1}_{\,\,2}> \!\!\!-\!\!\!\!\!-{\scriptstyle 3}=t(\cW)$. 


\subsubsection{}\label{eg:Fig8-int-forest}
Moving radially into $B^4$ from left to right, Figure~\ref{Fig8-knot-twisted-W-fig} shows the Figure-8 knot in $S^3$ bounding a properly immersed disk $A=D_1\imra B^4$ supporting $\cW=D_1\cup W_{(1,1)}$ with $t(\cW)=(1,1)^\iinfty=^{\,\,1}_{\,\,1}> \!\!\!-\!\!\!\!\!-{\iinfty}$.

The track of the indicated null-homotopy of the Figure-8 knot describes the properly immersed disk $D_1$ with two self-intersections that are paired by the clean $+1$-twisted Whitney disk $W_{(1,1)}$. Part of $W_{(1,1)}$ is visible in the middle picture, and the unlink in the right hand picture can be capped off by two embedded disks which form the rest of $D_1$ and $W_{(1,1)}$. The twisting $\omega (W_{(1,1)})=1$ of $W_{(1,1)}$ is explained by Figure~\ref{Fig8knot-Wdisk-Whitney-section-fig}.
\begin{figure}[h]
\centerline{\includegraphics[scale=.45]{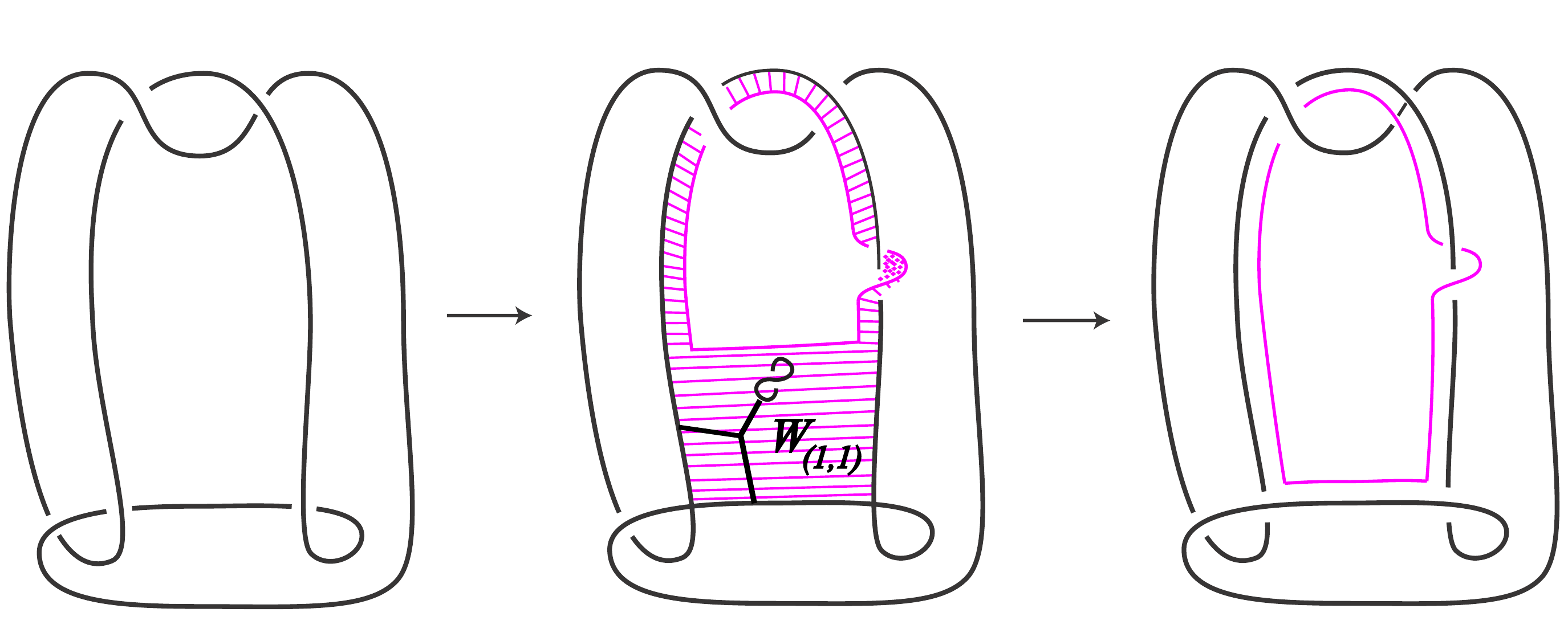}}
\caption{}
\label{Fig8-knot-twisted-W-fig}
\end{figure}
\begin{figure}[h]
\centerline{\includegraphics[scale=.45]{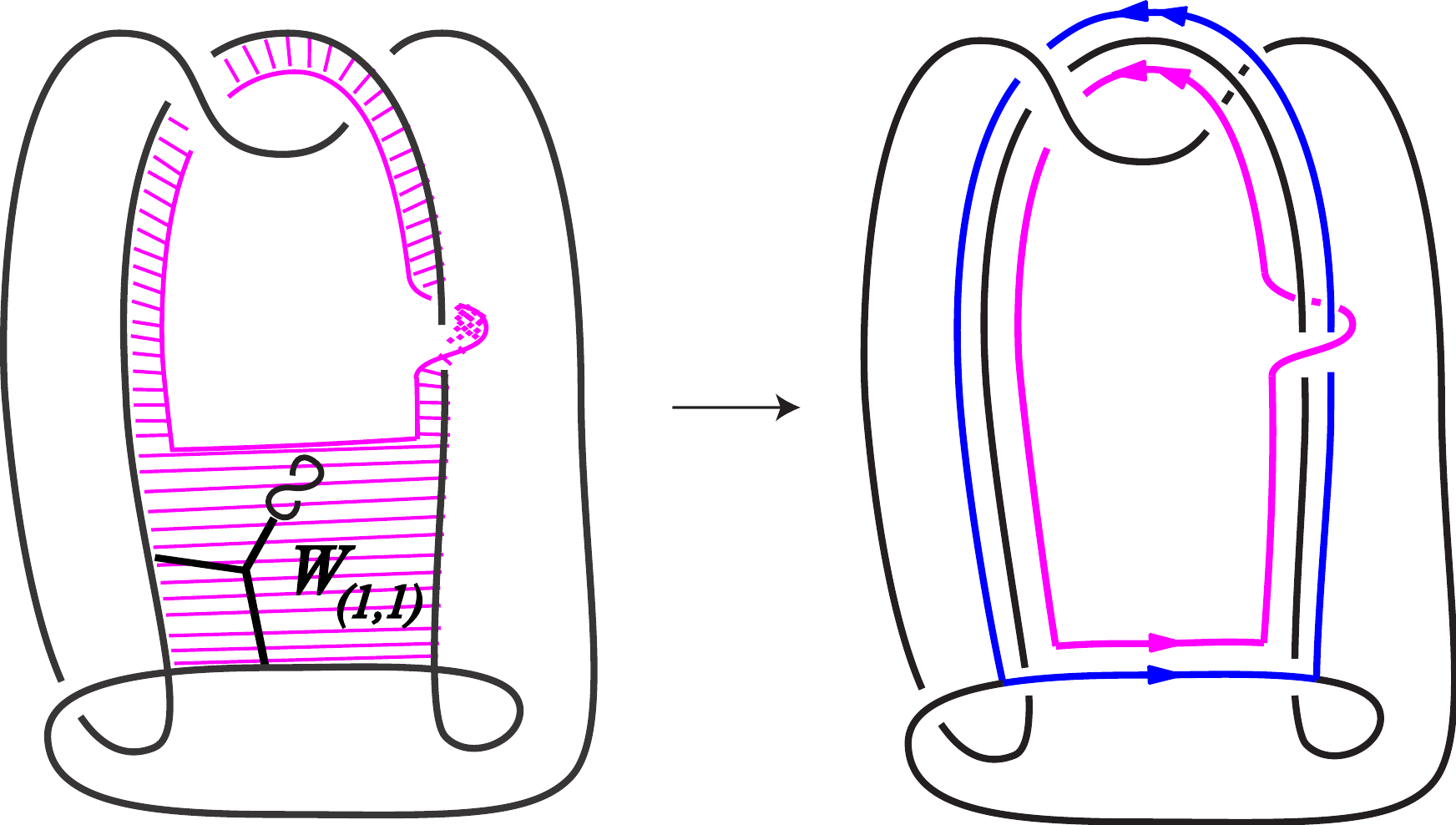}}
         \caption{The two right-most pictures from Figure~\ref{Fig8-knot-twisted-W-fig}, but here with a Whitney section $\overline{\partial W_{(1,1)}}$ shown in blue. Compare Figure~\ref{W-subspaces}. 
The $+1$-linking between the purple and blue circles corresponds to the twisting 
$\omega(W_{(1,1)})=1$. }
         \label{Fig8knot-Wdisk-Whitney-section-fig}
\end{figure}

%
\subsubsection{}\label{eg:no-trees-no-problem} 
If $\cW$ is a Whitney tower on $A$ such that $t(\cW)=\emptyset$,
then $A$ is regularly homotopic to an embedding 
(Exercise~\ref{ex:no-trees-no-problem}).

%

%
\subsection{Higher-order Whitney disks and intersections}\label{subsec:order-of-trees-disks-ints}
We will be using the following grading of our trees:
\begin{defn}\label{def:order-of-trees}
\item The \emph{order} of a tree, rooted or unrooted, is defined to be the number of trivalent vertices.
\end{defn}

Having associated trees to Whitney disks and intersection points  in section~\ref{subsec:trees-for-w-disks-and-ints}, we can define \emph{higher-order Whitney disks} and \emph{higher-order intersections} in Whitney towers:
\begin{defn}\label{def:order-of-w-disks-and-ints}
\item The \emph{order} of a Whitney disk or intersection point in a Whitney tower is defined to be the order of its corresponding tree.
\end{defn}

The components $A_i$ of the underlying surface are also referred to as \emph{order 0 surfaces}.

Linking numbers of links in $S^3$ can be computed by summing order $0$ intersections between order $0$ surfaces in $B^4$ bounded by the links. The higher-order intersections and Whitney disk twistings in a Whitney tower bounded by a link are closely related to ``higher-order'' linking numbers (Milnor's link invariants), as will be explained in Section~\ref{sec:twisted-order-n-classification-arf-conj}.

%
%
%


%
\subsection{Order $n$ framed Whitney towers}

\begin{defn}\label{def:order-n-framed-W-tower}
$\cW$ is an \emph{order $n$ framed Whitney tower} if
\begin{enumerate}
\item 
every framed tree $t_p$ in $t(\cW)$ is of order~$\geq n$, and 
\item
there are no $\iinfty$-trees in $t(\cW)$.
\end{enumerate}
\end{defn}

So in an order $n$ framed $\cW$ all unpaired intersections have order $\geq n$, and all Whitney disks are framed.


%
\subsection{Order $n$ twisted Whitney towers}\label{subsec:def-of-order-n-twisted-W-towers}

\begin{defn}\label{def:order-n-twisted-W-tower}
$\cW$ is an \emph{order $n$ twisted Whitney tower} if 
\begin{enumerate}
\item 
every framed tree $t_p$ in $t(\cW)$ is of order~$\geq n$, 
 and 
\item
every twisted $\iinfty$-tree in $t(\cW)$ is of order~$\geq \frac{n}{2}$.
\end{enumerate}
\end{defn}

So in an order $n$ framed $\cW$ all unpaired intersections have order $\geq n$, and all Whitney disks of order less than $n/2$ are framed.

The reason that Definition~\ref{def:order-n-twisted-W-tower} allows twisted Whitney disks in orders at least $n/2$ will become clear in Section~\ref{sec:twisted-order-n-classification-arf-conj}, where it will be shown (following  \cite{CST1}) that
intersection invariants extracted from the intersection forests of order $n$ twisted Whitney towers, and the corresponding order-raising obstruction theory, lead to a classification of
order $n$ twisted Whitney towers on properly immersed disks in the $4$-ball in terms of Milnor invariants and higher-order Arf invariants of the links on the boundary.
In \cite{CST1} an analogous classification of order $n$ framed Whitney towers is also derived from this twisted classification.
\subsection{Other gradings of Whitney towers}
We mention here just two of the other variations on organizing Whitney towers.
See also \cite[Sec.1.5]{CST1} for a brief comparison discussion.

\subsubsection{Non-repeating order $n$ Whitney towers:}\label{subsubsec:non-repeating-w-towers}

If $t_p$ is the tree associated to an intersection $p$ in a Whitney tower such that $t_p$ has distinct univalent labels, then
$t_p$ is called a \emph{non-repeating tree} and $p$ is called a \emph{non-repeating intersection}.

\begin{defn}\label{def:non-rep-tower}
A Whitney tower $\cW$ is an order $n$ \emph{non-repeating} Whitney tower if all non-repeating $t_p\in t(\cW)$ are of order $\geq n$.
\end{defn}
Note that there are no restrictions on $\iinfty$-trees in order $n$ non-repeating Whitney towers, and no restrictions on any \emph{repeating trees} which do not have distinctly labeled univalent vertices. 

Non-repeating Whitney towers characterize being able to ``pull apart'' components:
\begin{thm}[\cite{ST3}]\label{thm:pull-apart}
$A=\cup^m_{i=1}A_i\imra X$ supports an order $m-1$ non-repeating $\cW$ 
if and only if
$A$ is homotopic rel $\partial A$ to $A'=\cup^m_{i=1}A'_i$
with $A'_i\cap A'_j=\emptyset$ for all $i\neq j$.
\end{thm}
The obstructions to the existence of order $n$ non-repeating Whitney towers are analogous to Milnor's non-repeating link homotopy invariants \cite{M1}, and for immersed disks in the $4$-ball are equivalent to the first non-vanishing non-repeating Milnor invariants \cite[Thm.8]{ST3}.

In Section~\ref{sec:2-spheres-in-4-manifolds} we will examine the obstructions to non-repeating Whitney towers on $2$-spheres in $4$-manifolds in detail for low orders.

%
%

%
\subsubsection{Symmetric Whitney towers:}\label{subsec:symmetric-w-towers}

\begin{defn}\label{def:symmetric-w-tower}
A Whitney tower $\cW$ is \emph{symmetric} if all Whitney disks in $\cW$ are framed and only have interior intersections with Whitney disks of the same order.
\end{defn}

The natural grading of symmetric Whitney towers is by \emph{height}:
\begin{defn}\label{def:height}
A symmetric Whitney tower $\cW$ has \emph{height $n$} if $\cW$ has order $(2n-2)$ as a framed Whitney tower.
\end{defn}

The Whitney disks in a symmetric Whitney tower correspond to symmetric rooted trees $Y^n$ defined in Figure~\ref{symmetric-trees-fig}.
\begin{figure}[h]
        \centerline{\includegraphics[scale=.5]{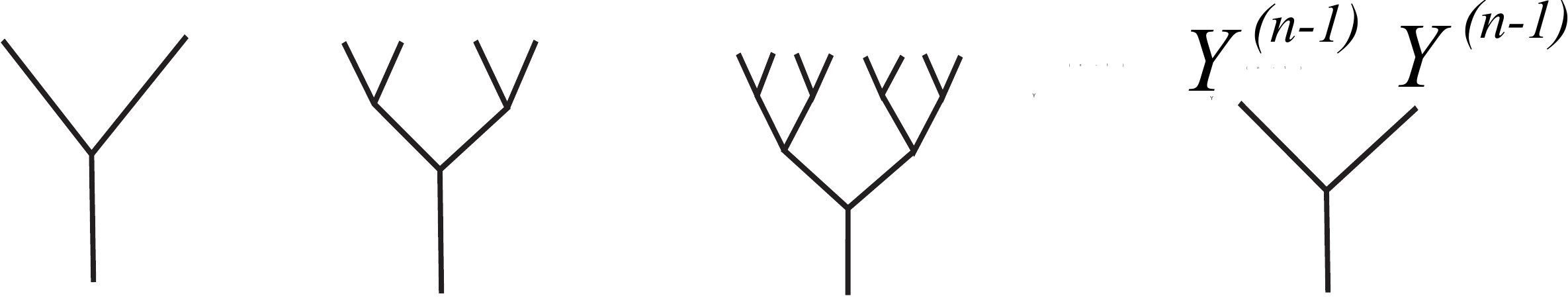}}
        \caption{The symmetric rooted-trees of height 1, 2, 3, and $n$}
        \label{symmetric-trees-fig}
\end{figure}

So an equivalent definition of a symmetric Whitney tower of height $n$ is that $t(\cW)$ only contains trees which are inner products $\langle Y^n,Y^n\rangle$ of height $n$ symmetric rooted trees.

We also have the following variation extending this definition to half-integers: If $t(\cW)$ only contains trees which are inner products $\langle Y^n,Y^{n-1}\rangle$, then $\cW$ is said to be of \emph{height $n.5$}.

\begin{thm}[\cite{COT}]
If a link $L\subset S^3$ bounds immersed disks into $B^4$ supporting $\cW$ of height $n+2$, then $L$ is \emph{$n$-solvable} in the sense of Cochran--Orr--Teichner.
\end{thm}

A knot bounds a height~$2$ Whitney tower if and only if it has vanishing Arf invariant; a knot bounds a height~$2.5$ Whitney tower if only if it is algebraically slice;
and in general signature invariants are known to provide upper bounds on the possible heights of Whitney towers bounded by knots \cite{COT}. 
There is an extensive literature related to the $n$-solvable filtration of classical concordance.

A complete height-raising obstruction theory for symmetric Whitney towers is not known. Necessary conditions for raising height can be formulated in terms of the vanishing of signatures, but algebraic invariants whose vanishing suffices for raising height have not been formulated in general.
The order-raising construction given in section~\ref{subsec:order-raising-proof-sketch} destroys any Whitney disk symmetry. 
First open case: Given a knot in $S^3$ bounding $\cW\subset B^4$ of height $2.5$, there are no known obstructions to the knot bounding a Whitney tower of height $3$ in $B^4$.

%
%


\subsection{Tree Orientations}\label{subsec:w-tower-tree-orientations}
Recall that our Whitney towers are oriented, with the underlying surface $A$ coming with a given orientation, and with fixed orientations chosen for the Whitney disks in $\cW$.  
With an eye towards later defining invariants in terms of $t(\cW)$ which will be independent of Whitney disk orientations, we introduce here two conventions from the Whitney tower literature which relate Whitney disk and tree orientations.  Via either of these two conventions, the choices of Whitney disk orientations will be seen in section~\ref{subsec:w-disk-orientation-choices-AS} to correspond to \emph{antisymmetry relations} among the trees, and will be quotiented out in the target of the invariants.
 
As per section~\ref{subsec:trees-for-w-disks-and-ints},
the rooted tree associated to a Whitney disk $W$ in $\cW$ can be 
mapped into $\cW$, with the trivalent vertex adjacent to the root contained in the interior of $W$, and each other trivalent vertex contained in the interior of a lower-order Whitney disk supporting $W$.
The two descending edges from each trivalent vertex determine a ``corner'' of the corresponding Whitney disk that does not contain the third ascending edge. For instance, in Figure~\ref{WdiskIJandIJKint-fig} the corner determined by the descending edges from the indicated trivalent vertex contains the negative intersection point paired by the corresponding Whitney disk.
And in the right side of Figure~\ref{higher-order-intersection-color-and-with-tree} the corners of the Whitney disks determined by the descending edges of the two trivalent vertices of the tree each contain the negative intersection paired by the Whitney disk. 

With this terminology we define the following conventions for aligning the vertex orientations on trees with the chosen orientations on Whitney disks:

\textbf{The negative corner convention:} The trees associated to all Whitney disks in $\cW$ are 
mapped into $\cW$ with the requirement that the corner of each Whitney disk determined by the two descending edges of each
trivalent vertex contains the \emph{negative} intersection point between the sheets paired by the corresponding Whitney disk.
Then the vertex orientations of each tree are taken to be induced from
the Whitney disk orientations.

\textbf{The positive corner convention:} The trees associated to all Whitney disks in $\cW$ are 
mapped into $\cW$ with the requirement that the corner of each Whitney disk determined by the two descending edges of each
trivalent vertex contains the \emph{positive} intersection point between the sheets paired by the corresponding Whitney disk.
Then the vertex orientations of each tree are taken to be induced from
the Whitney disk orientations.



\begin{figure}[h]
\centerline{\includegraphics[scale=.4]{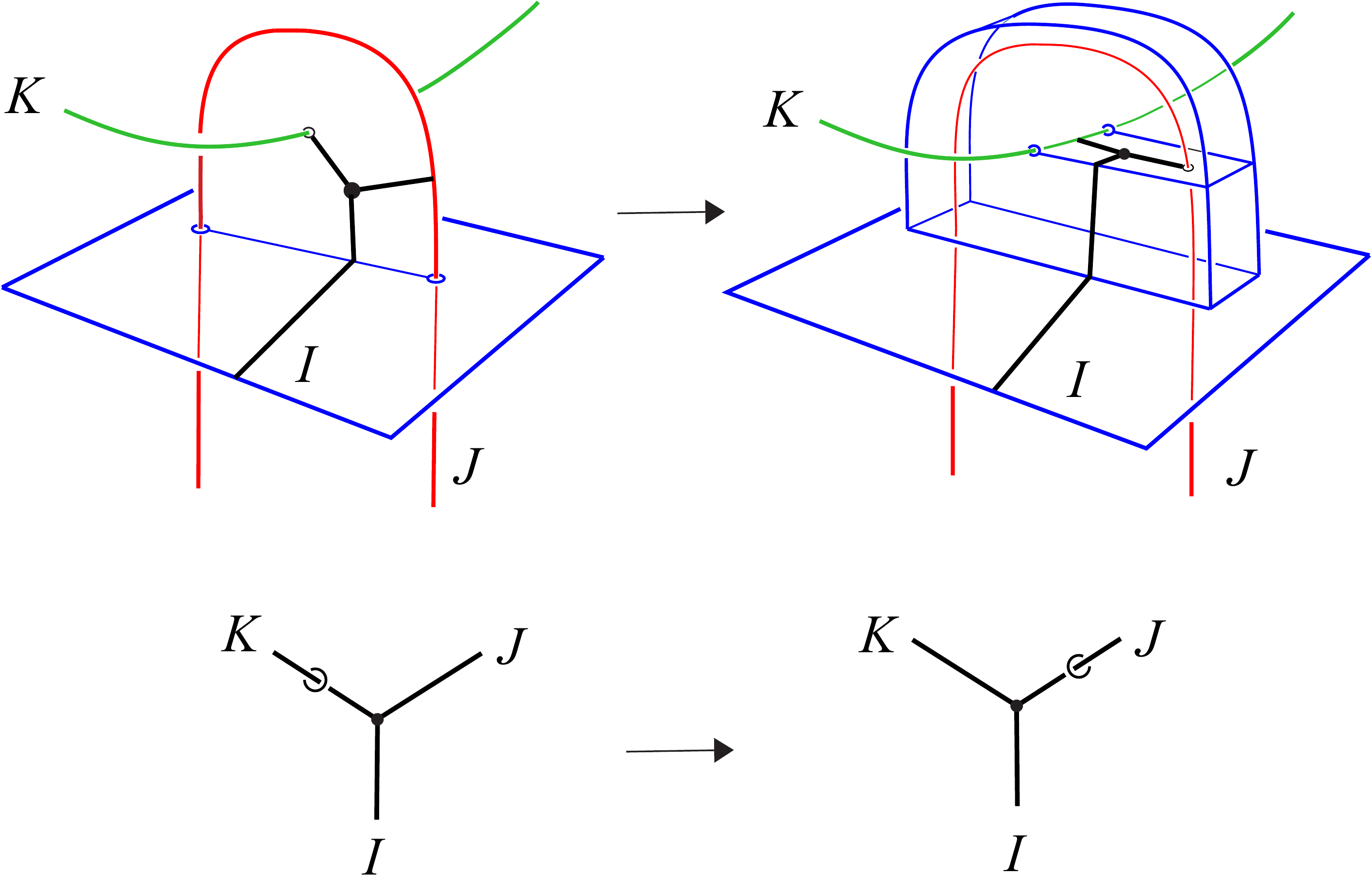}}
\caption{Top: A $W_{(I,J)}$-Whitney move on the $I$-sheet in a split Whitney tower creates a pair of intersections which can be paired by a `meridional' Whitney disk $W_{(K,I)}$ that intersects the $J$-sheet in a single point. Bottom: Indicating the unpaired intersection by a small circle linking its corresponding edge, the top construction realizes $\langle (I,J),K \rangle=\langle (K,I),J \rangle$.}
\label{fig:move-int-puncture-in-tree}
\end{figure}

\subsection{`Moving' unpaired intersections in their trees}\label{subsec:moving-unpaired-int-edge}
Again with an eye towards extracting invariants from $t(\cW)$, examination of our notation reveals that the tree $t_p:=\langle (I,J),K \rangle$
associated to
$p\in W_{(I,J)}\pitchfork W_K$ does not keep track of which edge of $t_p$ contains the unpaired intersection $p$, since for instance $\langle (I,J),K \rangle=\langle I,(J,K) \rangle=\langle (K,I),J \rangle$. It seems reasonable ask ``Why not also keep track of the edge in $t_p$ corresponding to $p$?'', but it turns out that robust information is carried by the shape of the tree rather than the intersection point itself. This is illustrated in 
Figure~\ref{fig:move-int-puncture-in-tree}, which shows the geometric realization of $\langle (I,J),K \rangle=\langle (K,I),J \rangle$ by a controlled local construction that does not create any new trees.
This construction preserves the signed tree of the unpaired intersection point, using either convention in section~\ref{subsec:w-tower-tree-orientations} (Exercise~\ref{ex:move-int-preserves-sign}).
By iterating this construction an un-paired intersection can be ``moved to any edge of its tree''.

\subsection{Geometric Jacobi Identity in four dimensions}\label{subsec:geo-4d-jacobi}
We close this section by showing the neccessity of certain relations in the target of any homotopy invariant of $A$ represented by $t(\cW)$ for $\cW$ supported by $A$. The proof of the following theorem provides a nice illustration of the material covered so far:
\begin{thm}[\cite{CST}]\label{thm:IHX}
There exist four $2$-spheres $A_1\cup A_2\cup A_3\cup A_4\imra S^4$ in the $4$-sphere supporting $\cW$ with intersection forest $t(\cW)$ equal to the three signed trees in Figure~\ref{Jacobi-identity-trees}.
\begin{figure}[h]
        \centerline{\includegraphics[width=90mm]{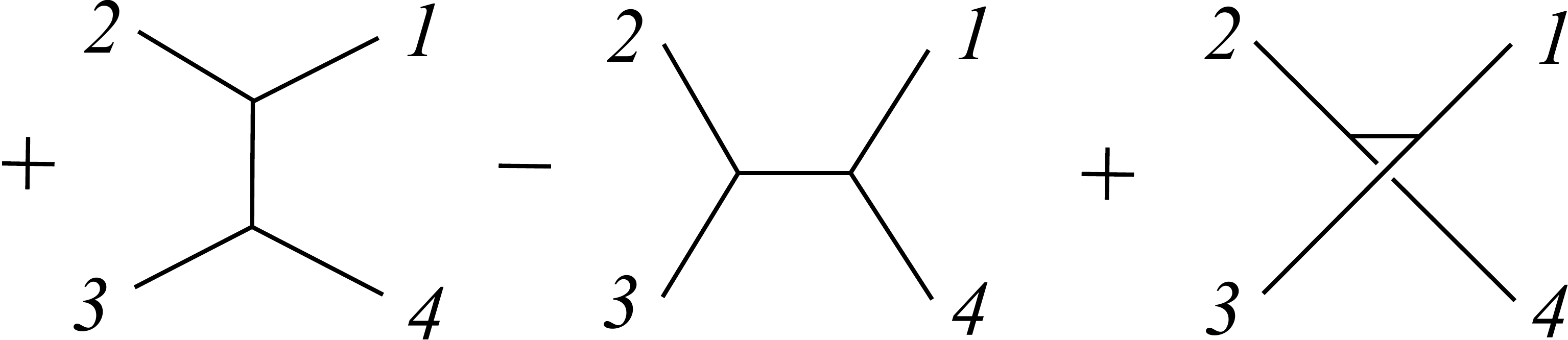}}
        \caption{}
        \label{Jacobi-identity-trees}
\end{figure}
\end{thm}
This is a ``Jacobi Identity'' in the sense that taking the $4$-labeled univalent vertices as roots and identifying the resulting rooted trees with Lie brackets yields the three terms $[[1,2],3]-[1,[2,3]]+[[3,1],2]$, and this sum of trees must vanish in the target of any homotopy invariant represented by $t(\cW)$ since $\cW$ is supported by homotopically trivial $2$-spheres.


\begin{proof}[Proof of Theorem~\ref{thm:IHX}]
Start with disjoint embeddings $A_i:S^2\to B^4\subset S^4$, $i=1,2,3,4$.

Then do a single finger move on each of $A_1,A_2,A_3$ into $A_4$, yielding the left picture of Figure~\ref{IHX-1-fig}, in which $A_4$ appears as the horizontal plane of each picture. The Whitney disks $W_{(3,4)}$, $W_{(4,1)}$ and $W_{(2,4)}$ in the center picture of Figure~\ref{IHX-1-fig} are inverse to the finger moves (Exercise~\ref{ex:W-move-inverse-to-finger-move}).
\begin{figure}[h!]
        \centerline{\includegraphics[width=\textwidth]{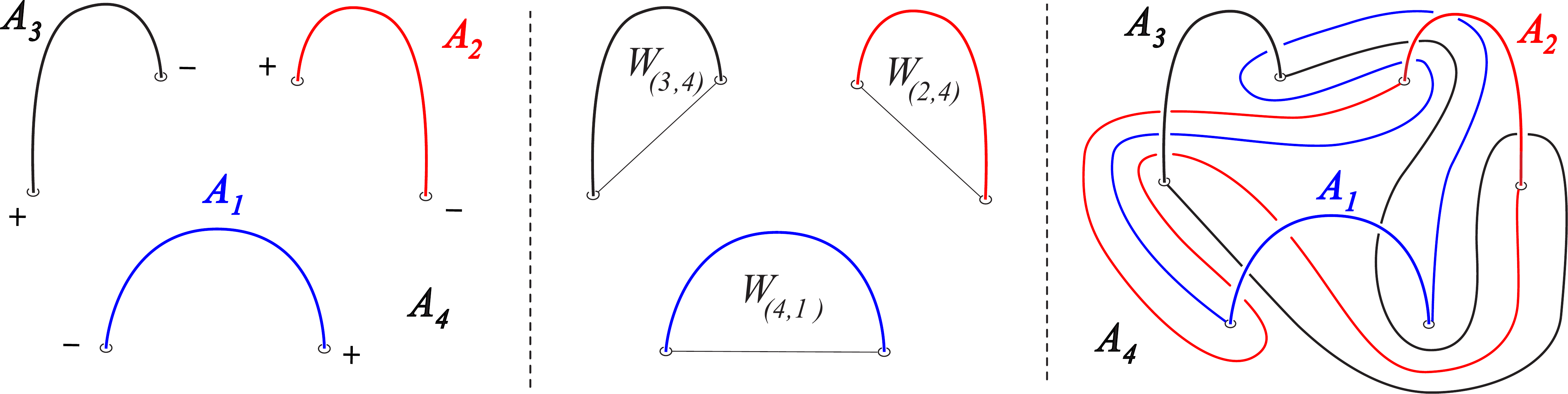}}
        \caption{}
        \label{IHX-1-fig}
\end{figure}

We will describe modifications of the order $1$ Whitney disks in the center of Figure~\ref{IHX-1-fig} that yield new disjointly embedded framed order $1$ Whitney disks which will have the boundaries shown in the right picture of Figure~\ref{IHX-1-fig}.
These modifications will create pairs of order $1$ intersections that admit disjointly embedded framed order $2$ Whitney disks which each have a single order $2$ intersection corresponding to a term of the Jacobi identity. The entire construction will be supported in the $4$-ball described by Figure~\ref{IHX-1-fig} as the present, and extending into past and future where the arcs of $A_1$, $A_2$ and $A_3$ extend as products.

First we will construct the left-most term of the Jacobi identity in Figure~\ref{Jacobi-identity-trees}, suppressing orientations for the moment. Start by changing a collar of $W_{(3,4)}$ as indicated in the left of Figure~\ref{IHX-3-4-fig}. Still referring to this new Whitney disk as $W_{(3,4)}$, we have created a new pair of order $1$ intersections $\{q,r\}=A_2\pitchfork W_{(3,4)}$, as shown in the figure. Since the collar of $W_{(3,4)}$ containing $q$ and $r$ is parallel to $A_4$, it follows that $q$ and $r$ have opposite signs, since the  intersections $A_2\pitchfork A_4$ created by the finger move of $A_2$ into $A_4$ have opposite signs. 
\begin{figure}[ht!]
        \centerline{\includegraphics[scale=.45]{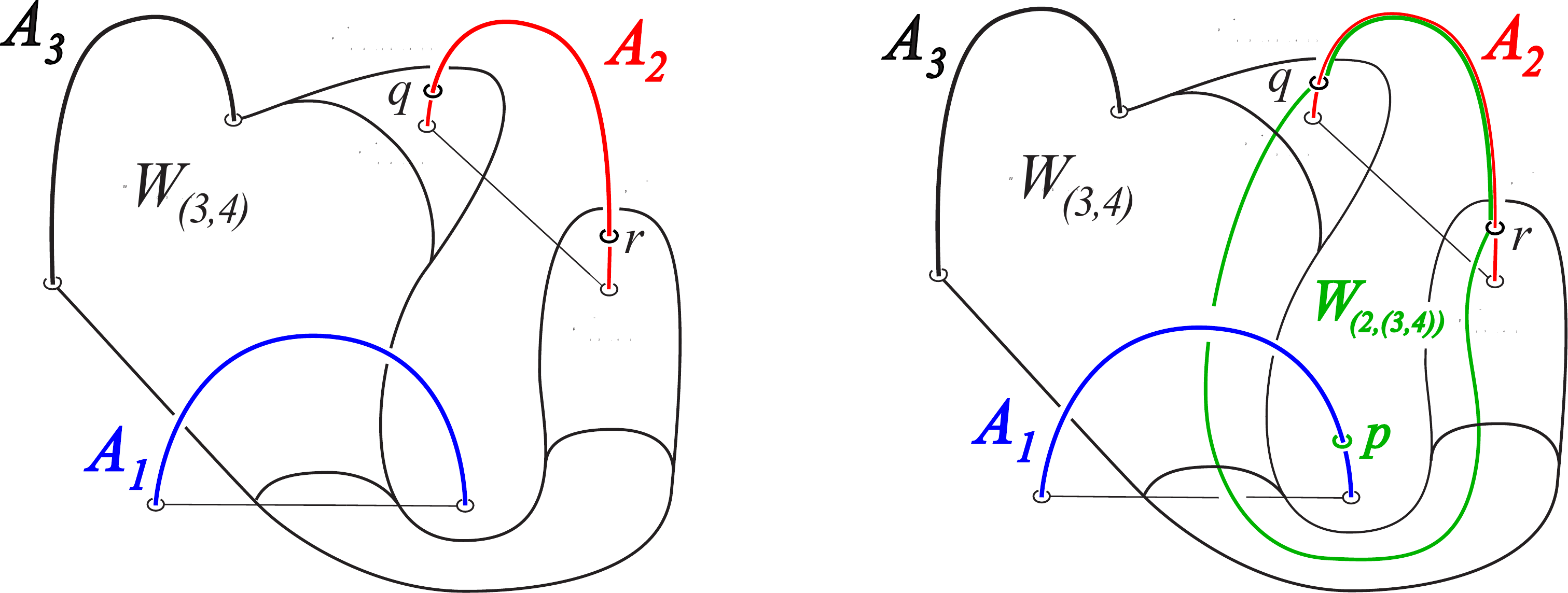}}
        \caption{}
        \label{IHX-3-4-fig}
\end{figure}
Next add an order 2 Whitney disk $W_{(2,(3,4))}$ pairing $q$ and $r$ as on the right of Figure~\ref{IHX-3-4-fig}, where $\partial W_{(2,(3,4))}$ is shown as green. Part of $W_{(2,(3,4))}$ is formed from the original order $1$ Whitney disk $W_{(2,4)}$ with a collar of the arc of $\partial W_{(2,4)}$ on $A_4$ removed, and the rest of $W_{(2,(3,4))}$ is parallel to $A_4$. This creates a single order $2$ intersection
$p=A_1\cap W_{(2,(3,4))}$, as shown in the figure.
This new order 1 $W_{(3,4)}$ and the order 2 $W_{(2,(3,4))}$ are each embedded, and completely contained in the present.

In the left of Figure~\ref{IHX-5-and-with-other-arcs-fig} the tree $t_p=^{\,\,3}_{\,\,4}> \!\!\!-\!\!\!<^{\scriptstyle 2}_{\scriptstyle 1}$ associated to $p=A_1\cap W_{(2,(3,4))}$ is included (but suppressed from view are the continuations of the edges changing sheets into the order~0 2-spheres).
If $W_{(3,4)}$ and $W_{(2,(3,4))}$ are oriented to agree with the orientation of $A_4$ where they are parallel to $A_4$, then
this embedding of $t_p$ conforms to the positive corner convention of section~\ref{subsec:w-tower-tree-orientations}, and the sign $\epsilon_p$ of $p$ is positive. 
So $\epsilon_p\cdot t_p$ is the left-most term in the Jacobi relator of Figure~\ref{Jacobi-identity-trees}.

\begin{figure}[h!]
        \centerline{\includegraphics[scale=.45]{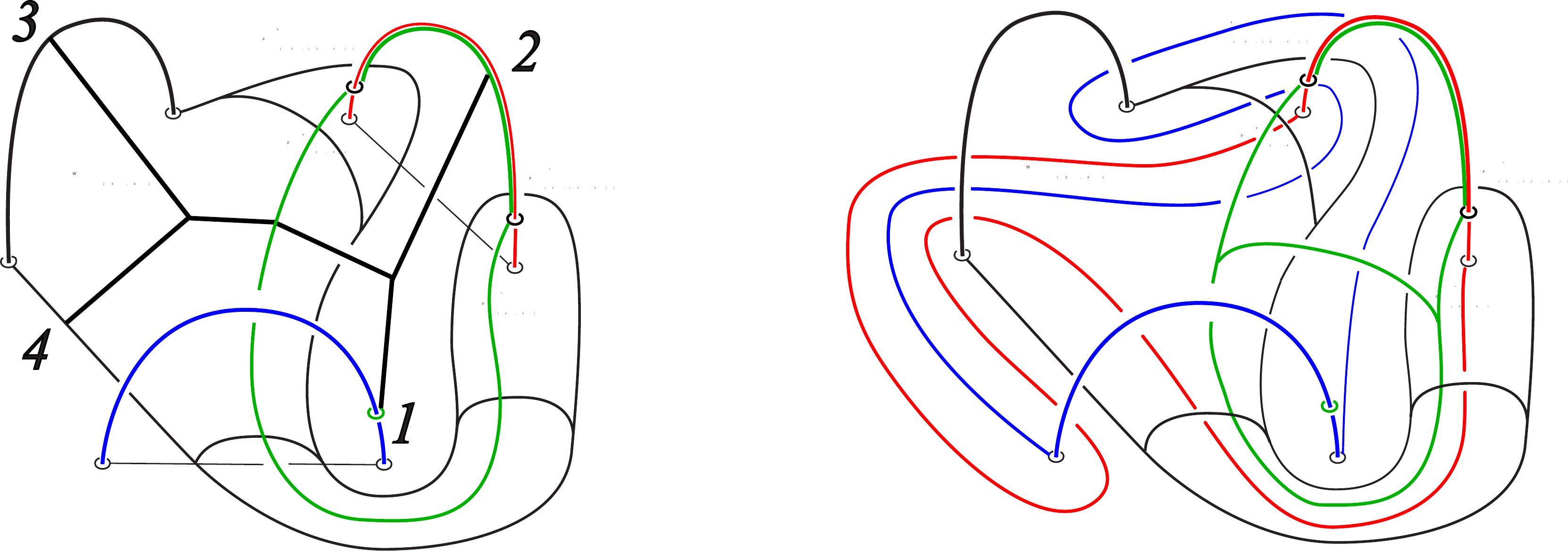}}
        \caption{}
        \label{IHX-5-and-with-other-arcs-fig}
\end{figure}
It remains to construct the other two terms in Figure~\ref{Jacobi-identity-trees}. The right side of Figure~\ref{IHX-5-and-with-other-arcs-fig} shows the just constructed $W_{(3,4)}$ and $W_{(2,(3,4))}$, and the red and blue boundaries of the new order 1 Whitney disks that we want to create. 
Observe that the parts of the red and blue boundaries that lie on $A_4$ extend to small embedded collars in the present that are disjoint from $W_{(3,4)}$ and $W_{(2,(3,4))}$ as well as the four 2-spheres. So by extending the inner collar boundary of red into the past, and the inner collar boundary of blue into future, the same construction that was just done in the present can be carried out in nearby past and future $B^3$-slices to yield the other two desired trees. See Exercise~\ref{exercise:finish-ixh}.  

Note that the construction necessarily extends into both past and future because it yields new disjointly embedded order 1 Whitney disks with boundaries as in the right of Figure~\ref{IHX-1-fig},
and this configuration forms the Borromean Rings which is not a slice link.
\end{proof}
%

\begin{cor}\label{cor:IHX}
For any Whitney tower $\cW$ on a properly immersed surface $A$ in a $4$--manifold, the local `IHX relation' of finite type theory (Figure~\ref{IHX-relation-fig}) is needed in the target of any invariant of $A$ represented by $t(\cW)$.
\begin{figure}[h]
        \centerline{\includegraphics[scale=.35]{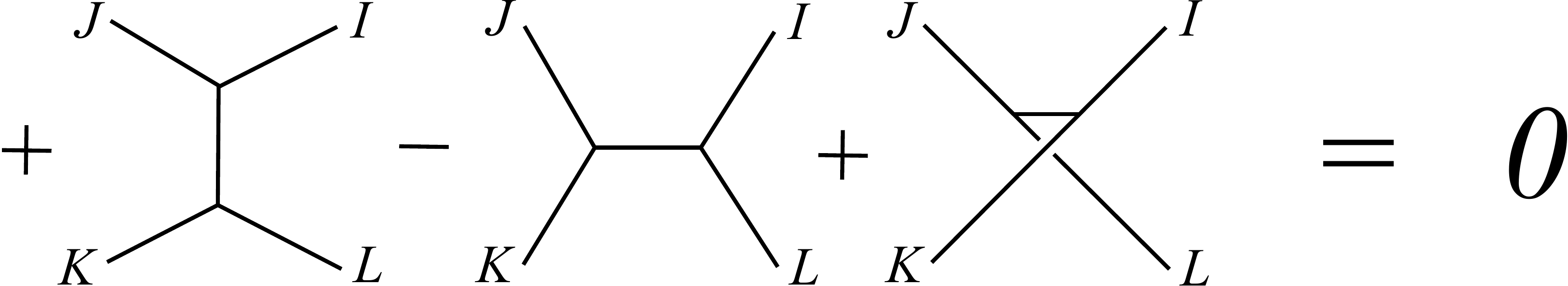}}
        \caption{The IHX relation.}
        \label{IHX-relation-fig}
\end{figure}
\end{cor}

In this more general local relation the univalent vertices of the three trees represent arbitrary (fixed) sub-trees.
To see that this corollary follows from Theorem~\ref{thm:IHX}, observe that clean Whitney disks $W_I$, $W_J$, $W_K$ and $W_L$ corresponding to any given rooted trees $I$, $J$, $K$ and $L$ can be created by finger moves (Exercise~\ref{ex:create-clean-W-for-any-tree}). Then, after a connected sum with $S^4$ along a $4$--ball in the complement of $\cW$, tubing the $2$-spheres from the theorem into the interiors of these Whitney disks creates exactly the IHX relator in the corollary, since the tubes are supported near arcs which can be taken to be disjoint from the $2$-complex $\cW$.


\subsection{Section~\ref{sec:disks-in-B4} Exercises}

\subsubsection{Exercise:}
Visualize the $2$-component unlink in $S^3$ bounding the disjoint disks in $B^4$ shown in Figure~\ref{disjoint-sheets-fig}. 
\begin{figure}[h]
\includegraphics[scale=.6]{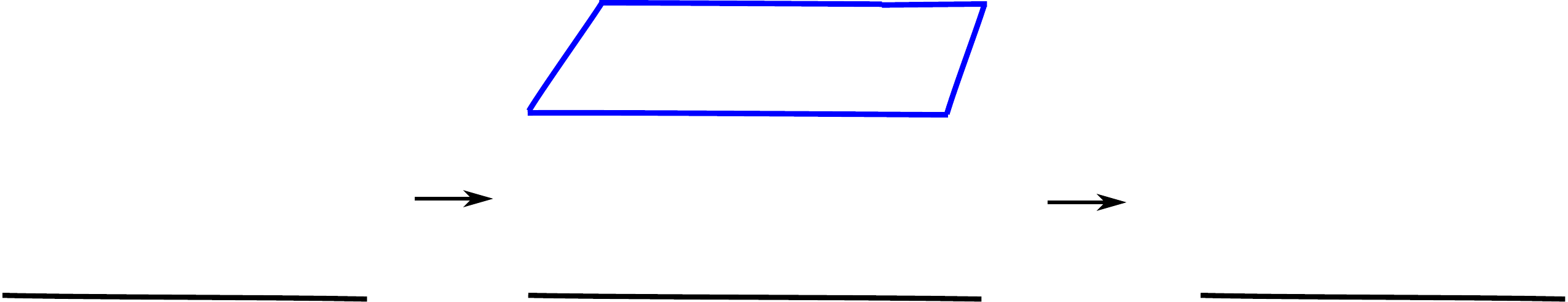}
\caption{Disjoint disks in $B^4=B^3\times I$}
\label{disjoint-sheets-fig}
\end{figure}

\subsubsection{Exercise:}\label{exercise:see-hopf-link-on-boundary}
Visualize the Hopf link $=$ $\partial A\cup\partial B\subset S^3=\partial (B^3\times I)$ in each of Figure~\ref{transverse-intersection-fig-1}
and  Figure~\ref{transverse-intersection-fig-2}. 

\subsubsection{Exercise:}\label{ex:finger-move-symmetric}
Starting with the bottom row of Figure~\ref{finger-move-movies-fig},
draw the finger move that pushes the horizontal blue sheet down into the black sheet. Observe that the results of the two finger moves are isotopic. 

\subsubsection{Exercise:}\label{ex:finger-move-ints-have-opposite-signs}
Check that the two intersections created by a finger move have opposite signs.
(Recall that the sign $\epsilon_p\in\{+,-\}$ of an intersection $p$ between oriented sheets $A$ and $B$ in an oriented $4$-manifold $X$ is defined to be $+$ (respectively, $-$) if the orientation of $X$ at $p$ agrees (respectively, disagrees) with
the concatenation of the orientations of $A$ and $B$ at $p$.)

\subsubsection{Exercise:}\label{ex:W-move-symmetric}
Check that Figure~\ref{W-move-other-sheet-fig} shows the result of a model Whitney move which adds the Whitney bubble to the black sheet instead of the blue sheet in Figure~\ref{fig:model-W-move}. (The two copies of the Whitney disk in the Whitney bubble are in grey).
Convince yourself that a model Whitney move is symmetric up to isotopy by finding an isotopy between Figure~\ref{W-move-other-sheet-fig} and the right side of Figure~\ref{fig:model-W-move}. 
\begin{figure}[h]
\includegraphics[width=\textwidth]{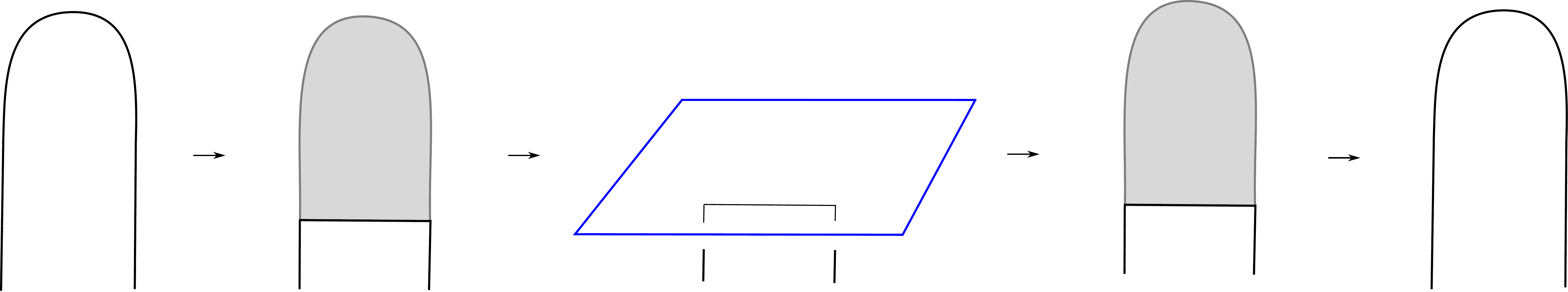}
\caption{}
\label{W-move-other-sheet-fig}
\end{figure}

\subsubsection{Exercise:}\label{ex:W-move-inverse-to-finger-move}
The oppositely-signed pair of intersections created by a finger move (Figure~\ref{finger-move-before-and-after-fig}) are contained in a local $4$-ball and admit a model Whitney disk (Figure~\ref{Whitney-disk-pic-and-movie}). Check that the result of doing a Whitney move on such a Whitney disk (Figure~\ref{fig:model-W-move}) is isotopic to not doing the finger move in the first place. We sometimes say that such a Whitney disk is ``inverse'' to the finger move.

\subsubsection{Exercise:}\label{ex:count-w-move-ints}
Describe the new intersections created by a $W$-Whitney move in terms of the interior intersections $W$ has with surfaces, the self-intersections of $W$, and the twisting $\omega(W)$.

\subsubsection{Exercise:}\label{exercise:boro-rings-on-boundary}
Visualize the Borromean Rings $\partial A\cup\partial B\cup \partial C\subset S^3=\partial (B^3\times I)$ in (both sides of) Figure~\ref{fig:W-disk-int-and-W-move-color}. HINT: See the Bing-double of one component of Exercise~\ref{exercise:see-hopf-link-on-boundary}.

\subsubsection{Exercise:}\label{exercise:bing-hopf-on-boundary}
Visualize the Bing-double of the Hopf link as the boundaries of the red, blue, green and yellow sheets in $S^3=\partial (B^3\times I)$ in the left side of Figure~\ref{higher-order-intersection-color-and-with-tree}. 


\subsubsection{Exercise:}\label{ex:no-trees-no-problem}
Show that if $\cW$ is a Whitney tower on $A$ such that $t(\cW)=\emptyset$,
then $A$ is regularly homotopic to an embedding.


\subsubsection{Exercise:}\label{ex:move-int-preserves-sign}
Check that the construction of Figure~\ref{fig:move-int-puncture-in-tree} 
preserves the sign of the unpaired intersection point using either the positive or negative convention in section~\ref{subsec:w-tower-tree-orientations}. You may assume that the embedded trees in the figure satisfy the orientation convention, and check that this implies that the sign of the unpaired intersections are the same on the left and right.

\subsubsection{Exercise:}\label{ex:create-clean-W-for-any-tree}
Let $\cW$ be a Whitney tower on a surface $A$, and let $I$ be any rooted tree.
Show that, without changing $t(\cW)$, a clean framed Whitney disk $W_I$ with associated tree $I$ can be created by performing finger moves.
(HINT: A finger move is supported near an arc, hence can be arranged to miss any other surface.)

\subsubsection{Exercise:}\label{ex:push-down}
Figure~\ref{fig:W-disk-int-and-push-down} shows how an intersection between a green sheet and the interior of a Whitney disk $W$ can be eliminated by a finger move, at the cost of creating an oppositely-signed pair of intersections between the green sheet and one of the sheets paired by $W$. This is called ``pushing down'' an intersection in $W$. Use pushing down to show that if $A$ supports an order~1 framed Whitney tower $\cW$, then $A$ supports an order~1 framed Whitney tower $\cW'$ whose Whitney disks are disjointly embedded.

\begin{figure}[h]
\includegraphics[scale=.6]{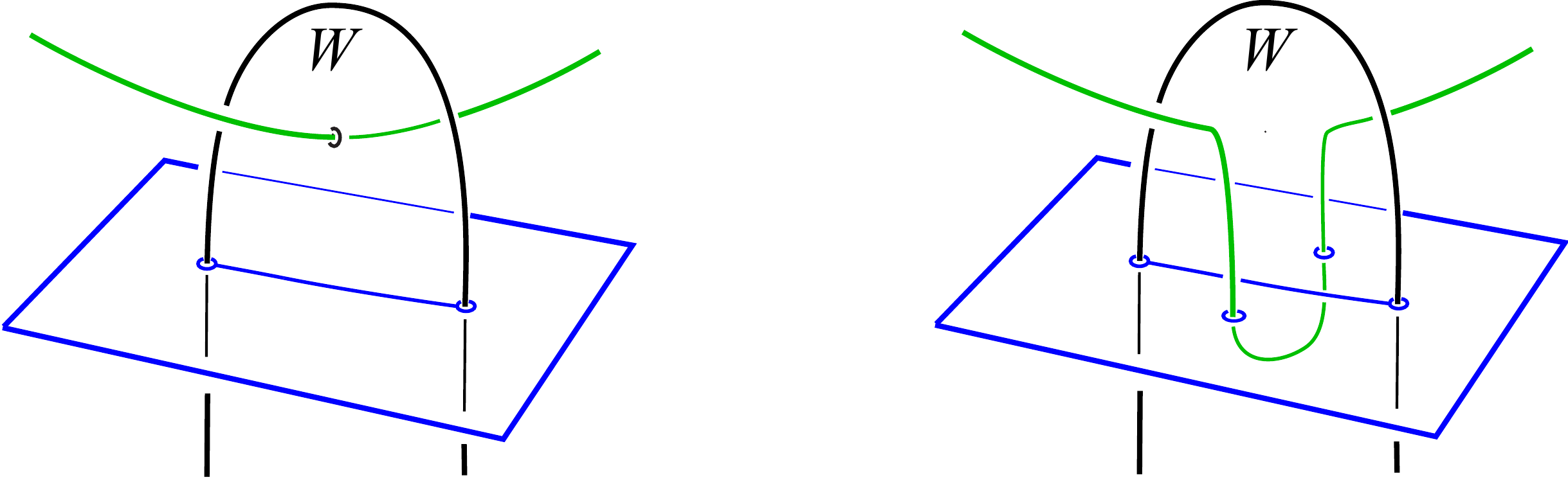}
\caption{`Pushing down' an intersection.
}
\label{fig:W-disk-int-and-push-down}
\end{figure}

\subsubsection{Exercise:}\label{ex:push-w-disk-boundary}
The green-blue intersections created by the finger move in Figure~\ref{fig:W-disk-int-and-push-down} admit a local Whitney disk $V$ which is inverse to the finger move (Exercise~\ref{ex:W-move-inverse-to-finger-move}), such that $\partial V$ intersects $\partial W$ transversely in the green sheet. Draw $V$ into the right side of Figure~\ref{fig:W-disk-int-and-push-down} (hanging down ``underneath'' the blue sheet). Observe that $V$ and $W$ can be made disjoint by extending to a collar of $V$ an isotopy of $\partial V$ that pushes $\partial V\pitchfork\partial W$ away across either blue-black intersection paired by $W$. This isotopy of $V$ creates an interior intersection between (the new) $V$ and the black sheet. Draw this new $V$, including its interior intersection with the black sheet, and check that this intersection between $V$ and black has the same order as the original intersection between $W$ and the green sheet.

\subsubsection{Exercise:}\label{ex:embedded-trees}
Check that the trees $t_p\subset \cW$ of section~\ref{subsec:trees-for-w-disks-and-ints} can always be arranged to be disjointly embedded in $\cW$, for any number of unpaired intersections in $\cW$. (Do not assume that $\cW$ is split (Figure~\ref{split-w-tower-with-trees-fig}), and don't forget the possibility of self-intersections in a Whitney disk.)

\subsubsection{Exercise:}\label{exercise:finish-ixh}
Complete the proof of Theorem~\ref{thm:IHX} by constructing the other two trees of Figure~\ref{Jacobi-identity-trees} analogously using past and future
(see \cite{CST} for the solution), and check that all the Whitney disks are framed.




\section{Order $n$ twisted Whitney towers in the $4$-ball}\label{sec:twisted-order-n-classification-arf-conj}

Recall from Definition~\ref{def:order-n-twisted-W-tower} that a Whitney tower $\cW$ is \emph{twisted of order $n$}
if 
every unpaired intersection $p$ in $\cW$ is of order~$\geq n$ 
and 
every twisted Whitney disk $W_J$ in $\cW$ is of order~$\geq \frac{n}{2}$;
or equivalently, if
every framed tree $t_p$ in the intersection forest $t(\cW)$ is of order~$\geq n$ 
and 
every twisted tree $J^\iinfty$ in $t(\cW)$ is of order~$\geq \frac{n}{2}$.

We say that a link $L\subset S^3=\partial B^4$ \emph{bounds} an order $n$ twisted Whitney tower $\cW$ if
$\cW\subset B^4$ is an order $n$ twisted Whitney tower whose order zero surfaces are immersed disks bounded by the components of $L$. 

Following \cite{CST1,CST2}, using algebraic results from \cite{CST3}, this section describes a classification of links in $S^3$ bounding order $n$ twisted Whitney towers:

Section~\ref{subsec:twisted-tree-groups} defines abelian groups $\cT^\iinfty_n$ generated by trees, with relations corresponding to controlled modifications of twisted Whitney towers.
These groups are the targets for intersection invariants $\tau_n^\iinfty(\cW):=[t(\cW)]\in \cT^\iinfty_n$ defined in section~\ref{subsec:twisted-order-n-invariant} which have the property
that
$L$ bounds an order $n$ twisted $\cW$ with $\tau_n^\iinfty(\cW)=0$ if and only if $L$ bounds an order $n+1$ twisted Whitney tower.
The identification of these order-raising obstruction-theoretic invariants with Milnor invariants and higher-order Arf invariants of the link on the boundary leads to a classification of order $n$ twisted Whitney towers in $B^4$ (Corollary~\ref{cor:mu-arf-classify-twisted}).

 
The higher-order Arf invariants appear for each $n\equiv 2$ mod $4$, and take values in specific finite $\Z/2\Z$-vector spaces (Definition~\ref{def:Arf-j}). 
In this setting the main open problem is to determine precisely the image of these higher-order Arf invariants, which the \emph{Higher-order Arf invariant Conjecture} states is maximal (Conjecture~\ref{conj:higher-order-arf}).

For links bounding order $n$ \emph{framed} Whitney towers (Definition~\ref{def:order-n-framed-W-tower}) there is an analogous classification that is more complicated to describe because, in addition to Milnor invariants and higher-order Arf invariants, it also involves \emph{higher-order Sato--Levine invariants} which represent obstructions to framing odd order twisted Whitney towers and correspond to certain projections of Milnor invariants.  
In \cite{CST1} the framed classification is derived from the twisted one described here. 

We also remark that Jae Choon Cha has shown that the higher-order Arf invariants measure the potential difference between the existence of twisted $\cW$ in $B^4$ versus rational homology $B^4$s.
Namely, $L\subset S^3$ has vanishing Milnor invariants through order $n$ if and only if $L$ bounds an order $n+1$ twisted $\cW$ in a \emph{rational homology $B^4$} (see \cite[Thm.C]{Cha2}).

Throughout this section the fixed index set $\{1,2,,\ldots,m\}$ is usually suppressed from notation.



\subsection{The order $n$ twisted tree groups}\label{subsec:twisted-tree-groups}
Recall 
our terminology and conventions for trees (sections~\ref{subsec:trees}--\ref{subsec:int-forest-def}, and \ref{subsec:order-of-trees-disks-ints}), including that \emph{order} is the number of trivalent vertices.

First we define framed tree groups:
\begin{defn}\label{def:untwisted-tree-groups}
Denote by
$\cT_n$ the free abelian group on order $n$ framed trees modulo the local \emph{antisymmetry} (AS) and \emph{Jacobi identity} (IHX) relations in Figure~\ref{fig:ASandIHXtree-relations}.
\end{defn}
\begin{figure}[h]
\centerline{\includegraphics[scale=.9]{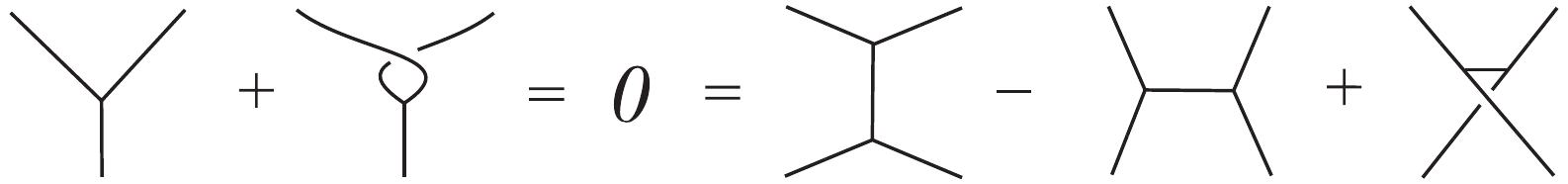}}
\caption{The AS (left) and IHX (right) local relations. Here univalent vertices represent arbitrary fixed sub-trees.}
         \label{fig:ASandIHXtree-relations}
\end{figure}


The target twisted tree groups $\cT^\iinfty_n$ for the intersection invariants $\tau_n^\iinfty(\cW)$ for order $n$ twisted Whitney towers $\cW\subset B^4$ bounded by $L\subset S^3$ will be defined separately for odd and even $n$.
After giving the definitions in terms of generating trees and relations, the geometric meaning of the relations will be discussed.


\begin{defn}
For each $j\geq 1$, the \emph{order $2j-1$ twisted tree group} $\cT^{\iinfty}_{2j-1}$ is the quotient of $\cT_{2j-1}$ by \em{boundary-twist} relations:
\[
\quad i\,-\!\!\!\!\!-\!\!\!<^{\,J}_{\,J}\,\,=\,0
\] 
where $J$ ranges over all order $j-1$ subtrees. 
\end{defn}


\begin{defn}
For each $j\geq 0$, the \emph{order $2j$ twisted tree group} $\cT^{\iinfty}_{2j}$ is the quotient of the free abelian group on framed trees of order $2j$ 
and $\iinfty$-trees of order $j$ by the following relations:
\begin{enumerate}
     \item \emph{AS} and \emph{IHX} relations on order $2j$ framed trees
   \item \emph{symmetry} relations: $(-J)^\iinfty = J^\iinfty$
  \item\label{item:twisted-IHX-relation-def} \emph{twisted IHX} relations: $I^\iinfty=H^\iinfty+X^\iinfty- \langle H,X\rangle $
   \item {\em interior-twist} relations: $2\cdot J^\iinfty=\langle J,J\rangle $
\end{enumerate}
\end{defn}

In item~(\ref{item:twisted-IHX-relation-def}) the three twisted trees differ locally as in the right of Figure~\ref{fig:ASandIHXtree-relations}.

See \cite[Sec.4.8]{CST4} for an interpretation of $J\mapsto J^\iinfty$ as a quadratic refinement of the $\cT_n$-valued intersection form on rooted trees (or on Whitney disks).



\subsection{Geometric meaning of the relations} \label{subsec:geometry-of-relations}
Both the odd and even order twisted tree groups contain 
the AS and IHX relations which apply to framed tree generators. 
We have already seen in Corollary~\ref{cor:IHX} the necessity of including the IHX relations  
in defining an invariant from $t(\cW)$ since IHX trees can be created locally.
Upon fixing the positive or negative corner convention (section~\ref{subsec:w-tower-tree-orientations}), the signs $\epsilon_p=\pm$ of the framed trees $\epsilon_p\cdot t_p$ in $t(\cW)$ only depend on the orientation of the underlying order $0$ surface modulo the antisymmetry relations (section~\ref{subsec:w-disk-orientation-choices-AS}).

In the odd order groups $\cT^{\iinfty}_{2j-1}$, which contain the obstructions to the existence of an order $2j$ twisted Whitney tower, the \emph{boundary-twist relations} correspond geometrically to the fact that 
performing a boundary twist (Figure~\ref{boundary-twist-fig}) on an order $j$ Whitney disk $W_{(i,J)}$ creates an order $2j-1$ intersection point
$p\in W_{(i,J)}\cap W_J$ with associated tree $t_p=i\,-\!\!\!\!\!-\!\!\!<^{\,J}_{\,J}$
and changes $\omega (W_{(i,J)})$ by $\pm1$. 
By Exercise~\ref{ex:create-clean-W-for-any-tree}, any number of clean framed $W_{(i,J)}$ can be created in any Whitney tower, so any number of $t_p=i\,-\!\!\!\!\!-\!\!\!\!<^{\,J}_{\,J}$
can be created by this construction. It follows from the obstruction theory (section~\ref{subsec:order-raising-proof-sketch}) that after arranging such trees into ``algebraically cancelling'' pairs, the corresponding unpaired intersections can be exchanged for ``geometrically cancelling'' intersections admitting Whitney disks.   
Since this can be done at the cost of only creating order $j$ twisted Whitney disks, which are allowed in an order $2j$ Whitney tower, the trees $i\,-\!\!\!\!\!-\!\!\!<^{\,J}_{\,J}$ do not represent obstructions.



\begin{figure}[ht!]
        \centerline{\includegraphics[scale=.5]{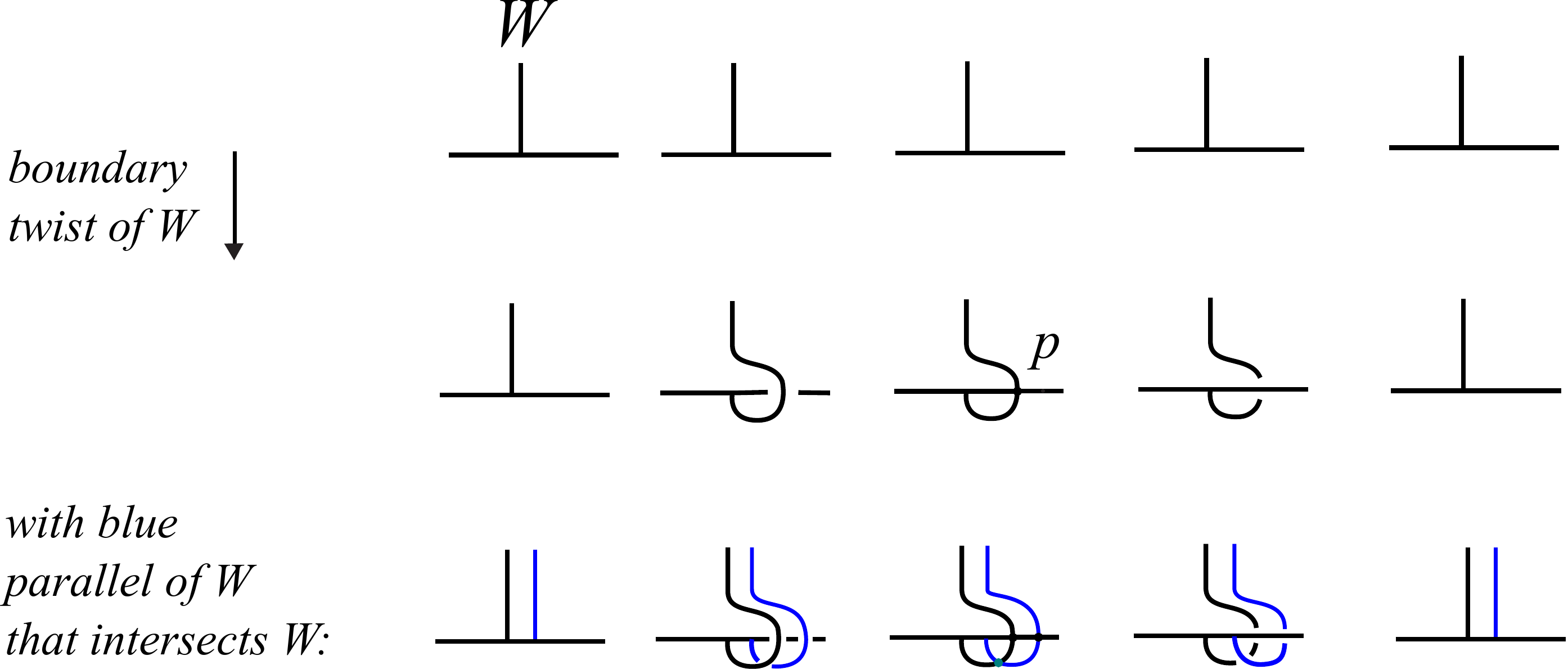}}
        \caption{`Side view' near a point in $\partial W$ of the boundary-twist operation on $W$ which changes $\omega(W)$ by $\pm 1$,
and creates a transverse intersection $p$ between $W$ and a sheet paired by $W$.}
        \label{boundary-twist-fig}
        \end{figure}

In the even order target groups $\cT^{\iinfty}_{2j}$:
The \emph{symmetry relation} corresponds to the fact that the twisting $\omega(W)$ (section~\ref{subsec:framed-w-disks}) is independent of the orientation of the Whitney disk $W$, with the minus sign denoting that the cyclic orderings at the trivalent vertices of $-J$ differ from those of $J$ at an odd number of vertices.
The \emph{twisted IHX relation} corresponds to the effect of performing a Whitney move in the presence of a twisted Whitney disk, as described in Lemma~\ref{lem:w-move-twistedIHX} and \cite[Lem.4.1]{CST1}. The \emph{interior-twist relation} corresponds to the fact that creating a 
$\pm1$ self-intersection
in a $W_J$ by a local cusp-homotopy \cite[Sec.1.6]{FQ} changes the twisting by $\mp 2$ (Figure~\ref{interior-twist-fig}). The result of such a cusp-homotopy is the same as the local cut-and-paste operation described in \cite[Sec.1.3]{FQ}.
For any $J$, a clean $W_J$ can be created by finger moves (Exercise~\ref{ex:create-clean-W-for-any-tree}), then a $\pm$-interior twist $W_J$ will change $t(\cW)$ by:
$$
\pm \langle J,J\rangle\quad\mp\quad 2\cdot J^\iinfty
$$

Thus, all the relations in $\cT^{\iinfty}_n$ can be realized by controlled modifications of Whitney towers altering their intersection forests, without changing the homotopy class of the underlying order $0$ surface.

\begin{figure}[h!]
        \centerline{\includegraphics[scale=.22]{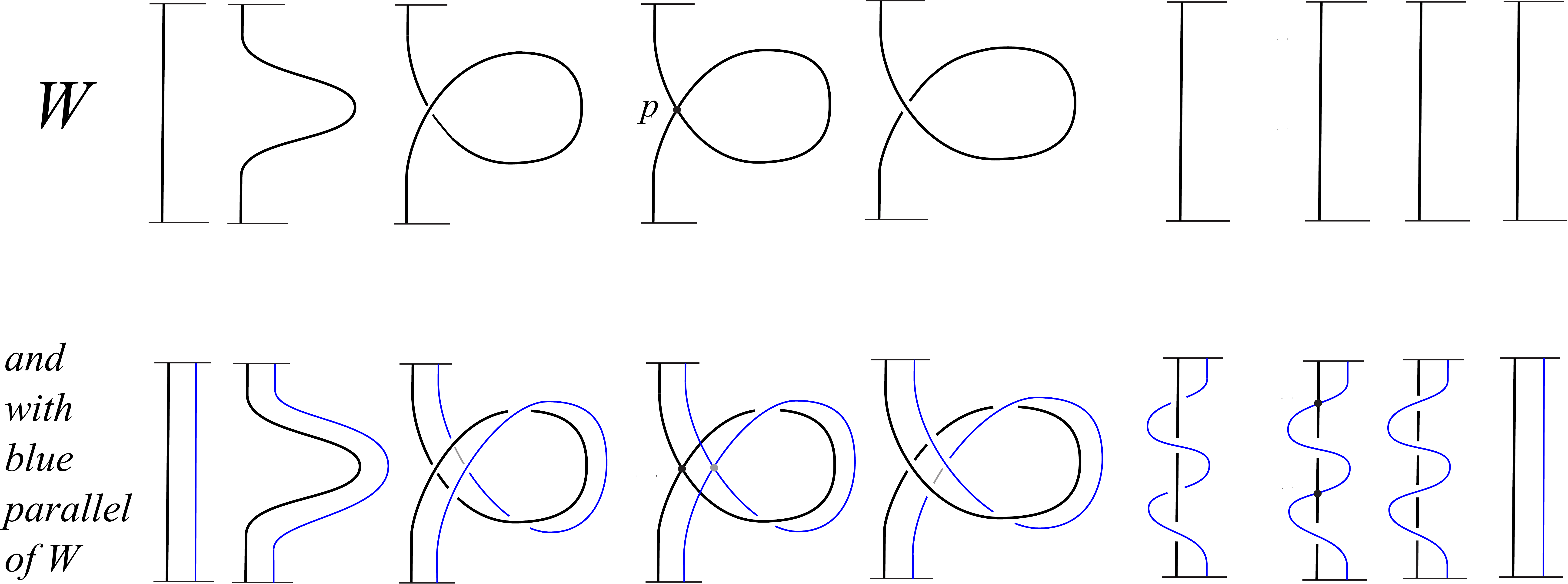}}
        \caption{After a $\pm$\emph{-interior twist} interior twist,
shown near an arc in $W$ that runs between the two sheets: $\omega(W)$ has changed by $\mp 2$, and
a new $p\in W\pitchfork W$ has been created.}
\label{interior-twist-fig}
        \end{figure}




%

\subsection{Intersection/obstruction theory for order $n$ twisted Whitney towers}\label{subsec:twisted-order-n-invariant}

\begin{defn}[Def.2.9 of \cite{CST1}]\label{def:twisted-order-n-invariant}
For an order $n$ twisted Whitney tower $\cW$, let $t_n(\cW)$ denote the sub-multiset $t_n(\cW)\subset t(\cW)$ consisting of all order $n$ framed trees and order $n/2$ twisted trees in $t(\cW)$. Define the \emph{order $n$ twisted intersection invariant}:
$$
\tau_n^\iinfty(\cW):=[t_n(\cW)]\in\cT_n^\iinfty
$$
\end{defn}
If $\cW$ is an order $n$ twisted Whitney tower, then the intersection forest $t(\cW)$ may apriori contain
framed trees of order $>n$ and $\iinfty$-trees of order $> n/2$, but in fact any such Whitney disks in $\cW$ can be deleted and/or modified yielding 
$t_n(\cW)=t(\cW)$ (see Exercise~\ref{ex:eliminate-higher-order-trees}).


%

\begin{thm}[Thm.1.9 of \cite{CST1}]\label{thm:twisted-order-raising}
A link $L\subset S^3$ bounds an order $n$ twisted $\cW\subset B^4$ with $\tau_n^\iinfty(\cW)=0\in\cT_n^\iinfty$ if and only if $L$ bounds an order $n+1$ twisted Whitney tower.
\end{thm}


\emph{Idea of proof:} For the ``only if'' direction, first realize relations by geometric constructions, as discussed in section~\ref{subsec:geometry-of-relations}, to arrange that all trees in $t_n(\cW)$ occur in oppositely-signed isomorphic pairs. Then use controlled maneuvers to arrange that all order $n$ intersections admit Whitney disks and all Whitney disks of order $\leq n/2$ are framed.
See section~\ref{subsec:order-raising-proof-sketch} for an outline of this proof, and \cite[Sec.4]{CST1} for details.
For the ``if'' direction see Exercise~\ref{ex:order-n+1-twisted-tower-has-vanishing-invariant}.

\subsection{Quick review of first non-vanishing Milnor invariants.}\label{subsec:intro-Milnor-review} 
Let $L\subset S^3$ be an $m$-component link with fundamental group $G=\pi_1(S^3\setminus L)$.
By \cite[Thm.4]{M2}, if the
longitudes of $L$ lie in the $(n+1)$-th term $G_{n+1}$ of the
lower central series of $G$, then a choice of meridians induces an isomorphism
$
\frac{G_{n+1}}{G_{n+2}}\cong\frac{F_{n+1}}{F_{n+2}}
$
where $F=F(m)$ is the free group on $\{x_1,x_2,\ldots,x_m\}$.

Let $\sL=\sL(m)$ denote the free Lie algebra (over $\Z$) on generators $\{X_1,X_2,\ldots,X_m\}$. It is $\N$-graded, $\sL=\oplus_n \sL_n$, where the degree~$n$ part $\sL_n$ is the
additive abelian group of length $n$ brackets, modulo Jacobi
identities and self-annihilation relations $[X,X]=0$.
The multiplicative abelian group $\frac{F_{n+1}}{F_{n+2}}$ of
length $n+1$ commutators is isomorphic to
$\sL_{n+1}$, with $x_i$ mapping to  $X_i$ and commutators mapping to Lie brackets.

In this setting, denote by $l_i$ the image of the $i$-th longitude in  $\sL_{n+1}$ under the above isomorphisms and
define the \emph{order $n$ Milnor invariant}
$\mu_n(L)$ by
$$
\mu_n(L):=\sum_{i=1}^m X_i \otimes l_i \in \sL_1 \otimes
\sL_{n+1}
$$
This definition of $\mu_n(L)$ is the first non-vanishing ``total'' Milnor invariant of order $n$, and corresponds to {\em all} Milnor invariants of \emph{length} $n+2$ in the original formulation of \cite{M1,M2}. 
The original $\bar{\mu}$-invariants are computed from the longitudes via the Magnus expansion as integers modulo indeterminacies coming from invariants of shorter length. Since we will only be concerned with first non-vanishing $\mu$-invariants, we do not use the ``bar'' notation $\bar{\mu}$.

It turns out that  $\mu_n(L)$ lies in the kernel $\sD_n$
of the bracket map $\sL_1 \otimes \sL_{n+1}\rightarrow \sL_{n+2}$ 
(eg.~by ``cyclic symmetry'' \cite{FT2}).


\subsection{The summation maps $\eta_n$}\label{subsec:eta-map}
The connection between $\tau^\iinfty_n(\cW)$ and $\mu_n(L)$ is via a homomorphism $\eta_n : \cT^\iinfty_n \to \sD_n$ which is most easily described by regarding rooted trees of order $n$ as elements of $\sL_{n+1}$ in the usual way:
For $v$ a univalent vertex of an order $n$ framed tree $t$, denote by $B_v(t)\in\sL_{n+1}$ the Lie bracket of generators $X_1,X_2,\ldots,X_m$ determined by the formal bracketing of indices
which is gotten by considering $v$ to be a root of $t$.

\begin{defn}\label{def:eta}
Denoting the label of a univalent vertex $v$ by $\ell(v)\in\{1,2,\ldots,m\}$, the
map $\eta_n:\cT^\iinfty_n\rightarrow \sL_1 \otimes \sL_{n+1}$
is defined on generators by
$$
\eta_n(t):=\sum_{v\in t} X_{\ell(v)}\otimes B_v(t)
\quad \, \,
\mbox{and}
\quad \, \,
\eta_n(J^\iinfty):= \frac{1}{2}\,\eta_n(\langle J,J \rangle)
$$
The first sum is over all univalent vertices $v$ of $t$, and the second expression lies in $\sL_1 \otimes \sL_{n+1}$ 
because the coefficients of $\eta_n(\langle J,J \rangle)$ are even. Here $J$ is a rooted tree of order $j$ for $n=2j$.
\end{defn}

Examples of $\eta_n$ for $n=1,2$:
\[
 \begin{array}{lll}
\eta_1\!\left( {\scriptstyle 1}-\!\!\!\!\!-\!\!\!<^{\,\,3}_{\,\,2}\,\right)  &= \quad  X_1\otimes\,-\!\!\!\!\!-\!\!\!<^{\,\,3}_{\,\,2}\,\,\ + \quad X_2\otimes \,{\scriptstyle 1}\!-\!\!\!\!\!-\!\!\!<^{\,\,3}_{\,\,}\,\,\  + \quad X_3\otimes \, {\scriptstyle 1}\!-\!\!\!\!\!-\!\!\!<^{\,\,}_{\,\,2}\,  
\\
&= \quad  X_1\otimes\,[X_2,X_3]+ X_2\otimes [X_3,X_1]  +  X_3\otimes [X_1,X_2].
\end{array}
\]
And,
 \[
\begin{array}{llc}
\eta_2\!\left( {\scriptstyle \iinfty}-\!\!\!\!\!-\!\!\!<^{\,\,2}_{\,\,1}\,\right)  &= \frac{1}{2}\,\eta_2\!\left(^{\,\,1}_{\,\,2}>\!\!\!-\!\!\!\!\!-\!\!\!<^{\,\,2}_{\,\,1}\right) \\ 
&= X_1\,\otimes _{\,\,2}\!>\!\!\!-\!\!\!\!\!-\!\!\!<^{\,\,2}_{\,\,1} \,+\,  X_2\,\otimes  ^{\,\,1}\!>\!\!\!-\!\!\!\!\!-\!\!\!<^{\,\,2}_{\,\,1}   
\\
&= X_1\otimes\,[X_2,[X_1,X_2]]+ X_2\otimes [[X_1,X_2],X_1]. 
\end{array}
\]


%
%

The image of $\eta_n$ is equal to the bracket kernel $\sD_n<\sL_1 \otimes \sL_{n+1}$, by \cite[Lem.32]{CST2}.

\begin{thm}[\cite{CST1}]\label{thm:Milnor invariant}
If $L$ bounds a twisted Whitney tower $\cW$ of order $n$, then the order $q$ Milnor invariants $\mu_q(L)$ vanish for $q<n$, and
\[
\mu_n(L) = \eta_n \circ\tau^\iinfty_n(\cW) \in \sD_n
\]
\end{thm}

\emph{Idea of proof:} The existence of the order $n$ twisted $\cW$ implies, via Dwyer's Theorem, that the inclusion $S^3\setminus L\to B^4\setminus \cW$ induces an isomorphism on the $(n+1)$th lower central quotients of $\pi_1$, so the longitudes of $L$ can be computed in $B^4\setminus \cW$.  
It turns out that the corresponding iterated commutators are displayed exactly according to $\eta_n \circ\tau^\iinfty_n(\cW)$, with the key observation being that a meridian to a Whitney disk is a commutator of meridians to the sheets paired by the Whitney disk. See \cite[Sec.4]{CST2} or \cite[Thm.3.1]{Cha2}.

\subsection{The order $n$ twisted Whitney tower filtration on links}

Recall that a link $L\subset S^3=\partial B^4$ \emph{bounds} an order $n$ twisted Whitney tower $\cW$ if
$\cW\subset B^4$ is an order $n$ twisted Whitney tower whose order $0$ surfaces are immersed disks bounded by the components of $L$.

We say that links $L_0$ and $L_1$ in $S^3$ are \emph{twisted Whitney tower concordant of order $n$} if $L_0\subset S^3\times\{0\}$ and $L_1\subset S^3\times\{1\}$ cobound a collection $A\imra S^3\times[0,1]$ of immersed annuli such that $A$ supports an order $n$ twisted Whitney tower (with $A$ inducing the reversed orientation on $L_1$).


Denote by $\W^\iinfty_n$ the set of links in $S^3$ bounding order $n$ twisted Whitney towers in $B^4$
modulo the equivalence relation of order $n+1$ twisted Whitney tower concordance.

The twisted ``order-raising'' Theorem~\ref{thm:twisted-order-raising} implies the following essential criterion for links to represent 
equal elements in $\W^\iinfty_n$:

\begin{cor}[{\cite[Cor.3.3]{CST1}}]\label{cor:tau=w-concordance}
Links $L_0$ and $L_1$ represent the same element of $\W^\iinfty_n$
if and only if there exist order $n$ twisted Whitney towers $\cW_i$ in $B^4$ with $\partial\cW_i=L_i$ and $\tau^\iinfty_n(\cW_0)=\tau^\iinfty_n(\cW_1)\in\cT^\iinfty_n$.
\end{cor}

\begin{proof}
If $L_0$ and $L_1$ are equal in $\W^\iinfty_n$ then they cobound $A$ supporting an order $n+1$ twisted Whitney tower $\cV$ in $S^3\times I$, and any order $n$ twisted Whitney tower $\cW_1$ in $B^4$ bounded by $L_1$ can be extended by $\cV$ to 
form an order $n$ twisted Whitney tower $\cW_0$ in $B^4$ bounded by $L_0$, with 
$\tau^\iinfty_n(\cW_0)=\tau^\iinfty_n(\cW_1)\in\cT^\iinfty_n$
since $\tau^\iinfty_n(\cV)$ vanishes. 

Conversely, suppose that $L_0$ and $L_1$ bound order $n$ twisted Whitney towers $\cW_0$ and $\cW_1$ in $4$--balls $B_0^4$ and $B_1^4$, with 
$\tau^\iinfty_n(\cW_0)=\tau^\iinfty_n(\cW_1)$. Then constructing $S^3\times I$
as the connected sum $B_0^4\# B_1^4$ (along balls in the complements of $\cW_0$ and $\cW_1$), and tubing together the corresponding order zero disks of $\cW_0$ and $\cW_1$, and taking the union of the Whitney disks in $\cW_0$ and 
$\cW_1$, yields a collection $A$ of properly immersed annuli connecting $L_0$ and $L_1$ and supporting an order $n$ twisted Whitney tower $\cV$.  Since the orientation of the ambient $4$--manifold has been reversed for one of the original Whitney towers, say $\cW_1$, which results in a global sign change for 
$\tau^\iinfty_n(\cW_1)$, it follows that $\cV$ has vanishing order $n$ intersection invariant:
$$
\tau^\iinfty_n(\cV)=\tau^\iinfty_n(\cW_0)-\tau^\iinfty_n(\cW_1)=\tau^\iinfty_n(\cW_0)-\tau^\iinfty_n(\cW_0)=0\in\cT^\iinfty_n
$$
So by Theorem~\ref{thm:twisted-order-raising}, $A$ is homotopic (rel $\partial$) to $A'$ supporting
an order $n+1$ twisted Whitney tower, and hence $L_0$ and $L_1$ are equal in $\W^\iinfty_n$.
\end{proof} 

The \emph{band sum} $L\#_\beta L'\subset S^3$
of oriented $m$-component links $L$ and $L'$ along bands $\beta$ is defined as follows: Form $S^3$ as the connected sum of $3$--spheres containing $L$ and $L'$ along balls in the link complements.  Let
$\beta$ be a collection of disjointly embedded oriented bands joining like-indexed link components such that the band orientations are compatible with the link orientations. Take the usual connected sum of each pair of components along the corresponding band. Although it is well-known that the concordance class of $L\#_\beta L'$ depends in general on $\beta$, it turns out that the image of 
$L\#_\beta L'$ in $\W^\iinfty_n$ does not depend on $\beta$:

\begin{lem}[\cite{CST1}]\label{lem:link-sum-well-defined}   
For links $L$ and $L'$ representing elements of $\W^\iinfty_n$, any band sum $L\#_\beta L'$ represents an element of 
$\W^\iinfty_n$ which only depends on the equivalence classes of $L$ and $L'$ in $\W^\iinfty_n$.
\end{lem}
See Lemma~3.6 of \cite{CST1} or Exercise~\ref{ex:link-band-sum-well-defined} for a proof of Lemma~\ref{lem:link-sum-well-defined} using the following:


\begin{lem}[\cite{CST1}]\label{lem:exists-tower-sum}
If $L$ and $L'$ bound order $n$ twisted Whitney towers $\cW$ and $\cW'$ in $B^4$, then for any bands $\beta$ there exists an order $n$ twisted
Whitney tower $\cW^\#\subset B^4$ bounded by $L\#_\beta L'$, such that 
$t(\cW^\#)=t(\cW)+ t(\cW')$.
\end{lem}
See Lemma~3.7 of \cite{CST1}, or Exercise~\ref{ex:exists-tower-for-band-sum}, for a proof.


\subsection{Definition of the realization maps}\label{subsec:realization-maps}
We define ``realization'' maps $R^\iinfty_n:\cT^\iinfty_n\to \W^\iinfty_n$ for all $n$ as follows: Given any group element $g\in\cT^\iinfty_n$, by Lemma~\ref{lem:realization-of-geometric-trees} just below there exists an $m$-component link $L\subset S^3$ bounding
an order $n$ twisted Whitney tower $\cW\subset B^4$ such that $\tau^\iinfty_n(\cW)=g\in\cT^\iinfty_n$.
Define $R^\iinfty_n(g)$ to be the class determined by $L$ in $\W^\iinfty_n$.
This is well-defined (does not depend on the choice of such $L$) by Corollary~\ref{cor:tau=w-concordance}. 

\begin{lem}[\cite{CST1}]\label{lem:realization-of-geometric-trees}
For any multiset $\sum_p\ \epsilon_p \cdot  t_p \,\, + \sum_J\ \omega (W_J) \cdot  J^\iinfty$ there exists a link $L$ bounding a Whitney tower $\cW$
with intersection forest $t(\cW)= \sum_p\ \epsilon_p \cdot  t_p \,\, + \sum_J \ \omega (W_J) \cdot  J^\iinfty$. 
\end{lem}
The proof of Lemma~\ref{lem:realization-of-geometric-trees} follows Tim Cochran's technique of ``Bing-doubling along a tree'' to realize individual trees, and then uses band sums of links via Lemma~\ref{lem:exists-tower-sum}
to realize sums (multiset unions) of trees
(see Lemma~3.8 of \cite{CST1} or Exercise~\ref{ex:realize-trees}). 

For instance, to compute the image of the framed tree 
$t=\langle (1,2),(3,1)\rangle=\,^{1}_{2}>\!\!\!\!-\!\!\!-\!\!\!\!\!-\!\!\!\!<^{\,\,1}_{\,\,3}$ under $R^\iinfty_2$, one Bing-doubles each component of the Hopf link and then bands together two components as in Figure~\ref{fig:Cochran-R2-map} to get
a link $L$ bounding a Whitney tower $\cW$ with $\tau^\iinfty_2(\cW)=t$.
Any framed tree can be realized analogously by applying iterated Bing-doubling to the Hopf link to get the desired tree shape, and then banding components to get the desired univalent labels.

\begin{figure}[h]
\centerline{\includegraphics[scale=.35]{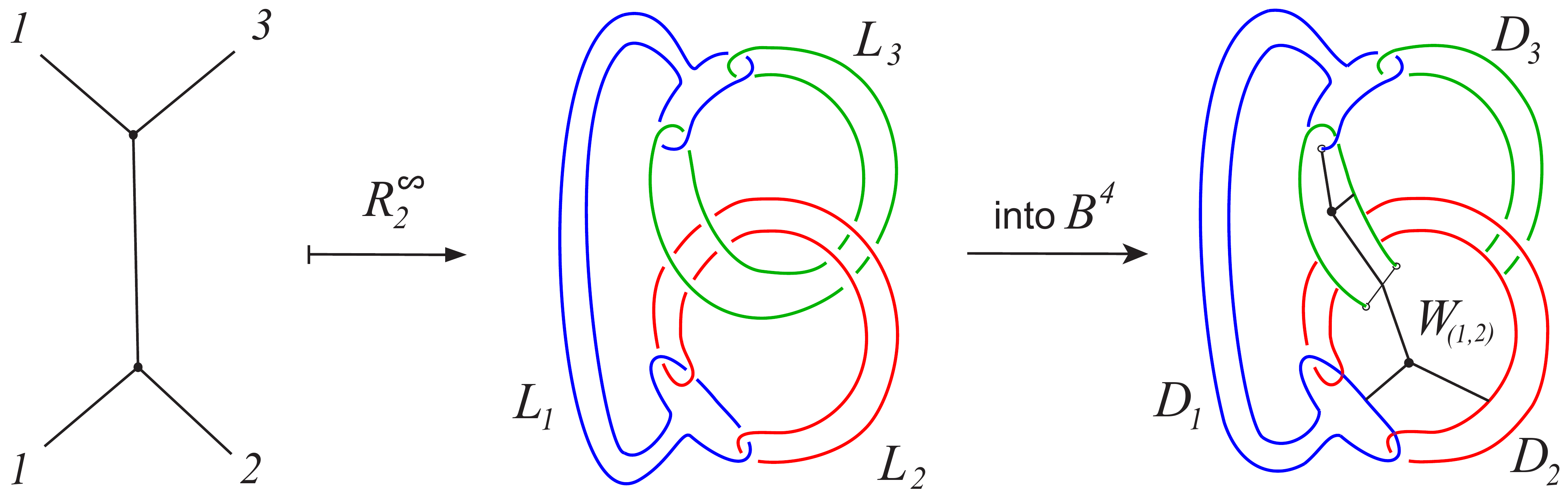}}
\caption{}\label{fig:Cochran-R2-map}
\end{figure}


To compute the image of the twisted tree $(2,1)^\iinfty=\,^{2}_{1}>\!\!\!\!-\!\!\!-\!\!\!\!\!-\,\iinfty$ under $R^\iinfty_2$, one applies a twisted Bing-double to the unknot as in Figure~\ref{fig:twisted-tree-R2-map} to get
$L$ bounding a Whitney tower $\cW$ with $\tau^\iinfty_2(\cW)=(2,1)^\iinfty$.
Any twisted tree $J^\iinfty$ with twisting coefficient $\omega\in\Z$ can be realized similarly by starting with a single $\omega$-twisted Bing-doubling of the unknot and then applying iterated (untwisted) Bing-doublings to get the desired tree shape, and then banding components to get the desired univalent labels.

\begin{figure}[h]
\centerline{\includegraphics[scale=.35]{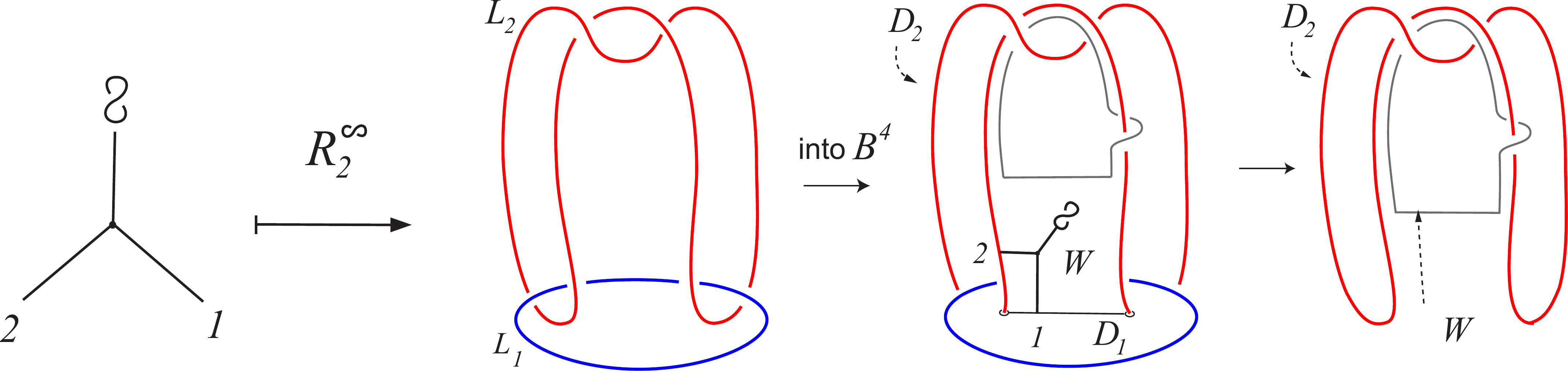}}
\caption{}\label{fig:twisted-tree-R2-map}
\end{figure}

In combination with Theorem~\ref{thm:Milnor invariant},
the following result will play a key role in classifying order $n$ twisted Whitney towers in the $4$--ball:
\begin{thm}[\cite{CST1}]\label{thm:realization-surjection}
The realization maps are epimorphisms $R^\iinfty_n:\cT^\iinfty_n\sra \W^\iinfty_n$.
\end{thm}


\begin{proof} From Lemma~\ref{lem:link-sum-well-defined} the band sum of links gives a well-defined operation in $\W^\iinfty_n$ which is clearly associative and commutative, with the $m$-component unlink representing an identity element.  The realization maps are homomorphisms by Lemma~\ref{lem:exists-tower-sum} and surjectivity is proven as follows: Given any link $L$ representing an element of $\W^\iinfty_n$, choose a twisted Whitney tower $\cW$ of order $n$ with boundary $L$ and compute $\tau:=\tau^\iinfty_n(\cW)$. Then take $L':= R^\iinfty_n(\tau)$, a link that's in the image of $R^\iinfty_n$ and for which we know a Whitney tower $\cW'$ with boundary $L'$ and $\tau^\iinfty(\cW') = \tau$.
By Corollary~\ref{cor:tau=w-concordance} it follows that $L$ and $L'$ represent the same element in $\W^\iinfty_n$.
\end{proof}


%
%

%
%
\subsection{Computing the order $n$ twisted Whitney tower filtration}\label{subsec:computing-twisted-filtration}

From Theorem~\ref{thm:Milnor invariant} and Theorem~\ref{thm:realization-surjection} we have the following commutative triangle diagram~\ref{diagram:triangle} of epimorphisms:
\begin{equation}\label{diagram:triangle}
    \xymatrix{
\cT^\iinfty_n \ar@{->>}[r]^{R_n^\iinfty} \ar@{->>}[rd]_{\eta_n} & \W^\iinfty_n \ar@{->>}[d]^{\mu_n}\\
& \sD_{n}
}\tag{$\bigtriangledown$}
\end{equation}
  

The following algebraic result is a consequence of the proof \cite{CST3} of 
a combinatorial conjecture of J. Levine:
\begin{thm}[\cite{CST1}] \label{thm:twisted-three-quarters-classification} 
The maps $\eta_n:\cT^\iinfty_n \to \sD_n$ are isomorphisms for $n\equiv 0,1,3\,\mod 4$.
\end{thm}
%
From this theorem we immediately get:
\begin{cor}

For $n\equiv 0,1,3\,\mod 4$, the Milnor invariants
$\mu_n\colon \W^\iinfty_n\to \sD_n$  and the twisted realization maps 
$R^\iinfty_n : \cT^\iinfty_n \to\W^\iinfty_n$ are \emph{isomorphisms}.

\end{cor}

By \cite{O}, $\sD_n$ is a free abelian group of known rank for all $n$,  
so we have a complete computation of $\W^\iinfty_n\cong\sD_n\cong\cT^\iinfty_n$ in three quarters of the cases.

\begin{rem}\label{rem:tau-well-defined-in-3-4-of-cases}
It also follows that in these orders the order $n$ twisted intersection invariant
$\tau^\iinfty_n(\cW)\in\cT^\iinfty_n$ only depends on the concordance class of 
$L=\partial\cW$. In particular, $\tau^\iinfty_n(\cW)\in\cT^\iinfty_n$ is independent of the choice of Whitney tower $\cW$ on the (unique) homotopy class rel $\partial$ of the order $0$ immersed disks bounded by $L$.
\end{rem}

%

Towards understanding the remaining cases $n\equiv 2\mod 4$, 
we have another consequence of \cite{CST3} which was derived as Corollary~6.6 of \cite{CST1}:
\begin{prop}[\cite{CST1}]\label{prop:kerEta4j-2}
For rooted trees $J$ of order $j-1$, the map $1\otimes J \mapsto \iinfty-\!\!\!\!-\!\!-\!\!\!<^{\,J}_{\,J}\,\,\in\cT^\iinfty_{4j-2}$ 
induces an isomorphism: 
$$\Z_2 \otimes \sL_j\cong\Ker\{\eta_{4j-2}:\cT_{4j-2}^\iinfty\to\sD_{4j-2}\}$$
\end{prop}
%
%

%

So $\Z_2 \otimes \sL_j$ is also an upper bound on
$\sK^\iinfty_{4j-2}:=\Ker\{\mu_{4j-2}: \W^\iinfty_{4j-2} \sra \sD_{4j-2}\}$.
Denoting by $\alpha^\iinfty_j: \Z_2 \otimes \sL_j
\sra \sK^\iinfty_{4j-2}$ the epimorphism induced by $R^\iinfty_{4j-2}$, we have the following extension of the above triangle diagram~\ref{diagram:triangle} in these orders:

$$
\xymatrix{
\langle 1\otimes J\rangle\ar@{<->}[dr]\ar@{=}[r]&\Z_2 \otimes \sL_j\ar@{->>}[rr]^{\alpha^\iinfty_j}\ar@{>->}[dr]&&\sK^\iinfty_{4j-2}\ar@{>->}[d]\\
& \langle \iinfty-\!\!\!\!-\!\!-\!\!\!<^{\,J}_{\,J} \rangle\ar@{>->}[r]&\cT^\iinfty_{4j-2} \ar@{->>}[r]^{R_{4j-2}^\iinfty} \ar@{->>}[rd]_{\eta_{4j-2}}&\W^\iinfty_{4j-2} \ar@{->>}[d]^{\mu_{4j-2}}\\&&& \sD_{4j-2}&
}
$$

By inverting the induced isomorphism $\overline{\alpha^\iinfty_{j}}$ on $(\mathbb Z_2\otimes {\sL}_{j})/\Ker \alpha^\iinfty_{j}$, we get the following definition:
\begin{defn}\label{def:Arf-j}  
The \emph{higher-order Arf invariants} $\Arf_{j}$ are defined by
$$
\Arf_{j}:=(\overline{\alpha^\iinfty_{j}})^{-1}:\sK^\iinfty_{4j-2}\to(\mathbb Z_2\otimes {\sf L}_{j})/\Ker \alpha^\iinfty_{j}
$$
\end{defn}
As a corollary we get the computation of $\W^\iinfty_n$ for all $n$, and a characterization of links bounding order $n$ twisted Whitney towers:
\begin{cor}[\cite{CST1}]\label{cor:mu-arf-classify-twisted} 
The abelian groups ${\sf W}^\iinfty_{n}$ are classified by Milnor invariants $\mu_n$ and, in addition, higher-order Arf invariants $\Arf_j$ for $n=4j-2$.

In particular, a link bounds an order $n+1$ twisted $\cW$ if and only if its Milnor invariants and higher-order Arf invariants vanish up to order $n$.
\end{cor}



\subsection{The higher-order Arf invariant Conjecture}\label{subsec:higher-arf-conjecture}
%
%
%
%


In the case $j=1$ of Definition~\ref{def:Arf-j}, $\Ker \alpha^\iinfty_{1}$ is trivial and we have $\Arf_{1}:\sK^\iinfty_{2} \overset{\cong}{\to} (\Z_2\otimes \sL_1) \cong (\Z_2)^m$,
given by classical Arf invariants of the link components \cite[Lem.10]{CST2}. 

The \emph{higher-order Arf invariant conjecture} states that $\Ker \alpha^\iinfty_{j}$ is trivial for all $j$:
\begin{conj}\label{conj:higher-order-arf}
 $\Arf_{j}:\sK^\iinfty_{4j-2}\to\mathbb Z_2\otimes {\sf L}_{j}$ are isomorphisms for all $j$.
\end{conj}

Assuming this conjecture the classification in these orders would be described by the following diagram:

$$
\xymatrix{
\Z_2 \otimes \sL_j\ar@{>->}[dr]&&\sK^\iinfty_{4j-2}\ar@{>-->>}[ll]_{\Arf_{j}}\ar@{>->}[d]\\
& \cT^\iinfty_{4j-2} \ar@{->>}[r]^{R_{4j-2}^\iinfty} \ar@{->>}[rd]_{\eta_{4j-2}}&\W^\iinfty_{4j-2} \ar@{->>}[d]^{\mu_{4j-2}}\\& &\sD_{4j-2}
}
$$

Conjecture~\ref{conj:higher-order-arf} would imply $\W_n^\iinfty\xrightarrow{\tau_n^\iinfty}\cT^\iinfty_n$ is an isomorphism for all $n$.
That is, the intersection invariants $\tau_n^\iinfty$ taking values in $\cT^\iinfty_n$ are independent of Whitney tower choices and characterize links bounding order $n$ twisted Whitney towers.


\begin{figure}[h]
\centerline{\includegraphics[scale=.4]{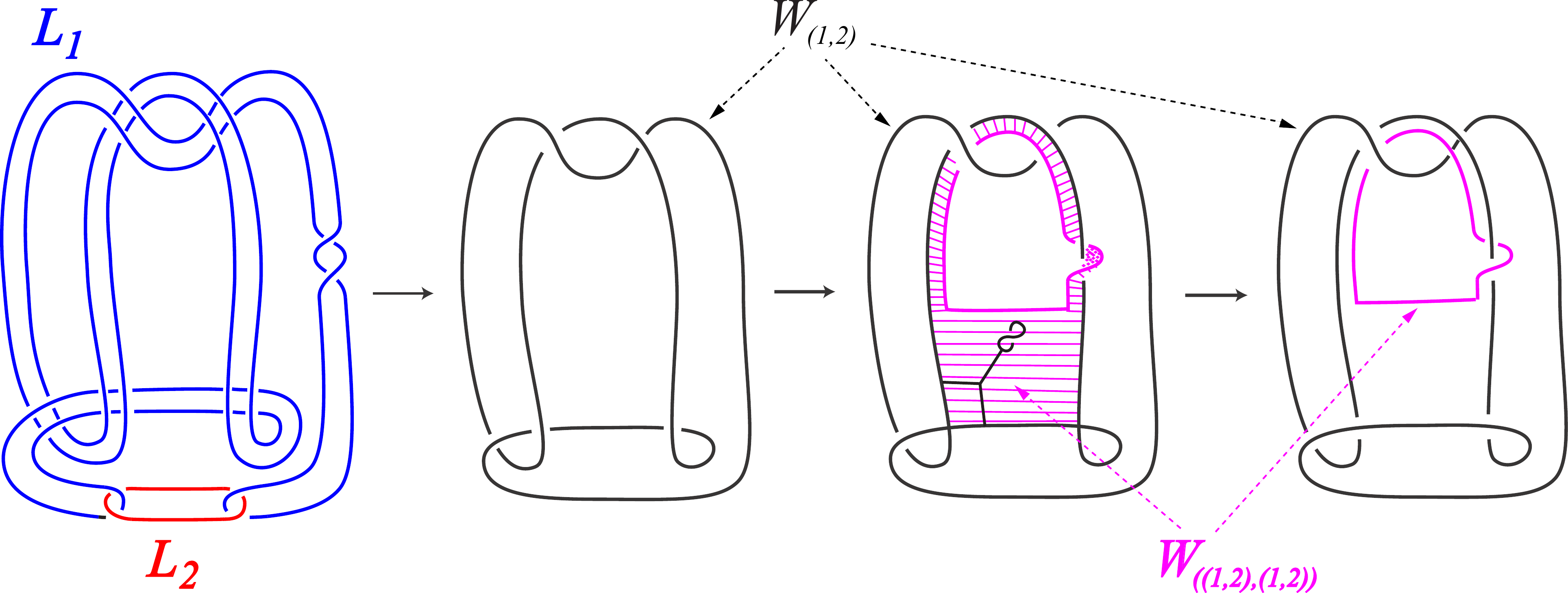}}
\caption{The Bing double of the Figure-8 knot bounds an order~6 twisted Whitney tower $D_1\cup D_2\cup W_{(1,2)}\cup W_{((1,2),(1,2))}\subset B^4$, with $\omega(W_{((1,2),(1,2))})=1$.} \label{fig:bing-fig-8}
\end{figure}
\subsection{Problems, Questions and re-formulations}\label{subsec:arf-questions-reformulations}

As just mentioned in section~\ref{subsec:higher-arf-conjecture}, Conjecture~\ref{conj:higher-order-arf} is true for $k=1$, with $\Arf_1$ given by the classical Arf invariants of the link components \cite{CST2}, but it remains an open problem whether $\Arf_k$ is non-trivial for any $k>1$. 
The links $R^\iinfty_{4k-2}( \iinfty-\!\!\!\!\!-\!\!\!<^{\,J}_{\,J}\,\,)$ realizing the image of $\Arf_{k}$ can all be constructed as internal band sums of iterated Bing doubles of knots having non-trivial classical Arf invariant (\cite[Lem.13]{CST2}), see Figure~\ref{fig:bing-fig-8}. Such links are known not to be \emph{slice}
by work of J.C. Cha \cite{Cha}, providing evidence in support of Conjecture~\ref{conj:higher-order-arf}.

So the fundamental open problem in this setting is: 

\begin{prob}
Determine the precise image of $\Arf_j\leq\Z_2\otimes\sL_j$ for $j\geq 2$.
\end{prob}

The following specific lowest order open question is already interesting: 
\begin{question}\label{question:does-bing-fig-8-bound-order-7-twisted}
Does the Bing double of the Figure-8 knot $R^\iinfty_{6}( \iinfty-\!\!\!\!\!-\!\!\!<^{\,(1,2)}_{\,(1,2)}\,\,)\in\sW^\iinfty_6$ bound an order~7 twisted Whitney tower?
\end{question}
This importance of this lowest order open question is magnified by the following fact \cite[Prop.14]{CST2}:
\begin{prop}[\cite{CST2}]
\emph{If} the Bing double of the Figure-8 knot \emph{does} bound an order~7 twisted Whitney tower, then $\Arf_j$ are \emph{trivial} for all $j\geq 2$.
\end{prop}

Conjecture~\ref{conj:higher-order-arf} predicts a negative answer to Question~\ref{question:does-bing-fig-8-bound-order-7-twisted}, which can be can be phrased as the following restriction on the possible twisted Whitney towers on $2$--spheres in the $4$--ball:
\begin{conj}
There does not exist $A:S^2\cup S^2\imra B^4$ supporting $\cW$ with $$t(\cW)=\iinfty-\!\!\!\!\!-\!\!\!<^{\,(1,2)}_{\,(1,2)} \quad\mbox{(possibly + higher-order trees)}.
$$
\end{conj}

\subsection{Higher-order Arf invariant Conjecture
and Finite Type invariants}\label{subsec:finite-type-arf}

Habegger and Massbaum \cite{HM} have shown that Milnor invariants are the only \emph{rational} finite type concordance invariants of (string) links.
The classical Arf invariant of a knot is known to be determined by the first non-trivial finite type isotopy invariant, and as stated above, $\Arf_1$ corresponds to the classical Arf invariants of the link components. So it is natural to ask:

\begin{question}\label{question:arf-finite-type}
Are the $\Arf_{j}$ for $j>1$ also determined by finite type isotopy invariants?
\end{question}

For the first interesting case in this setting, a negative answer to Question~\ref{question:does-bing-fig-8-bound-order-7-twisted} can be formulated as:
\begin{conj}\label{conj:framed-arf-degree-6}
The sum of trees in Figure~\ref{fig:arf-framed-trees} represents a non-trivial finite type concordance invariant of 2-component string links
(first-non-vanishing, $\Z/2\Z$-coefficients).
\end{conj}

\begin{figure}[h]
\centerline{\includegraphics[scale=.45]{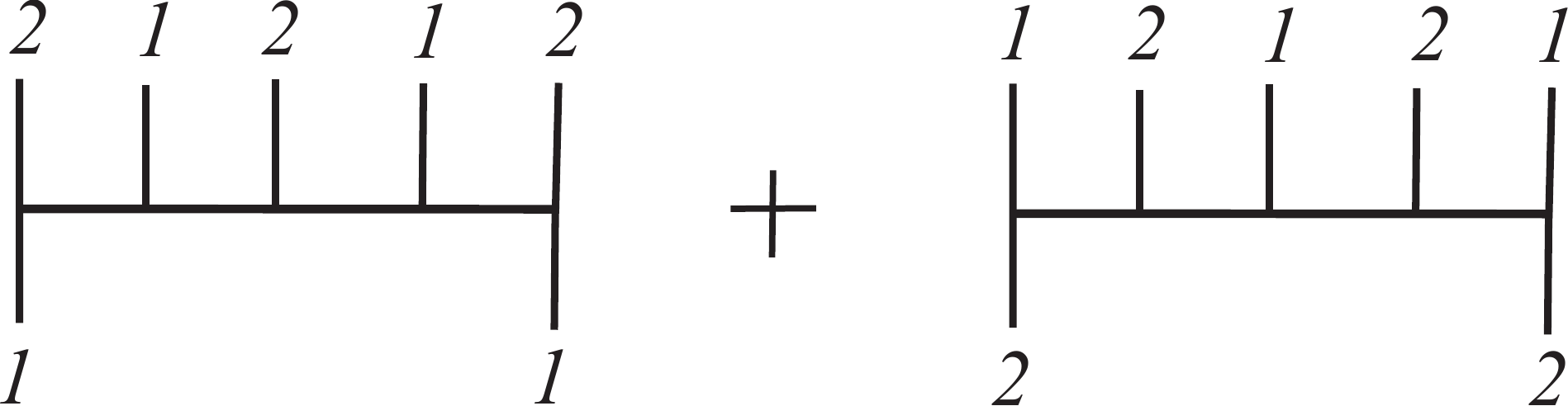}}
\caption{}\label{fig:arf-framed-trees}
\end{figure}

The invariant of Conjecture~\ref{conj:framed-arf-degree-6} would be finite type \emph{degree}~6, 
and all finite type concordance invariants of string links in degrees $\leq 5$ have been characterized by J-B. Meilhan and A. Yasuhara in \cite{MY}. 
So this conjecture appears to lie at the frontier of current understanding in this setting.

Degree~6 finite type invariants are related to order~5 intersection invariants in the setting of \emph{framed} Whitney towers \cite{CST,ST2},
and the sum of trees in Figure~\ref{fig:arf-framed-trees} (modulo higher-order trees) can be gotten by applying twisted IHX moves and boundary twists to a Whitney tower $\cW$ with $t(\cW)=\iinfty-\!\!\!\!\!-\!\!\!\!<^{\,(1,2)}_{\,(1,2)}$ bounded by any link $R^\iinfty_{6}( \iinfty-\!\!\!\!\!-\!\!\!\!<^{\,(1,2)}_{\,(1,2)}\,\,)\in\sW^\iinfty_6$, see Exercise~\ref{ex:framed-arf-trees-from-twisted}. This illustrates how the higher-order Arf invariants shift down one order in the \emph{framed} order $n$ Whitney tower filtration \cite{CST1}. Conjecture~\ref{conj:framed-arf-degree-6} corresponds in this setting to: 
\begin{conj}\label{conj:bing-arf-doesnt-bound-framed-order-6}
The Bing double of any knot with non-trivial classical Arf invariant does not bound an order $6$ \emph{framed} Whitney tower.
\end{conj}

By Exercise~\ref{ex:no-spheres-with-framed-arf-trees}, Conjecture~\ref{conj:bing-arf-doesnt-bound-framed-order-6} can also be phrased as a restriction on Whitney towers supported by 2--spheres in 4--space:
\begin{conj}\label{conj:no-spheres-with-framed-arf-trees}
There does not exist $A:S^2\cup S^2\imra B^4$ supporting $\cW$ with $t(\cW)$ equal to the trees in Figure~\ref{fig:arf-framed-trees}, possibly plus higher-order trees.
\end{conj}
We remark that each of the two trees in Figure~\ref{fig:arf-framed-trees} \emph{individually} represents a non-trivial \emph{higher-order Sato--Levine invariant} which is determined by a non-trivial order $6$ $\mu$-invariant \cite{CST1}, so neither of these trees can appear by itself (plus higher-order trees) as $t(\cW)$ for $\cW$ on 2-spheres in $B^4$ (compare Exercise~\ref{ex:t(W)-on-S2-in-B4-restricted}).

Recent work of Danica Kosanovi\'{c} \cite{Danica1} includes progress towards showing that the \emph{embedding calculus} of Goodwillie--Weiss \cite{GW,Weiss} 
determines universal finite type invariants for knots over the integers, as conjectured in \cite{BCSS}. 
Kosanovi\'{c} works with certain $2$-complexes called \emph{capped gropes} which are very closely related to Whitney towers \cite{CST1,S1}, and one could hope that, via a generalization of \cite{Danica1} to links, the homotopy-theoretic techniques of the embedding calculus might be able to detect the higher-order Arf invariants:
\begin{question}\label{question:embedding-calculus}
Can the $\Arf_{j}$ be detected by the embedding calculus?
\end{question}

\subsection{Higher-order Arf invariant conjecture and transfinite Milnor invariants}\label{subsec:arf-transfinite-mu}
In \cite{Cha-Orr} Jae Choon Cha and Kent Orr defined certain transfinite invariants of $3$--manifolds which can be interpreted as providing a generalization of Milnor invariants. Their invariants include certain finite-length $\theta_k$-invariants which can take values in finite abelian groups, and in \cite[Sec.14(9)]{Cha-Orr} they ask:
\begin{question}\label{question:transfinite}
Are the higher-order Arf invariants related to the $\theta_k$-invariants?
\end{question}


%
%
%
%
%

\subsection{Section~\ref{sec:twisted-order-n-classification-arf-conj} Exercises}

\subsubsection{Exercise}\label{ex:eliminate-higher-order-trees}
If $\cW$ is an order $n$ twisted Whitney tower, then the intersection forest $t(\cW)$ may contain
framed trees of order $>n$ and $\iinfty$-trees of order $>n/2$ in addition to those representing $\tau_n^\iinfty(\cW)\in\cT_n^\iinfty$.
Show that by deleting Whitney disks of order $>n$, boundary-twisting, and pushing-down intersections (Figure~\ref{fig:W-disk-int-and-push-down}), these higher-order
trees in $t(\cW)$ can be eliminated while preserving the twisted order~$n$ of the resulting Whitney tower.
(See discussion in \cite{CST1}, section~4.1 `Notation and Conventions'.)

%

\subsubsection{Exercise}\label{ex:order-n+1-twisted-tower-has-vanishing-invariant}
If $\cW$ an order $n+1$ twisted Whitney tower, observe that by definition 
$\cW$ is also an order $n$ twisted Whitney tower, and check that $\tau^\iinfty_n(\cW)=0\in\cT^\iinfty_n$.

\subsubsection{Exercise}\label{ex:exists-tower-for-band-sum}
Show that if $L$ and $L'$ bound order $n$ twisted Whitney towers $\cW$ and $\cW'$ in $B^4$, then for any $\beta$ there exists an order $n$ twisted
Whitney tower $\cW^\#\subset B^4$ bounded by $L\#_\beta L'$, such that 
$t(\cW^\#)=t(\cW)+ t(\cW')$.
(Lemma~3.7 of \cite{CST1}.)

\subsubsection{Exercise}\label{ex:link-band-sum-well-defined} 
Use the previous exercise to show that for links $L$ and $L'$ representing elements of $\W^\iinfty_n$, any band sum $L\#_\beta L'$ represents an element of 
$\W^\iinfty_n$ which only depends on the equivalence classes of $L$ and $L'$ in $\W^\iinfty_n$. (Lemma~3.6 of \cite{CST1}.)

\subsubsection{Exercise}\label{ex:realize-trees}
Given any framed tree $\langle I,J\rangle$, construct a link $L\subset S^3$ bounding $\cW\subset B^4$ with
$t(\cW)=\langle I,J\rangle$. HINT: Apply Bing-doubling as needed to the Hopf link, as in Figure~\ref{fig:Cochran-R2-map}.
(See Lemma~3.8 of \cite{CST1}.)

\subsubsection{Exercise}\label{ex:realize-twisted-trees}
Given any integer $n$ and any rooted tree $J$ of positive order, construct a link $L\subset S^3$ bounding $\cW\subset B^4$
with $t(\cW)=n\cdot J^\iinfty$.
HINT: Start with the $n$-twisted Bing double of the unknot (see Figure~\ref{fig:twisted-tree-R2-map} for the case $n=1$), then apply untwisted iterated Bing-doubling as needed.
(See Lemma~3.8 of \cite{CST1}.)

\subsubsection{Exercise}\label{ex:tau-of-bing-hopf-not-zero} 
From Exercise~\ref{exercise:bing-hopf-on-boundary}, the Bing-double of the Hopf link bounds an order $2$ twisted Whitney tower $\cW$ as in Figure~\ref{higher-order-intersection-color-and-with-tree}.
Show that the Bing-double of the Hopf link does not bound an order $3$ twisted Whitney tower by checking that the order $2$ Milnor invariant $\mu_2=\eta_2\circ\tau^\iinfty_2(\cW)\in\sL_1 \otimes\sL_{3}$ is non-zero. Conclude that the Bing-double of the Hopf link also does not bound an order $3$ framed Whitney tower (since for any $n$ an order $n$ framed Whitney tower is also an order $n$ twisted Whitney tower by definition).

\subsubsection{Exercise}\label{ex:compute-mu-previous-order-2-examples} 
Compute the order $2$ Milnor invariants $\mu_2=\eta_2\circ\tau^\iinfty_2(\cW)$ for the links in Figures~\ref{fig:Cochran-R2-map}~and~\ref{fig:twisted-tree-R2-map}.

\subsubsection{Exercise}\label{ex:order-zero-twisted}
Order $0$ twisted tree groups and Milnor invariants:

Check from the definitions in section~\ref{subsec:intro-Milnor-review} that for $i\neq j$ the coefficient of 
$X_i\otimes X_j$ in $\mu_0(L)$ is the linking number of $L_i$ and $L_j$, which via the well-known computation of linking numbers by counting signed intersections
between the properly immersed disks $D_i$ and $D_j$ bounded by $L_i$ and $L_j$ is also equal to the coefficient 
of $X_i\otimes X_j$ in $\eta_0(\tau_0^\iinfty(\cW))$.

Although Milnor invariants are not usually defined for knots, for framed links it is natural to consider the 
framing $f_i$ of $L_i$ as an order $0$ (length $2$) integer Milnor invariant, and the coefficient of $X_i\otimes X_i$ in $\mu_0(L)$ is exactly $f_i$ when this framing is used to determine the $i$th longitude. To see that the coefficient 
in $\eta_0(\tau_0^\iinfty(\cW))$ of
$X_i\otimes X_i$ is also equal to $f_i$, let $d_i$ denote the number of positive self
intersections of $D_i$ minus the number of negative self-intersections  of $D_i$. Then
the relative Euler number of $D_i$ with respect to the framing $f_i$ on $L_i=\partial D_i$ is equal to $f_i-2d_i$
(see e.g.~Figure~19 of \cite{CST1} and accompanying discussion), and the terms of 
$\tau_0^\iinfty(\cW)$ which contribute via $\eta_0$ to the coefficient of $X_i\otimes X_i$
are exactly 
$
(d_i)\cdot \,i\,-\!\!\!-\!\!\!-\,i+(f_i-2d_i) \cdot\iinfty \,-\!\!\!-\!\!\!-\,i,
$
which get sent by $\eta_0$ to
$(f_i)\cdot X_i\otimes X_i$.

\subsubsection{Exercise}\label{ex:t(W)-on-S2-in-B4-restricted}
Show that if $g\in\cT_n^\iinfty$ is such that $\eta_n(g)\neq 0\in\sD_n$, 
then there does not exist any twisted Whitney tower tower $\cW$ on $2$--spheres in $B^4$ such that $t(\cW)$ represents $g\in\cT_n^\iinfty$.
(HINT: Otherwise tubing the spheres into disks bounded by an unlink would ``create'' non-trivial $\mu$-invariants.)

 \subsubsection{Exercise}\label{ex:one-term-of-arf-trees-on-spheres} 
Check that $\eta_6((((1,2),1),2)^\iinfty)\neq 0\in \sL_1 \otimes\sL_{6}$.
Conclude from Exercise~\ref{ex:t(W)-on-S2-in-B4-restricted} that there does not exist $A:S^2\cup S^2\imra B^4$ supporting a Whitney tower $\cW$ with $t(\cW)=(((1,2),1),2)^\iinfty$, perhaps plus trees of order $>6$.

\subsubsection{Exercise}\label{ex:framed-arf-trees-from-twisted}
Use the Whitney move twisted IHX relation of Lemma~\ref{lem:w-move-twistedIHX} and boundary-twisting to get the two trees in Figure~\ref{fig:arf-framed-trees} (plus higher-order trees) from the single tree $((1,2),(1,2))^\iinfty$.

%

%
\subsubsection{Exercise}
Check that $\eta_2\circ\tau^\iinfty_2(\cW)$ vanishes for the Bing double of Figure-8 knot.

\subsubsection{Exercise}\label{ex:bing-double-knot-with-non-trivial-arf}
If a knot $K$ bounds a Whitney tower $\cW$ with $t(\cW)=(1,1)^\iinfty$, show that the Bing-double of $K$ bounds a Whitney tower $\cV$ with
$t(\cV)=((1,2),(1,2))^\iinfty$.

\subsubsection{Exercise}\label{ex:no-spheres-with-framed-arf-trees}
Check that Conjecture~\ref{conj:bing-arf-doesnt-bound-framed-order-6} and Conjecture~\ref{conj:no-spheres-with-framed-arf-trees} are equivalent.

\subsubsection{Exercise}
Show that if there exists a pair of 2-spheres in $S^4$ supporting a Whitney tower $\cW$ such that $t(\cW)=((1,2),(1,2))^\iinfty$, then all higher order Arf invariants $\Arf_{j}$ for $j\geq 2$ vanish on all links. (See \cite[Prop.14]{CST2}.)



\section{Whitney towers on $2$-spheres in $4$-manifolds}\label{sec:2-spheres-in-4-manifolds}

In this section we consider Whitney towers on $A=A_1,A_2,\ldots,A_m\imra X$, where each $A_i$ is a $2$--sphere, and $X$ is a $4$--manifold. All manifolds are oriented and based. Each $A_i$ is equipped with a \emph{whisker}, ie.~an arc running between the basepoint of $A_i$ and the basepoint of $X$.

We will assume that $A$ is generically immersed, and each sphere $A_i$ has the same number of positive self-intersections as negative self-intersections. This can always be arranged by performing \emph{cusp homotopies} (\cite[Chap.1]{FQ}, same as the interior twist operation of Figure~\ref{interior-twist-fig}) and is a natural assumption in the setting of Whitney towers since it is satisfied as soon as $A$ supports a (framed or twisted) Whitney tower of positive order. Regular homotopy classes of such immersions are in one-to-one correspondence with homotopy classes, see eg.~\cite[Thm.1.2]{PRT}. Up to isotopy, a regular homotopy is a sequence of finger moves and Whitney moves, 
so homotopy invariance of invariants defined from the intersection forests of Whitney towers can be checked combinatorially (eg.~section~\ref{subsubsec:htpy-invariance-tau1}).

The classical intersection invariant $\lambda(A_i,A_j)$ and self-intersection invariant $\mu(A_i)$ are recalled in section~\ref{subsec:classical-mu-lambda}, and we formulate these invariants in the language of Whitney towers as order~0 invariants which give the complete obstruction to $A$ supporting an order~1 framed Whitney tower in section~\ref{subsec:order-0-lambda-mu}.
After explaining how the order~0 intersection pairing gives the complete obstruction to ``pulling apart'' a pair of spheres, ie.~making them disjoint by a homotopy (section~\ref{subsec:pulling-apart-pairs}), edge decorations in $\pi_1X$ for order~1 trees are introduced in section~\ref{subsec:decorated-order-1-trees} and the order~1 non-repeating invariant generalizing the classical intersection pairing is defined in the setting of pulling apart triples of components of $A$ (section~\ref{subsec:lambda1} and Theorem~\ref{thm:lambda1-vanishes}).
In section~\ref{subsec:higher-order-lambda} we briefly touch on possible higher-order non-repeating invariants and a general obstruction theory for pulling apart multiple components.

The order~1 generalization of the classical self-intersection invariant is defined in section~\ref{subsec:tau-1}. Here twistings on Whitney disks are relevant, and the vanishing of the invariant is equivalent to the existence of an order~2 framed Whitney tower on $A$  (Theorem~\ref{thm:tau1-vanishes}).

Throughout this section,
relevant open questions, problems and conjectures are included in the discussions.
Proofs of the main Theorems~~\ref{thm:lambda1-vanishes} and~\ref{thm:tau1-vanishes} are given in Section~\ref{sec:appendix}.


\subsection{Classical intersection form}\label{subsec:classical-mu-lambda}

The first obstructions to making the components of $A$ pairwise disjointly embedded by a homotopy
are the \emph{intersection invariants} $\lambda (A_i,A_j)$, which take values in the integral fundamental group ring $\Z[\pi_1X]$,
and the \emph{self-intersection invariants} $\mu (A_i)$,
which take values in a quotient of $\Z[\pi_1X]$.
These invariants are defined as
follows. Associate to
each transverse intersection point $p\in
A_i\pitchfork A_j$ the element
$g_p\in\pi_1X$ determined by a loop through $A_i$ and $A_j$ which
changes sheets at $p$.
More precisely, such a \emph{sheet-changing loop} runs along the whisker for $A_i$, then along any choice of path in $A_i$ to $p$, then along any choice of path in $A_j$ to the whisker of $A_j$, then along the whisker for $A_j$. Sheet-changing loops are required to avoid all singularities of $A_i$
and $A_j$ other than $p$, so $g_p$ does not depend on the choices of paths in $A_i$ and $A_j$ between $p$ and their whiskers because the domains of $A_i$ and $A_j$ are simply connected. 
See Figure~\ref{fig:double-point-elements}.
\begin{figure}[h]
\centerline{\includegraphics[scale=.4]{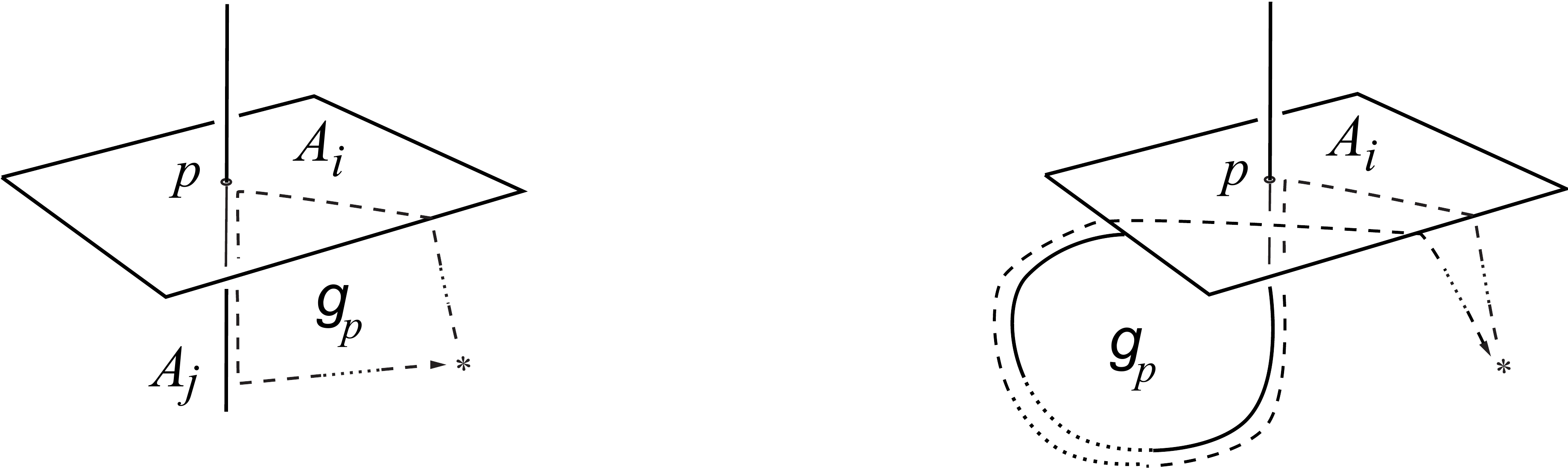}}
\caption{$g_p\in\pi_1X$ for $p\in A_i\pitchfork A_j$ with $i\neq j$ (left), and for $i=j$ (right).}
         \label{fig:double-point-elements}
\end{figure}

Summing over all such intersection points,
with the usual notion of the sign $\epsilon_p\in\{+,-\}$, defines:
$$
\lambda(A_i,A_j):=\sum_{p\in A_i\pitchfork A_j}\epsilon_p\cdot
g_p\in\Z[\pi_1X]
$$
and
$$
\mu(A_i):=\sum_{p\in A_i\pitchfork A_i}\epsilon_p\cdot
g_p\in\frac{\Z[\pi_1X]}{\{g-g^{-1}\}}.
$$

The relations in $\Z[\pi_1X]$ induced by $g= g^{-1}\in\pi_1X$ account for choices of orientations on the
sheet-changing loops through the self-intersections of $A_i$. 

Up to isotopy, a regular homotopy of $A$ is a sequence of finger moves and Whitney moves, each of which either introduces or eliminates oppositely-signed intersections having the same group element.
So $\lambda$ and $\mu$ are invariant under regular homotopy, eg.~\cite[Sec.4]{PRT}.

%



\subsection{Order $0$ invariants}\label{subsec:order-0-lambda-mu}

Here we slightly reformulate the classical invariants $\lambda$ and $\mu$ in the language of Whitney towers, with an eye towards higher-order generalizations.


The union $A=A_1\cup A_2\cup\cdots\cup A_m\subset X$ is by Definition~\ref{def:order-n-framed-W-tower} a framed Whitney tower of order~0.
To each order~$0$ intersection $p\in A_i\pitchfork A_j$ is associated the order~0 tree $t_p=\langle i,j\rangle$, as in section~\ref{subsec:trees-for-w-disks-and-ints}, and we think of $t_p\subset A_i\cup A_k$ as an embedded sheet-changing edge near $p$ with one univalent vertex in $A_i$ and the other univalent vertex in $A_k$. 
For each such $t_p$, choose a path in $A_i$ from the basepoint of $A_i$ to the $i$-labelled univalent vertex of $t_p$,
and a path in $A_j$ from the $j$-labelled univalent vertex of $t_p$ to the basepoint of $A_j$.
The union of $t_p$ (oriented from $i$ to $j$) together with these paths and the whiskers on $A_i$ and $A_j$,
defines a sheet-changing loop representing $g_p\in\pi_1X$, just as in section~\ref{subsec:classical-mu-lambda}.
We call the tree $t_p$, together with the label $g_p$ on its edge and an orientation of the edge from $i$ to $j$, a \emph{decorated tree} for $p$.  


Let $\cT_0$ denote the quotient of the free abelian group on order~0 decorated trees by the following OR \emph{orientation} relation:
$$
\mbox{ OR: }\quad\quad i\,\longrightarrow\!\!\!\!\overset{\,\,g_p}{-}\!\!\!-\!\!\!-\,j=i\,-\!\!\!-\!\!\!-\!\!\!\!\overset{g^{-1}_p}{-}\!\!\!\!\!\longleftarrow\,j
$$

Now the classical invariants $\lambda_0:=\lambda$ and $\mu_0:=\mu$ can be expressed as a single order $0$ invariant $\tau_0(A)$ represented by the intersection forest $t(A)$:

$$
\tau_0(A):=\sum_{p\in A_i\pitchfork A_j} \epsilon_p\cdot\,i\,\longrightarrow\!\!\!\!\overset{\,\,g_p}{-}\!\!\!-\!\!\!-\,j\in\cT_0
$$

%

In the language of Whitney towers we have:
\begin{thm}\label{thm:tau-0-vanishes}
$\tau_0(A)=0$ if and only if $A$ supports an order $1$ framed Whitney tower.
\end{thm}
Recall from Definition~\ref{def:order-n-framed-W-tower} that $A$ supporting an order~1 framed Whitney tower $\cW$ means that all intersections of $A$ are paired by framed order~1 Whitney disks in $\cW$.
The idea of the ``only if'' direction of the proof is that all intersections can be arranged in oppositely-signed pairs having the same group element, after orienting sheet-changing loops appropriately. Then null-homotopic Whitney circles can be constructed from the pairs of sheet-changing loops. These Whitney circles bound immersed Whitney disks which can be made framed by boundary-twisting (Figure~\ref{boundary-twist-fig}). 
See Exercise~\ref{ex:vanishing-tau-0-gives-order-1-w-tower}. For the ``if'' direction see Exercise~\ref{ex:order-1-w-tower-implies-tau-0-vanishes}.   
See also \cite[Lem.4.3]{PRT}.

%
%
%

It can always be arranged that the Whitney disks in an order~1 framed Whitney tower are disjointly embedded (Exercise~\ref{ex:push-down}), but they will in general have interior intersections with $A$ which
obstruct using them to guide Whitney moves homotoping $A$ to an embedding.

We remark that for half-dimensional spheres $A:S_1^d,S_2^d,\ldots,S_m^d\imra X^{2d}$ in a $2d$-dimensional manifold with $d>2$ the invariants analogous to $\lambda$ and $\mu$ give the complete obstruction to embedding $A$ in $X^{2d}$, because the interiors of Whitney disks on $A$ will have interiors disjoint from $A$ by general position. 

Notice that $\tau_0(A)$ splits into a direct sum of \emph{non-repeating} and \emph{repeating} invariants $\tau_0(A)=\sum_{i\neq j}\lambda_0(A_i,A_j)\oplus\sum_i\tau_0(A_i)$.
Before generalizing the full order~0 invariant $\tau_0(A)$ to an order~1 invariant $\tau_1(A)$, we will first consider the intermediate problem of generalizing the non-repeating summands $\lambda_0(A):=\sum_{i\neq j}\lambda_0(A_i,A_j)$ to order~1 invariants $\lambda_1(A_i,A_j,A_k)$ for triples with distinct $i,j,k$, and discuss the relationship to ``pulling apart'' triples of components (making them pairwise disjoint by a homotopy).


\subsection{Pulling apart pairs of spheres}\label{subsec:pulling-apart-pairs}

\begin{figure}[h]
\centerline{\includegraphics[scale=.4]{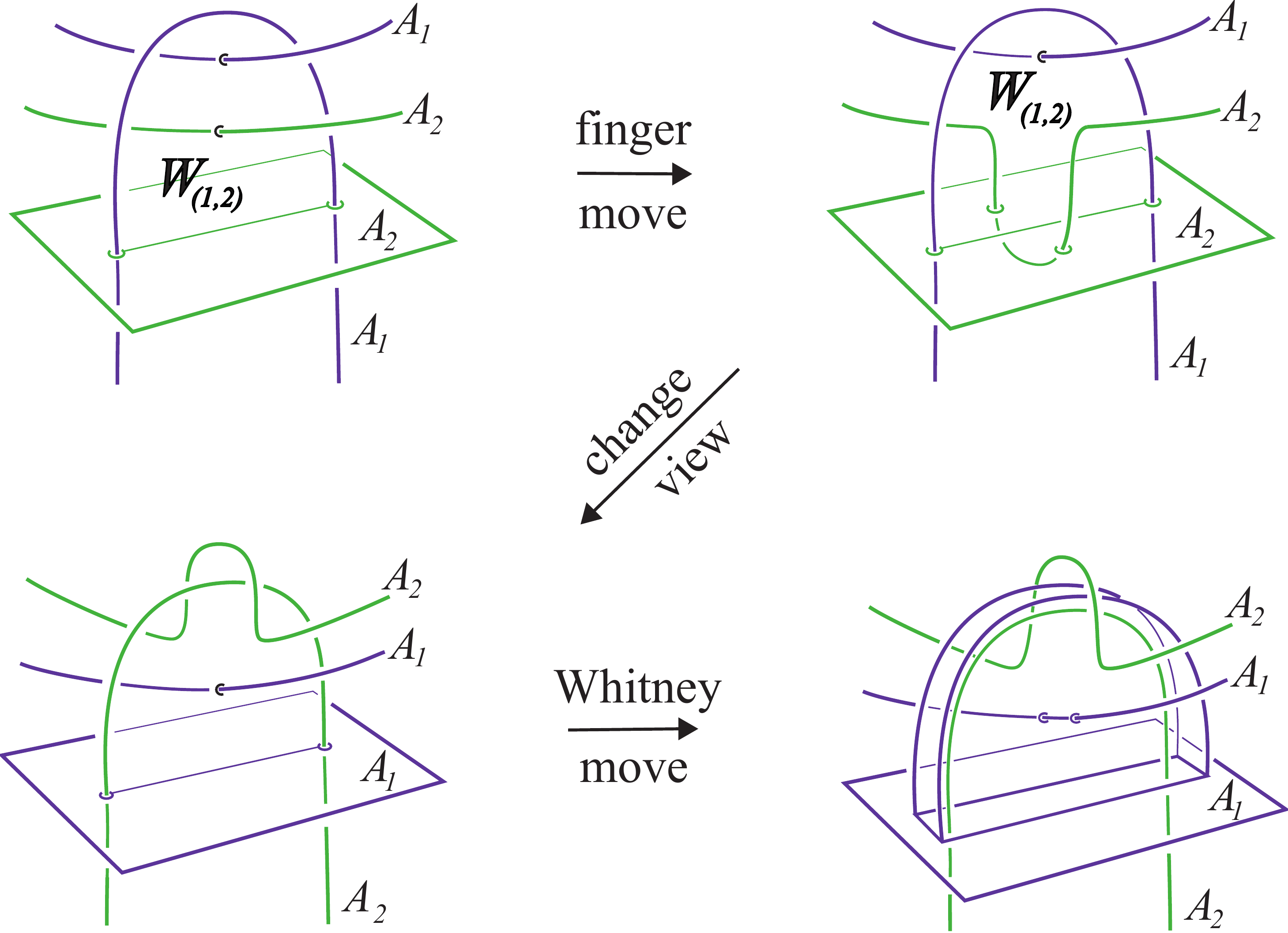}}
\caption{}
         \label{fig:pull-apart-A1-A2}
\end{figure}

The vanishing of $\lambda_0(A_1,A_2)=0\in\Z[\pi_1X]$ implies the existence of Whitney disks pairing $A_1\pitchfork A_2$ (Exercise~\ref{ex:lambda-0-vanishes-gives-w-disks}), and the union of $A_1\cup A_2$ together with a collection of such Whitney disks forms an \emph{order $1$ non-repeating Whitney tower }$\cW$, cf. Definition~\ref{def:non-rep-tower}. As an illustration of the proof of Theorem~\ref{thm:pull-apart} in this easiest case, Figure~\ref{fig:pull-apart-A1-A2} shows how the existence of such a $\cW$ leads to a homotopy that makes $A_1$ and $A_2$ disjoint:  Finger moves as in the top of Figure~\ref{fig:pull-apart-A1-A2} make $A_2$ disjoint from the interiors of all the Whitney disks, at the cost of only creating self-intersections in $A_2$. 
Now doing all the Whitney moves on $A_1$ makes $A_1\cap A_2=\emptyset$ at the cost of only creating self-intersections in $A_1$ as in the bottom of Figure~\ref{fig:pull-apart-A1-A2}. Any self-intersections and intersections among the Whitney disks will only lead to the creation of more self-intersections in $A_1$ upon doing the Whitney moves, hence such intersections have been suppressed from view in Figure~\ref{fig:pull-apart-A1-A2}.

The procedure of Figure~\ref{fig:pull-apart-A1-A2} appears to fail in the presence of a third sphere $A_3$, as it is not clear how to eliminate any $W_{(1,2)}\cap A_3$ without creating more intersections between $A_3$ and $A_1$ or $A_2$.
In order to generalize $\lambda_0(A_1,A_2)$ to an order~1 invariant $\lambda_1(A_1,A_2,A_3)$ which ``counts'' such order~1 intersections $W_{(i,j)}\cap A_k$ and gives the complete obstruction to pulling apart triples of 2-spheres, we next introduce edge decorations for order~1 trees.

\subsection{Decorated trees for order~1 intersections}\label{subsec:decorated-order-1-trees}
Let $W_{(i,j)}$ be a Whitney disk for a pair of intersections in $A_i\pitchfork A_j$.
To each intersection $p\in W_{(i,j)}\pitchfork A_k$ we associate a \emph{decorated} order~1 tree $t_p$ which is gotten by labeling each edge of the usual tree $\langle(i,j),k\rangle$ from section~\ref{subsec:trees-for-w-disks-and-ints} by an element of $\pi_1X$ as in Figure~\ref{fig:Y-w-disk-labeled-with-tree}. This requires a choice of whisker running from the trivalent vertex of each tree to the basepoint of $X$, and the group elements are determined by loops formed using sheet-changing edges oriented towards the Whitney disk followed by the chosen whisker emanating from the trivalent vertex. More precisely, taking $t_p$ to be embedded in $A_i\cup A_j\cup A_k\cup W_{(i,j)}$ as in Figure~\ref{fig:Y-w-disk-labeled-with-tree}, each edge is a sheet-changing embedded arc which can be oriented to change sheets into the Whitney disk, and each univalent vertex can be connected by a path in its order~0 surface to the surface basepoint. Together with the whiskers on $A_i$, $A_j$, $A_k$ and the trivalent vertex we get three oriented loops determining the corresponding edge decorations $a,b,c\in\pi_1X$. 


\begin{figure}[h]
\centerline{\includegraphics[scale=.5]{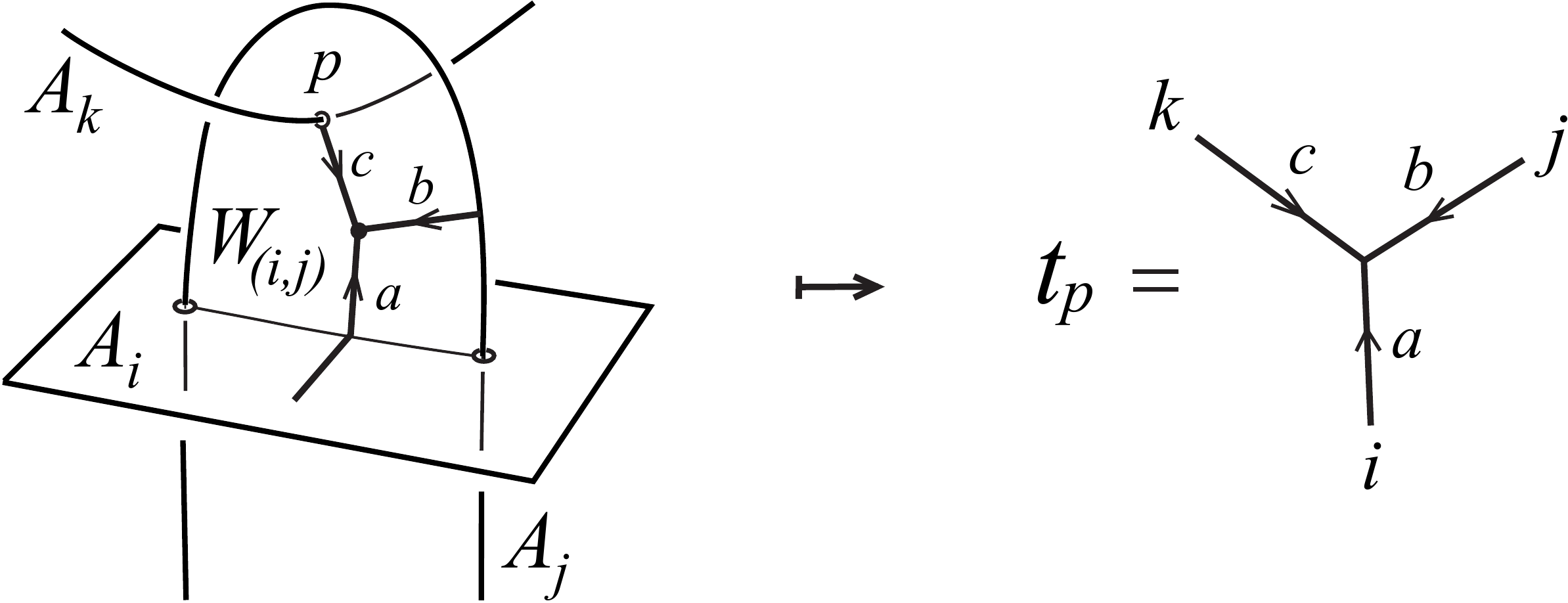}}
\caption{Edge decorations $a,b,c\in\pi_1 X$.}
         \label{fig:Y-w-disk-labeled-with-tree}
\end{figure}

\subsection{The order 1 non-repeating intersection invariant $\lambda_1$}\label{subsec:lambda1}

Consider a triple of immersed spheres $A=A_1,A_2,A_3\imra X$ supporting an \emph{order~1 non-repeating Whitney tower} $\cW$, ie.~all intersections $A_i\pitchfork A_j$ for distinct $i,j\in\{1,2,3\}$ are paired by Whitney disks in $\cW$.
The existence of such an order~1 non-repeating Whitney tower $\cW$ is equivalent to $A$ having pairwise vanishing 
$\lambda_0(A_i,A_j)=0\in\Z[\pi_1X]$ for $i\neq j$ (Exercise~\ref{ex:pairwise-lambda-0-vanishes-gives-order-1-nonrep-tower}), which we can succinctly express as $\lambda_0(A)=0$.

Recall from section~\ref{subsubsec:non-repeating-w-towers} that if $t_p$ associated to $p\in\cW$ 
has univalent vertices labelled distinctly by $1$, $2$ and $3$, then $t_p$ is called a \emph{non-repeating tree} and $p$ is called a \emph{non-repeating intersection}.

Denote by $\Lambda_1:=\Lambda_1(\pi_1X)$ the quotient of the free abelian group on order~1 decorated non-repeating trees
by the AS \emph{antisymmetry} and HOL \emph{holonomy} relations of Figure~\ref{fig:Relations-inward-arrows}. 

\begin{figure}[h]
\centerline{\includegraphics[scale=.7]{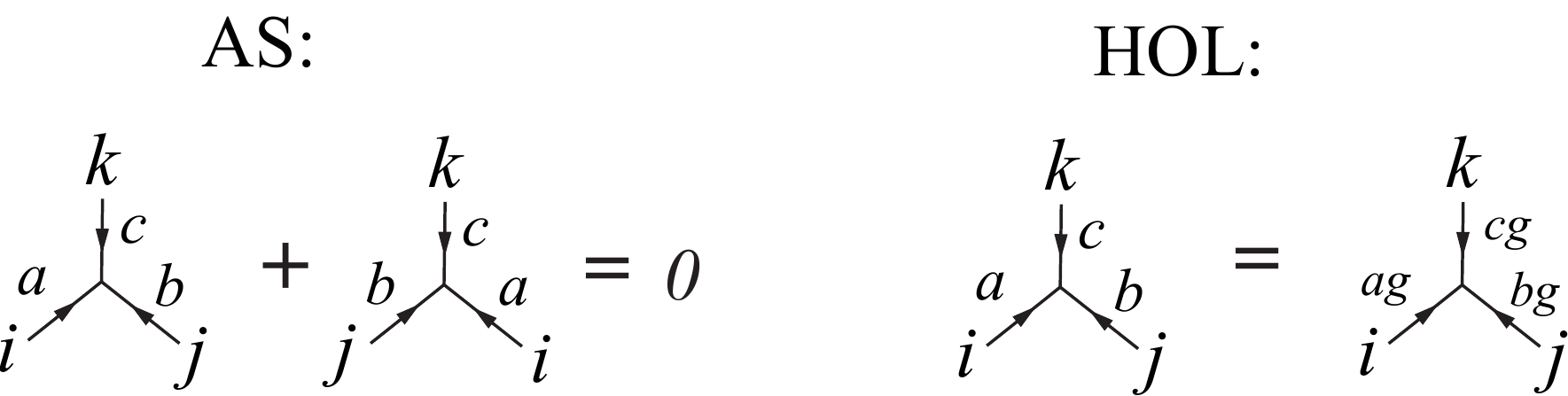}}
\caption{The \emph{Antisymmetry} and \emph{Holonomy} relations; $a,b,c,g\in\pi_1X$.}
         \label{fig:Relations-inward-arrows}
\end{figure}

For $A\imra X$ 
supporting an order~1 non-repeating Whitney tower $\cW$, we
define the \emph{order~1 non-repeating intersection invariant} $\lambda_1(A)$ in the quotient of $\Lambda_1$ by the INT \emph{intersection} relations shown in Figure~\ref{fig:INT-relation}:
$$
\lambda_1(A):=\sum \epsilon_p\cdot t_p \in\Lambda_1/\mbox{INT}
$$
where the sum is over all order~1 non-repeating intersections $p$ in $\cW$.

\begin{figure}[h]
\centerline{\includegraphics[scale=.4]{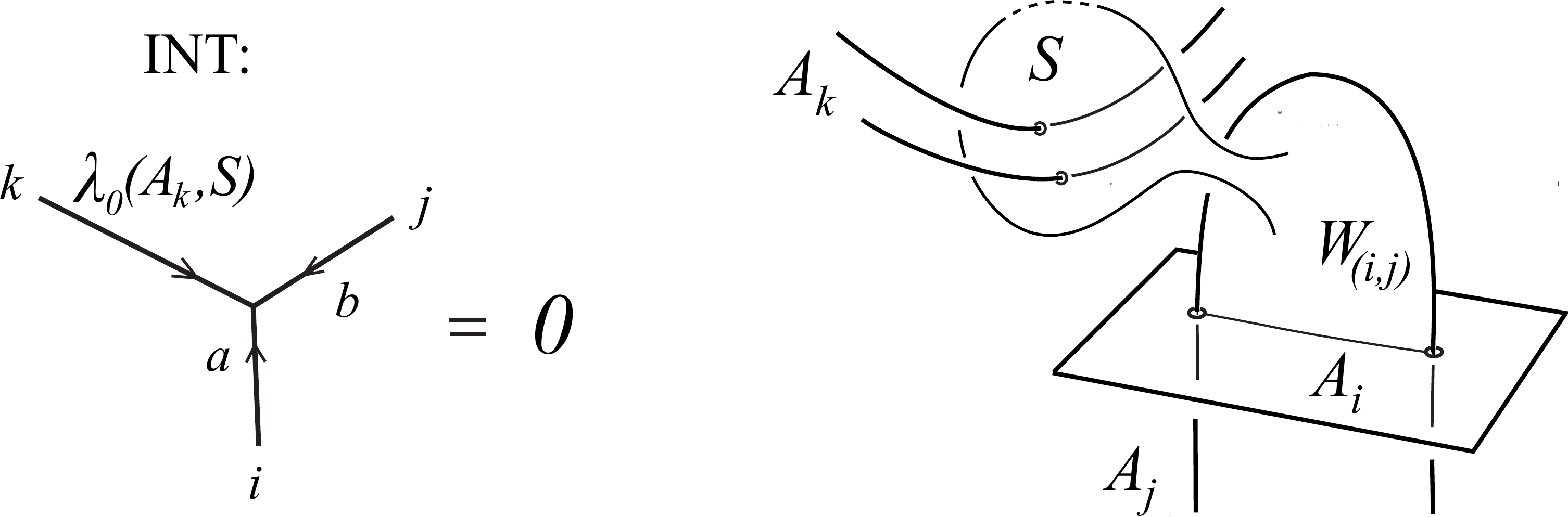}}
\caption{The INT \emph{intersection} relations in the target of $\lambda_1$, with $S:S^2\imra X$ varying over generators for $\pi_2(X)$ as a $\pi_1X$-module.}
         \label{fig:INT-relation}
\end{figure}

We next discuss how the AS \emph{antisymmetry}, HOL \emph{holonomy} and INT \emph{intersection} relations in the target of $\lambda_1$ account for indeterminacies due to choices in the construction of the (based, oriented) order $1$ non-repeating Whitney tower $\cW$:

The AS relations (Figure~\ref{fig:Relations-inward-arrows} left) account for the choice of orientations on the Whitney disks via a fixed convention choice for the induced vertex orientation of the trees (section~\ref{subsec:w-tower-tree-orientations}) so that signs of the order~1 intersections are well-defined (section~\ref{subsec:w-disk-orientation-choices-AS}). (In this non-repeating order~1 setting, AS relations could in fact be avoided by using the cyclic ordering of the distinct labels to prescribe orientations on all the Whitney disks in $\cW$, cf.~Exercise~\ref{ex:lambda1-cyclic-order}.) 

The HOL relations (Figure~\ref{fig:Relations-inward-arrows} right) account for the choices of whiskers on the trivalent vertices, since changing such a whisker corresponds to simultaneous right multiplication on the three group elements decorating the edges of the tree. 
We remark that by the splitting operation (section~\ref{subsec:split-w-towers}) it can be arranged that each Whitney disk contains exactly one tree, so that in a split Whitney tower choosing a whisker for each trivalent vertex is the same as choosing a whisker for each Whitney disk.




The INT relations (Figure~\ref{fig:INT-relation} left) account for choices of the interiors of Whitney disks in $\cW$, and depend on $A$ and $\pi_2X$ via the order $0$ intersection pairing $\lambda_0$. To clarify notation: The terms of $\lambda_0(A_k,S)$ decorating the edge of the tree in the left of Figure~\ref{fig:INT-relation} denote a linear combination of signed trees (section~\ref{subsubsec:tau1-relations}). Here $S$ is a $2$-sphere that has been tubed into the Whitney disk $W_{(i,j)}$, and $\lambda_0(A_k,S)$ is computed using a whisker for $S$ given by the tube together with a whisker for the trivalent vertex of the tree $(i,j)$ in $W_{(i,j)}$.
Note that the illustration of $S$ on the right side of Figure~\ref{fig:INT-relation} is schematic, as suggested by the dotted subarc.

The following characterization of $\lambda_1(A_1,A_2,A_3)$ shows in particular that it does not depend on the choice of order~1 non-repeating Whitney tower: 
\begin{thm}\label{thm:lambda1-vanishes}
$\lambda_1(A_1,A_2,A_3)$ only depends on the homotopy classes of the $A_i$, and the following three statements are equivalent:
\begin{enumerate}
\item\label{thm-item:lambda1-vanishes}
$\lambda_1(A_1,A_2,A_3)$ vanishes.

\item\label{thm-item:lambda1-vanishes-order-2}
$A_1\cup A_2\cup A_3$ admits an order $2$ non-repeating Whitney tower.

\item\label{thm-item:lambda1-vanishes-pull-apart}
$A_1,A_2,A_3$ can be made pairwise disjoint by a homotopy.

\end{enumerate}
\end{thm}
Theorem~\ref{thm:lambda1-vanishes} can be proved using arguments detailed in Section~\ref{sec:appendix} (see Exercise~\ref{ex:derive-non-repeating-order-1-thm-from-tau-1-proof}), with the equivalence of statements~(\ref{thm-item:lambda1-vanishes-order-2}) and~(\ref{thm-item:lambda1-vanishes-pull-apart}) being a special case of Theorem~\ref{thm:pull-apart} (see Exercise~\ref{ex:order-2-pull-apart-triple} for a proof in this case).

In the case that $X$ is simply connected $\lambda_1(A_1,A_2,A_3)$ reduces to the Matsumoto triple \cite{Ma} and the implication 
(\ref{thm-item:lambda1-vanishes}) $\Rightarrow$ (\ref{thm-item:lambda1-vanishes-pull-apart}) was shown by Yamasaki \cite{Y}.

For more than three components one similarly defines the order~1 non-repeating intersection invariant $\lambda_1(A)$ of $A=A_1,A_2,\ldots,A_m$ by counting, modulo the AS, HOL and INT relations, the distinctly-labeled decorated order~1 trees in the intersection forest of any order~1 non-repeating Whitney tower on $A$.
Then $\lambda_1(A)$ is a homotopy invariant of $A$ which vanishes if and only if $A$ supports an order~2 non-repeating Whitney tower.

Note that the OR relations of section~\ref{subsec:order-0-lambda-mu} are not needed in the order~1 setting because we are counting order~1 trees so edges can always be taken to be oriented towards the trivalent vertex. 

\subsection{Pulling apart $m$-tuples of spheres}\label{subsec:higher-order-lambda}
As described in detail in \cite{ST3}, the above discussion gives the framework for a complete obstruction theory for \emph{pulling apart $m$ $2$-spheres}, ie.~making the spheres pairwise disjoint by a homotopy: 
Theorem~\ref{thm:pull-apart} says that the existence of an order $m-1$ non-repeating Whitney tower on $A_1\cup A_2\cup\cdots\cup A_m$ suffices to pull the $A_i$ apart. And given an order~$n$ non-repeating Whitney tower $\cW$ supported by $A=A_1,A_2,\ldots,A_m$, the intersection forest $t(\cW)$ represents an obstruction to the existence of an order~$(n+1)$ non-repeating Whitney tower. Denoting by $\Lambda_n(m)$
the abelian group generated by order $n$ trees with distinctly labeled univalent vertices from $\{1,2,\ldots,m\}$ and edges decorated by elements of $\pi_1X$, modulo OR, AS, HOL and IHX relations, where for $n\geq 2$ the IHX relation is as in \ref{cor:IHX} but with distinct univalent labels and with edge decorations, this obstruction lives in a quotient of $\Lambda_n(m)$ by order $n$ \emph{intersection relations} $\mbox{INT}_n(A)$ which correspond to changes in $t(\cW)$ coming from tubing Whitney disk interiors into $2$-spheres. 
Since all the relations correspond to geometric manipulations of $\cW$, the vanishing of the obstruction implies the existence of an order~$(n+1)$ non-repeating Whitney tower via the same order-raising maneuvers as in Section~\ref{subsec:order-raising-proof-sketch}, which are all homogeneous in the univalent labels.

There are two main challenges here: One is concisely formulating the $\mbox{INT}_n(A)$ intersection relations, which for $n\geq 2$ can be non-linear.
And the other is showing that no additional relations other than OR, AS, HOL, IHX and $\mbox{INT}_n(A)$ are needed to ensure that $\lambda_n(A)$ does not depend on the choice of $\cW$. This second issue essentially amounts to showing that $\lambda_n(A)$ does not depend on the choices of Whitney disk boundaries. Evidence that the AS, HOL and IHX relations suffice to give independence of Whitney disk boundaries comes from Theorem~8 of \cite{ST3} which states that $\lambda_n(A)\in\Lambda_n(m)$ does not depend on the choice of $\cW$ in the setting that the components $A_i$ are properly immersed disks into the $4$-ball, where $\mbox{INT}_n(A)$ relations are trivial (and OR is not relevant).

By \cite[Lem.19]{ST3}, $\Lambda_n(m)$ is isomorphic to the direct sum of $\binom{m}{n+2}n!$-many copies of the integral group ring $\Z[\pi_1X^{(n+1)}]$ of the 
$(n+1)$-fold cartesian product $\pi_1X^{(n+1)}=\pi_1X\times\pi_1X\times\cdots\times\pi_1X$.
Note that $\Lambda_n(m)$ is trivial for $n\geq m-1$ since an order $n$ unitrivalent tree has $n+2$ univalent vertices.
The absence of torsion in
$\Lambda_n(m)$ is in alignment with the fact that the torsion subgroup (which is only $2$-torsion) of the tree group $\cT_n$  of Definition~\ref{def:untwisted-tree-groups} (repeating labels allowed) corresponds to the obstructions to converting twisted Whitney towers to framed Whitney towers in the $4$--ball \cite{CST3}.  Such obstructions are not relevant in the non-repeating setting since
the boundary-twisting operation (Figure~\ref{boundary-twist-fig})
can be used to eliminate twisted Whitney disks
at the cost of only creating intersections whose trees do not have distinctly-labeled univalent vertices.

See \cite{ST3} for further discussion of pulling apart $m$ components, including examples where $\lambda_n(A)$ is well-defined.

The case of pulling apart $4$-tuples of spheres in a simply connected $4$--manifold is already interesting, and only partially understood.
In \cite[Sec.8]{ST3} the order~2 invariant $\lambda_2(A)$ for $A=A_1,A_2,A_3,A_4$ supporting an order~2 non-repeating tower $\cW$ in a simply connected $X$ is examined in detail, including a precise formulation of the $\mbox{INT}_2(A)$ relations which account for choices of Whitney disk interiors. 

The following is Conjecture~29 of \cite{ST3}:
\begin{conj}\label{conj:lambda2-well-defined}
For any $A=A_1,A_2,A_3,A_4$ supporting an order~2 non-repeating tower $\cW$ in $X$, the components of $A$ can be pulled apart if and only if $\lambda_2 (A)\in\Lambda_2/\INT_2(A)$ as defined in \cite[Sec.8]{ST3} vanishes.
\end{conj}
As explained in \cite[Sec.8.3.6]{ST3}, the computation of the image of the INT$_2(A)$ relations in $\Lambda_2(4)\cong\Z\oplus\Z$ leads to some interesting number theory \cite{KN}.



\subsection{The order 1 self-intersection invariant $\tau_1$}\label{subsec:tau-1}
Having considered non-repeating order~1 invariants which generalize the classical order~0 intersection pairing and fit into an obstruction theory for pulling apart $m$-tuples of spheres, we next describe the order~1 generalization of the classical self-intersection invariant which fits into an obstruction theory for the existence of order~$n$ framed Whitney towers.
 
For simplicity we restrict attention to the case of a single immersed 2-sphere $A:S^2\imra X^4$ with 
$\tau_0(A)=0$ (or in classical notation $\mu_0 (A)=0$). The vanishing of $\tau_0(A)$ means that $A$ admits an order~$1$ framed Whitney tower $\cW$, and we want to ``count'' the decorated trees $t_p$ associated (as in section~\ref{subsec:decorated-order-1-trees}) to the order~1 intersections $p$ in $\cW$ to define an order~1 invariant $\tau_1(A)$ which does not depend on the choice of $\cW$ and gives the complete obstruction to $A$ bounding an order~2 framed Whitney tower.

Denote by $\tilde{\cT_1}:=\tilde{\cT_1}(\pi_1X)$ the quotient of the free abelian group on decorated order~1 trees by the above AS, HOL relations (Figure~\ref{fig:Relations-inward-arrows}) and also the new
FR \emph{framing relations} shown in Figure~\ref{fig:FR-relation-single-sphere}. The necessity of these FR relations is illustrated (schematically) on the right of Figure~\ref{fig:FR-relation-single-sphere}, which shows how performing opposite boundary-twists (Figure~\ref{boundary-twist-fig}) on a framed Whitney disk $W$ along different arcs of $\partial W$ yields a new framed Whitney disk that has two interior intersections with $A$ corresponding trees equal to the terms of an FR relation. Univalent labels are dropped from our trees since we are now working with just a single component. So the previously defined AS and HOL relations are now expressed in $\tilde{\cT_1}$ by the same equations but without univalent labels. 

\begin{figure}[h]
\centerline{\includegraphics[scale=.5]{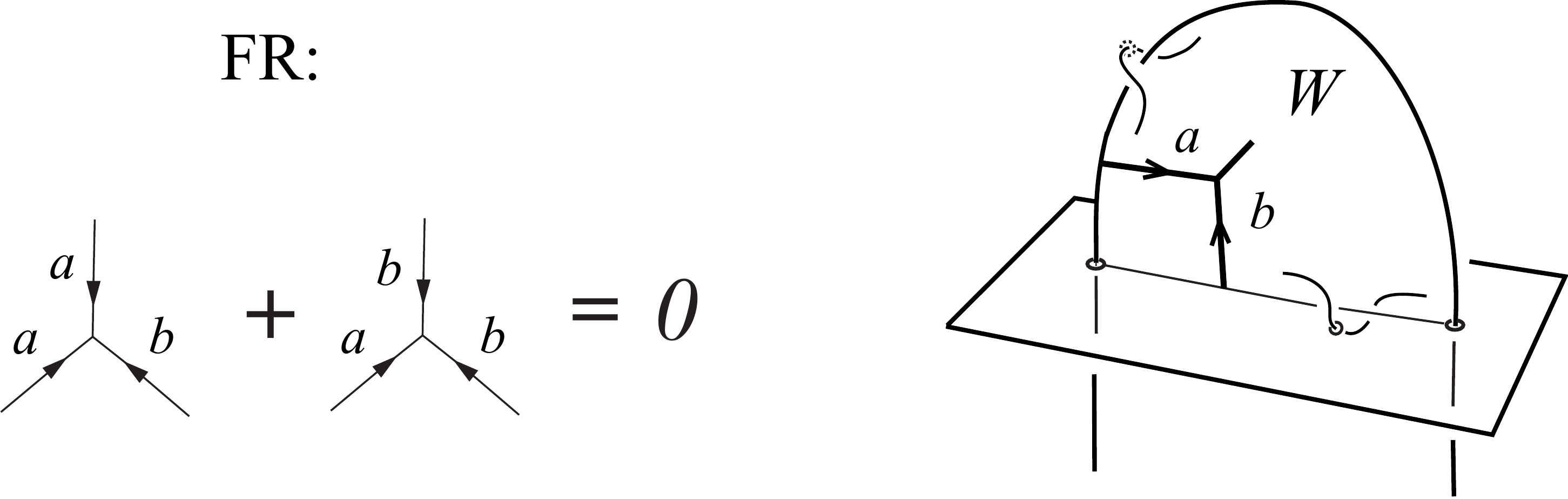}}
\caption{The FR framing relations; $a,b\in\pi_1X$}
         \label{fig:FR-relation-single-sphere}
\end{figure}

We also need to refine the INT relations to account for the effect that tubing into a sphere can have on Whitney disk twistings. 
For instance, if a framed Whitney disk is tubed into a sphere $S$ whose normal bundle has non-trivial 2nd Stiefel-Whitney number $\omega_2(S)$, then a framed Whitney disk can be recovered by boundary-twisting. Figure~\ref{fig:INT-relation-tau} shows these refined INT relations in the target of $\tau_1$, with the edge decoration $\omega_2(S)$ identified with $0$ or $1$ in $\Z[\pi_1X]$ depending on whether it is trivial or not. This is explained in complete detail in section~\ref{subsec:tau1-well-defined}, where the full definition of the INT relations is given, which also allows for $S$ to be an immersed $\RP^2$ in certain cases depending on the elements $a,b$.   
\begin{figure}[h]
\centerline{\includegraphics[scale=.4]{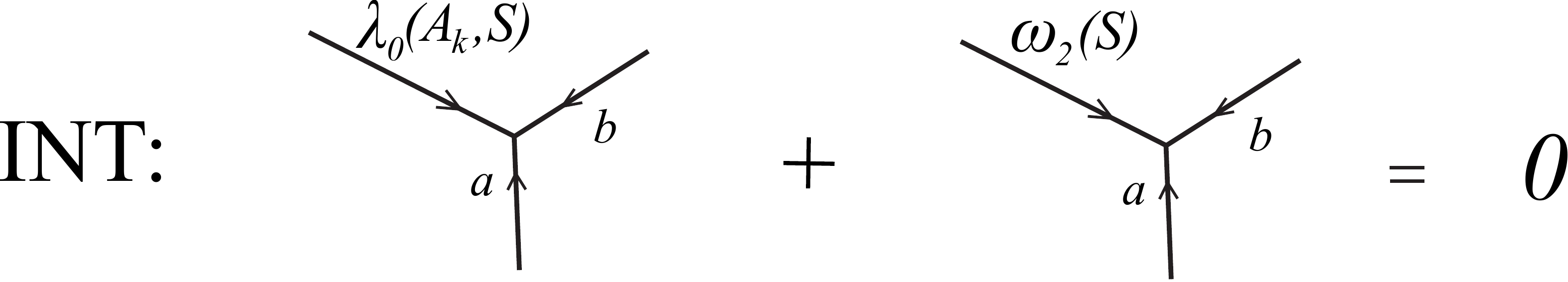}}
\caption{The INT \emph{intersection} relations in the target of $\tau_1$; with $a,b\in\pi_1X$, and $S:S^2\imra X$ varying over generators for $\pi_2(X)$ as a $\pi_1X$-module.}
         \label{fig:INT-relation-tau}
\end{figure}

So for $A$ with vanishing $\tau_0(A)$, choose an order~1 framed Whitney tower $\cW$ on $A$ and define the \emph{order~1 framed intersection invariant}:
$$
\tau_1(A):=\sum\epsilon_p\cdot t_p\in\tilde{\cT_1}/\mbox{INT}
$$
where the sum is over all order~1 intersections $p$ in $\cW$.

The following result of \cite{ST1} shows that $\tau_1(A)$ is an obstruction to homotoping $A$ to an embedding, and in particular that it does not depend on the choice of $\cW$:
\begin{thm}\label{thm:tau1-vanishes}

$\tau_1(A)$ only depends on the homotopy class of $A$, and the following four statements are equivalent:

\begin{enumerate}

\item\label{thm-item:tau1-vanishes}
$\tau_1(A)$ vanishes.

\item\label{thm-item:A-admits-order2-tower}
$A$ admits an order 2 framed Whitney tower (cf.~Definition~\ref{def:order-n-framed-W-tower}).

\item\label{thm-item:A-admits-height1-tower}
$A$ admits a height 1 Whitney tower (cf.~Definition~\ref{def:height}).

\item\label{thm-item:A-embeds-stably}
$A$ is \emph{stably homotopic to an embedding}.

\end{enumerate}
\end{thm}

Here $A$ being \emph{stably homotopic to an embedding} means that $A$ is homotopic to an embedding in the connected sum $X\#^nS^2\times S^2$ of $X$ with some number $n$ of copies of $S^2\times S^2$.


See section~\ref{subsec:tau1-well-defined} for proof that $\tau_1(A)$ is a well-defined homotopy invariant.
The equivalence of statements~(\ref{thm-item:lambda1-vanishes}) and (\ref{thm-item:A-admits-order2-tower})
is shown in section~\ref{subsec:tau1-vanishes-equals-order-2-w-tower}.

See \cite[Thm.2]{ST1} for 
the equivalence of statements~(\ref{thm-item:lambda1-vanishes}) and (\ref{thm-item:A-admits-height1-tower});
and \cite[Cor.1]{S3}
for the equivalence of statements~(\ref{thm-item:lambda1-vanishes}) and (\ref{thm-item:A-embeds-stably}).


\subsection{Examples and questions}\label{subsec:tau1-examples-and-questions}

If $X$ is simply-connected then the target $\tilde{\cT_1}/\mbox{INT}$ of $\tau_1(A)$ is $\Z/2\Z$ or $0$, depending on whether $A$ is spherically characteristic or not.
Here $A$ is \emph{spherically characteristic} if $\lambda_0(A,S)\equiv\lambda_0(S,S)$ mod $2$ for all $S\in\pi_2X$.
In this simply-connected case $\tilde{\cT_1}$ is generated by the order~1 tree with all trivial edge decorations, which is $2$-torsion by the AS relations, and is equal to zero by the INT relations if $A$ is not spherically characteristic.

For example, $3\CP^1\imra\CP^2$ is spherically characteristic, and Figure~\ref{fig:3CP1} shows the computation of $\tau_1(3\CP^1)=1\neq 0\in\tilde{\cT_1}/\mbox{INT}(3\CP^1)\cong\Z/2\Z$, so $3\CP^1$ is not homotopic to an embedding (compare \cite{KM}).

\begin{figure}[h]
\includegraphics[scale=.28]{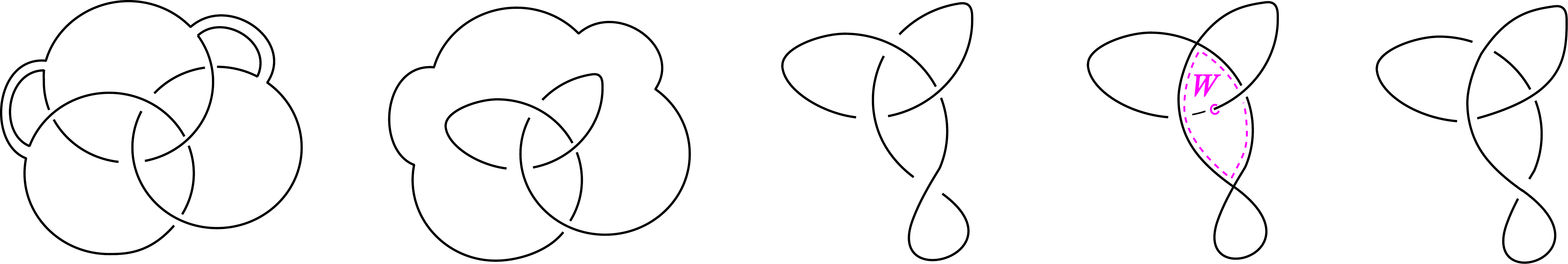}
\caption{Left: The boundary circle in $S^3$ of the sum of three parallel copies of the $+1$-framed 2-handle in a standard handle-body description of $\CP^2$. Moving to the right: The track into the $0$-handle $B^4$ of a null-homotopy of this circle describes the rest of $3\CP^1$, with two self-intersections admitting the framed Whitney disk $W$ that has a single transverse intersection with $3\CP^1$.}\label{fig:3CP1}
\end{figure}

For $X$ not necessarily simply-connected, taking the quotient of $\tilde{\cT_1}$ by $\pi_1X\to1$ yields $\tau_1(A)\in \Z/2\Z \mbox{ or } 0$ depending on whether $A$ is spherically characteristic or not, and in this quotient target $\tau_1$ reduces to the spherical \emph{Kervaire--Milnor} invariant $\km$ of
\cite{FQ,Stong}.

Even after trivializing all $\pi_1X$-decorations $\tau_1=\km$ sees global information in \emph{closed} 4-manifolds:

\begin{thm}[Freedman--Kirby, Kervaire--Milnor, Stong]\label{thm:tau1-signature}
Suppose $X$ is a smooth closed $4$-manifold, and $H_2(X;\Z/2\Z)$ is spherical.
If $A:S^2\imra X$ is characteristic and $\mu_0 A=0$, then
$$
(\pi_1X\to1):\quad\tau_1(A)\quad\mapsto\quad\frac{A\cdot A-\textrm{signature} (X)}{8}\quad\mbox{ \textrm{mod} }2
$$
\end{thm}
See the end of \cite{Stong} for a proof.


On the other hand, for $\pi_1X$ non-trivial the cardinality $|\tilde{\cT_1}(\pi_1X)|$ can be large. For example, if
$\pi_1X$ is left-orderable and the INT relations are trivial then $\tilde{\cT_1}(\pi_1X)$ is isomorphic to $\Z^\infty\oplus(\Z/2\Z)^\infty$ (see \cite[Prop.2.3.1]{S3}).

Let $\Gamma$ be a finitely presented group, and let $g$ be any element of $\tilde{\cT_1}(\Gamma)$. 
Then one can find a $4$-manifold $X$ with non-empty boundary such that $\pi_1X=\Gamma$, and $A:S^2\imra X$ with INT($A$) trivial and $\tau_1(A)=g$ (Exercise~\ref{Ex:tau1-realization-borro-rings}).


But finding examples of non-trivial $\tau_1(A)$ in \emph{closed} $4$-manifolds that depend on non-trivial edge decorations appears to be difficult, and no such examples are currently known.
Let $(1,1,1)$ denote the order~1 tree with all three edges decorated by the trivial element $1\in\pi_1X$.
Then $(1,1,1)$ is $2$-torsion by AS (or FR) relations, and generates an order~2 subgroup of $\tilde{\cT_1}(\pi_1X)$, for any $\pi_1X$.
We have the following open ``realization problem'' for closed $X$:

\begin{prob}\label{prob:realize-tau1-non-trivial-edges}
Find an example of $A\imra X$ such that the following three statements all hold:
\begin{enumerate}
\item $X$ is closed, and 
\item $|\tilde{\cT_1}(\pi_1X)/\emph{INT}(A)|\geq 3$, and
\item $\tau(A)$ is not contained in the subgroup generated by $[(1,1,1)]\in\tilde{\cT_1}(\pi_1X)/\emph{INT}(A)$.
\end{enumerate}
\end{prob}

Poincare duality contributes to the difficulty of Problem~\ref{prob:realize-tau1-non-trivial-edges} since if $\lambda_0(A,S)=1$ for some $S\in\pi_2X$, then $|\tilde{\cT_1}(\pi_1X)/\mbox{INT}(A)|\leq 2$ by the INT relations and $\tau_1$ reduces to $\km$.

An easier to state, but possibly more difficult, open realization problem is: Find an example of $A\imra X$ such that $\km(A)=0$ but $\tau_1(A)\neq 0$. 

In light of Theorem~\ref{thm:tau1-signature} and Problem~\ref{prob:realize-tau1-non-trivial-edges} one is led to the question:
\begin{question}
What global information is carried by non-trivial $\pi_1$-decorations in $\tau_1$ in closed 4-manifolds?
\end{question}

In comparison with the classifications of order~$n$ framed and twisted Whitney towers on 2-disks in the $4$-ball (\cite{CST2} and Section~\ref{sec:twisted-order-n-classification-arf-conj}), the higher-order invariants invariants for $2$-spheres in $4$-manifolds have barely begun to be defined, even in simply connected $4$-manifolds. Defining an order~2 framed invariant is already an interesting challenge:
\begin{prob}
For $A:S^2\imra X$ with vanishing $\tau_1(A)$, formulate and prove invariance of an order~2 framed invariant $\tau_2(A)$ whose vanishing is equivalent to $A$ admitting a framed order~3 Whitney tower.
\end{prob}
In addition to the issues that arise as in Conjecture~\ref{conj:lambda2-well-defined}, here one has to also keep track of Whitney disk twistings.




%
%
%

%
%
%
%
%
%
%
%
%
%
%
%
%
%
%
%
%
%
%
%


\subsection{Section~\ref{sec:2-spheres-in-4-manifolds} Exercises}\label{subsec:exercises-2-spheres}

\subsubsection{Exercise:}\label{ex:lambda-0-vanishes-gives-w-disks}
Suppose $A_i$ and $A_j$ are oriented immersed 2-spheres in an oriented 4-manifold $X$, and $p$ and $q$ are oppositely signed intersections in $A_i\pitchfork A_j$ with equal associated group elements $g_p=g_q\in\pi_1X$. Check that the union of a path in $A_i$ from $p$ to $q$ with a path in $A_j$ from $q$ to $p$ is a null-homotopic loop in $X$. Conclude that $p$ and $q$ admit an immersed Whitney disk.

\subsubsection{Exercise:}\label{ex:pairwise-lambda-0-vanishes-gives-order-1-nonrep-tower}
For a collection of immersed 2-spheres $A=A_1,A_2,\ldots,A_m\imra X^4$, the vanishing of $\lambda_0(A)$ (section~\ref{subsec:order-0-lambda-mu}) is equivalent to the pairwise vanishing of the classical intersection pairing $\lambda_0(A_i,A_j)\in\Z[\pi_1X]$ for $i\neq j$ (section~\ref{subsec:classical-mu-lambda}).
Show that $\lambda_0(A)$ vanishes if and only if $A$ supports an order~$1$ non-repeating Whitney tower.


\subsubsection{Exercise:}\label{ex:make-twisted-w-disk-framed}
Show that any twisted Whitney disk can be converted to a framed Whitey disk having the same boundary by applying the boundary-twisting operation of section~\ref{subsec:geometry-of-relations}.

\subsubsection{Exercise:}\label{ex:order-1-w-tower-implies-tau-0-vanishes}
Show that $\tau_0(A)$ vanishes if $A$ supports an order~1 framed Whitney tower.

\subsubsection{Exercise:}\label{ex:vanishing-tau-0-gives-order-1-w-tower}
Show that $A$ supports an order~$1$ framed Whitney tower if $\tau_0(A)$ vanishes.


\subsubsection{Exercise:}\label{ex:order-2-pull-apart-triple}
If $A_1,A_2,A_3$ support an order $2$ non-repeating Whitney tower, show that $A_1,A_2,A_3$ can be made pairwise disjoint by a homotopy.

\subsubsection{Exercise:}\label{ex:geo-cancelling-implies-alg-cancelling}
If order~1 intersections $p,q\in W_{(i,j)}\pitchfork A_k$ are paired by an order~2 Whitney disk $W_{((i,j),k)}$, then $p$ and $q$ have opposite signs by the definition of Whitney disks.
Check that the order~1 decorated trees $t_p$ and $t_q$ associated to such $p$ and $q$ are equal for appropriate choices of trivalent whiskers.

\subsubsection{Exercise:}\label{ex:lambda1-cyclic-order}

Show that in the target of $\lambda_1(A_1,A_2,A_3)$ of section~\ref{subsec:lambda1} the AS relations could be avoided by using the cyclic ordering of the distinct labels to prescribe orientations on all the Whitney disks in $\cW$ (via a choice of positive or negative corner convention cf. section~\ref{subsec:w-tower-tree-orientations}).

\subsubsection{Exercise:}\label{Ex:tau1-realization-borro-rings}
Let $\Gamma$ be a finitely presented group, and let $t$ be any order~1 decorated tree representing an element $[t]\in\tilde{\cT_1}(\Gamma)$. 
Find a $4$-manifold $X$ with non-empty boundary such that $\pi_1X=\Gamma$, and $A:S^2\imra X$ with INT($A$) trivial and $\tau_1(A)=[t]$.

%

\subsubsection{Exercise:}\label{Ex:tau1-realization-2}
Generalizing the previous exercise, let $\Gamma$ be a finitely presented group, and let $g$ be any element of $\tilde{\cT_1}(\Gamma)$. 
Find a $4$-manifold $X$ with non-empty boundary such that $\pi_1X=\Gamma$, and $A:S^2\imra X$ with INT($A$) trivial and $\tau_1(A)=g$.


\subsubsection{Exercise:}\label{Ex:tau1-realization-3}
Generalizing the previous exercise, let $\Gamma$ be a finitely presented group, let $g$ be any element of $\tilde{\cT_1}(\Gamma)$, and let 
$\{z_1,z_2,\ldots,z_n\}$ be any collection $z_i\in\Z[\Gamma]$. 
Find a $4$-manifold $X$ with non-empty boundary such that $\pi_1X=\Gamma$, and $A:S^2\imra X$ with INT($A$) determined by $\lambda_0(A,S_i)=z_i$ for $S_i$ generating $\pi_2X$, and $\tau_1(A)=g$.

\section{Appendix}\label{sec:appendix}

Section~\ref{subsec:tau1-well-defined} gives a detailed proof of the homotopy invariance
of the order~1 invariant $\tau_1(A)$ from section~\ref{subsec:tau-1}, and section~\ref{subsec:tau1-vanishes-equals-order-2-w-tower} proves that $\tau_1(A)$ is the complete obstruction to $A$ supporting an order~2 framed Whitney tower.
This proves the first two statements of Theorem~\ref{thm:tau1-vanishes}, and simpler versions of the analogous arguments prove the first two statements of Theorem~\ref{thm:lambda1-vanishes} (Exercise~\ref{ex:derive-non-repeating-order-1-thm-from-tau-1-proof}).

The splitting of Whitney towers illustrated in Figure~\ref{split-w-tower-with-trees-fig} is extended to twisted Whitney towers in
section~\ref{subsec:split-w-towers}.

Whitney move versions of the IHX and twisted IHX relations are described in detail in section~\ref{subsec:w-move-twistedIHX}.
These constructions are essential to the framed, twisted, and non-repeating order-raising obstruction theories. A proof of the twisted order-raising Theorem~\ref{thm:twisted-order-raising} is sketched in
section~\ref{subsec:order-raising-proof-sketch}.

%
%




\subsection{Homotopy invariance of $\tau_1(A)$}\label{subsec:tau1-well-defined}
Here we show that for $A:S^2\imra X^4$ the order~1 framed intersection invariant $\tau_1(A)$ of Theorem~\ref{thm:tau1-vanishes} in section~\ref{subsec:tau-1}
only depends on the homotopy class of $A$.

The definition of $\tau_1(A)$ requires that $A$ supports an order~1 framed Whitney tower $\cW$, the existence of which is equivalent to $A$ having vanishing order~0 self-intersection invariant $\tau_0(A)=0$ by Theorem~\ref{thm:tau-0-vanishes}. As recalled and clarified below, $\tau_1(A)$ is determined by the intersection forest $t(\cW)$, which in this setting is a multiset of signed decorated order~1 trees, one for each transverse intersection between $A$ and a Whitney disk in $\cW$.
The main challenge is to show that the element represented by $t(\cW)$ in the target of $\tau_1(A)$ does not depend on the choices of oriented Whitney 
disks in $\cW$. Then invariance under homotopies of $A$ will be shown in section~\ref{subsubsec:htpy-invariance-tau1}.

%
%
%
%
%
%
%
%
%
%
%
%
%
%
%
%
%
%

\subsubsection{Order~1 tree notation}\label{subsubsec:Y-tree-notation}
We introduce the following streamlined notation for the order~1 decorated trees of section~\ref{subsec:decorated-order-1-trees} in our current $1$-component setting: Since now the univalent labels are not relevant, a tree is determined by a cyclic ordering of the elements of $\pi_1X$ decorating the edges. By appropriately choosing the whisker on the trivalent vertex, one edge decoration on each tree can be normalized to the trivial element $1\in\pi_1X$. So for $a,b\in\pi_1X$ we will use the notation $(a,b)$ to denote the decorated order~1 tree with one edge decorated by $1$ and the other two edges decorated by $a$ and $b$ respectively, with the vertex orientation given by the cyclic ordering $1\to a\to b$.

As before, any edge decorations in $\Z[\pi_1X]$ are understood to be linear combinations of trees.


\subsubsection{Relations in target of $\tau_1$}\label{subsubsec:tau1-relations}
Recall from section~\ref{subsec:tau-1} that $\tilde{\cT_1}:=\tilde{\cT_1}(\pi_1X)$ denotes the abelian group on decorated order~1 trees modulo the AS \emph{antisymmetry}, HOL \emph{holonomy} and FR \emph{framing} relations (see just below), and that $\tau_1(A)$ takes values in $\tilde{\cT_1}$ modulo
the INT \emph{intersection} relations. Since these last relations depend on $A$ we sometimes denote them by INT$(A)$ for emphasis.

The relations in the target of $\tau_1(A)$ expressed in $\Z[\pi_1X]\times\Z[\pi_1X]$ via the streamlined notation are:

\begin{align*}
\mbox{AS:}\quad\quad&(a,b)=-(b,a)\\
\mbox{HOL:}\quad\quad&(a,b)=(ba^{-1},a^{-1})=(b^{-1},ab^{-1})\\
\mbox{FR:}\quad\quad&(1,a)+(a,a)=0\\
\mbox{INT:}\quad\quad&(a,\lambda_0(A, S) +\omega_2(S)\cdot 1 )=0
\end{align*}
Here $a,b\in\pi_1X$, and $S$ ranges over generators for $\pi_2X$ (as a module over $\pi_1X$).

The AS and HOL relations imply that the decorated trees $(a,1)$ and $(a,a)$ are $2$-torsion elements, which by FR are equal. This will be used in several places below.

Terms are expanded linearly in the INT relations, so 
$$
(a,\lambda_0(A, S) +\omega_2(S)\cdot 1 )=(a,\lambda_0(A, S))+ (a,\omega_2(S)\cdot 1 ), 
$$
and
$$
(a,\lambda_0(A, S))=(a,\sum\epsilon_p\cdot g_p):=\sum\epsilon_p\cdot(a,g_p) 
$$
 for $\lambda_0(A, S)=\sum_{p\in A\pitchfork S}\epsilon_p\cdot g_p$.
 
For the case $a^2=1\in\pi_1X$ in the INT relations, we allow $S$ to be an immersed $\RP^2$ representing $a$, ie.~$a\in\pi_1X$ is the image of $\pi_1\RP^2$.
When $S$ is an $\RP^2$ the order $0$ intersection pairing $\lambda_0(A,S)$ is only well-defined up to right multiplication by $a$ and change of sign, but the AS and HOL relations make $(a,\lambda_0(A, S))$ well-defined in $\tilde{\cT_1}$. This will be discussed in detail in section~\ref{subsubsec:RP2-INT} where this form of the INT relation is used in the proof that $\tau_1(A)$ is well defined.

\subsubsection{Definition of $\tau_1(A)$}\label{subsubsec:tau1-def}

For $A$ satisfying $\tau_0(A)=0$, choose any order~1 framed Whitney tower $\cW$ on $A$ and define:
$$
\tau_1(A):=\sum \epsilon_p\cdot t_p\in\tilde{\cT_1}/\mbox{INT}(A)
$$
where the sum is over all transverse intersections $p$ between $A$ and the Whitney disks in $\cW$.
We use the \emph{positive corner} convention from section~\ref{subsec:w-tower-tree-orientations} for inducing cyclic orientations at the trivalent vertices from the Whitney disk orientations.

We remark that by eliminating all order~2 intersections using the pushing-down operation (Exercise~\ref{ex:push-down}) it may be arranged that the intersection forest $t(\cW)$ (section~\ref{subsec:int-forest-def}) is equal to the sum defining $\tau_1(A)$, ie.~$\tau_1(A)=[t(\cW)]\in\tilde{\cT_1}/\mbox{INT}(A)$. In the following this will be assumed for clarity, although it is not necessary.

\subsubsection{Independence of Whitney disk interiors:}\label{subsubsec:w-disk-interiors}
Let
$W$ be a Whitney disk in $\cW$. If $V$
is another framed Whitney disk with same boundary $\gamma$ as $W$, then the union of
$W$ and $V$ is a map of a $2$--sphere, which might not be smooth
along $\gamma$. We want to use this sphere to express, via an INT relation, the change in $\tau_1(A)$ due to replacing $W$ by $V$. It is not obvious what intersections with $A$ might be created by perturbing the sphere to be smooth and transverse to $A$, so we will proceed carefully.

Denote by $\alpha$ and $\beta$ the two closed subarcs of $\alpha\cup\beta=\gamma:=\partial W=\partial V$ that run between the self-intersections $p$ and $q$ of $A$ that are paired by $W$ and $V$.

A small collar of $\alpha$ in $W$ determines a $1$-dimensional subbundle $w_\alpha$ of $\nu_XA|_\alpha$, the normal $D^2$-bundle to $A$ restricted to $\alpha$. Similarly, a small collar of $\alpha$ in $V$ determines a $1$-dimensional subbundle $v_\alpha$ of $\nu_XA|_\alpha$. Note that at each of the endpoints $p$ and $q$ of $\alpha$ there is a canonical $1$-dimensional subspace of $\nu_XA|_\alpha$ given by a short subarc of $\beta$. We may assume that $w_\alpha=v_\alpha$ near $p$ and $q$ by a small isotopy, so there is a relative rotation number $m\in\Z$ of $v_\alpha$ with respect to $w_\alpha$. A boundary-twist (Figure~\ref{boundary-twist-fig}) on $V$ changes $m$ by $\pm1$, so by performing $m$-many boundary-twists on $V$ we can get a new Whitney disk $V'$ such that the corresponding $1$-dimensional subbundle $v'_\alpha$ of $\nu_XA|_\alpha$ has zero relative rotation number with respect to $w_\alpha$, and hence it can be arranged by a small isotopy of $V'$ near $\alpha$ that $v'_\alpha=w_\alpha$ everywhere along $\alpha$. 

Applying the same discussion and boundary-twisting $n$ times along $\beta$, we can also arrange that a small collar $v'_\beta$ of $\beta$ in $V'$ satisfies 
$v'_\beta=w_\beta\subset\nu_XA|_\beta$ everywhere along $\beta$.


Now the unions $w:=w_\alpha\cup w_\beta$ and $v':=v'_\alpha\cup v'_\beta$ are collars of $\gamma$ in $W$ and $V'$ which are equal $w=v'$.

Removing these collars and gluing we get a map of a $2$-sphere $S:=(W\setminus w)\cup (V'\setminus v')$ which is smooth except along the gluing circle $\gamma':=\partial (W\setminus w)=\partial (V'\setminus v')$. Since $\gamma'$ is disjoint from $A$, perturbing $S$ to be smoothly immersed will not create any new intersections with $A$. (Note that we do not need to control self-intersections of $S$ because they do not correspond to contributions to $\tau_1$.)
The intersections between $S$ and $A$ are transverse and consist of $A\pitchfork W$ and $A\pitchfork V'$, where $A\pitchfork V'$ consists of $A\pitchfork V$ together with the $(m+n)$-many intersections of $A\pitchfork V'$ which were created by the boundary-twists on $V$. These boundary-twist intersections correspond to the trees $m(1,a)+n(a,a)$, where $m$ twists were done along the arc $\alpha\subset\gamma$ whose corresponding edge decoration has been normalized to $1$, and $n$ twists were done along the other arc $\beta\subset\gamma$ whose corresponding edge decoration is $a$ (Figure~\ref{fig:FR-relation-single-sphere}, with $b=1$). Taking the orientation of $S$ to be induced by the orientation of $V'$ and the opposite orientation of $W$, we have that the difference in $\tau_1(A)$ due to replacing $W$ by $V$ is $(a,\lambda_0(A,S))+m(1,a)+n(a,a)$. It just remains to check that $\omega_2(S)\equiv m+n$ mod $2$, so that this difference is an INT relation.

To compute $\omega_2(S)$ as a mod 2 self-intersection number, observe that $S$ is homotopic to $S'=W\cup V'$ by a homotopy supported near the common collars $w=v'$ that were deleted to form $S$ (Exercise~\ref{ex:independence-of-interior-htpy}). 
Now we can assume that $S'$ has been perturbed to be smooth along $\gamma$ without worrying about creating intersections with $A$ since we will only be counting the mod 2 intersections between $S'$ and a normal push-off $S''$.
Form $S''$ by taking any Whitney parallel push-off of $\gamma$ and extending over $W$ and $V'$ by Whitney sections to get the two hemispheres of $S''$, and count the number of intersections in $S'\pitchfork S''$: The number of intersections between $W$ and its push-off must be even, since $W$ is framed. 
The number of intersections between $V'$ and its push-off must equal $m+n$ mod 2, since $V'$ was created by $m+n$ boundary-twists on the framed $V$.
Each transverse intersection between $W$ and $V'$ will contribute two intersections between their push-offs.

So we have $\omega_2(S)\equiv m+n$ mod $2$,
and it follows that the
change in $\tau_1(\cW)$ resulting from replacing $W$ by $V$ is
exactly described by the INT($A$) relation $(a,\lambda_0(A,S))+m(1,a)+n(a,a)=(a,\lambda_0(A, S) +\omega_2(S)\cdot 1 )$.

\subsubsection{Independence of Whitney disk boundaries:}\label{subsubsec:w-disk-boundaries} 
To show that
$\tau_1(\cW)$ does not depend on the choices of boundaries of the
Whitney disks, for fixed pairings of the self-intersections
in $A$, it is convenient to temporarily weaken the definition of an order~1
framed Whitney tower by allowing transverse intersections among the
boundaries of the Whitney disks in $\cW$ (following \cite[Sec.10.8]{FQ} and \cite{ST1}). 
The definition of $\tau_1$ is extended to such Whitney
towers by assigning trees to the boundary intersections between
Whitney disks in the following way.



The Whitney disk orientations induce orientations on their boundaries via the usual convention
that $\overrightarrow{\partial W}$ together with a second inward
pointing vector give the orientation of $W$. 
We will use the
notation $\partial_+W$ to indicate the boundary arc of $W$ that
is oriented towards the positive self-intersection of $A$ paired by $W$, and
$\partial_-W$ for the boundary arc of $W$ that is oriented
towards the negative self-intersection paired by $W$.

Let $p\in
\partial_{\epsilon}W\cap\partial_{\delta}V$, for $ \epsilon,\delta\in\{+,-\}$, be a point
such that the ordered pair of tangent vectors
$(\overrightarrow{\partial_{\epsilon}
W},\overrightarrow{\partial_{\delta} V})_p$ is equal
to the orientation of $A$ at $p$.
Choose whiskers on the trivalent vertices of the rooted trees associated to $W$ and $V$ so that the edges dual to $\partial_{\epsilon}W$ and
$\partial_{\delta}V$ are each decorated by
the trivial element $1\in\pi_1X$. This determines elements
$a$ and $b$ decorating the other edges of the two respective trees. See Figure~\ref{W-disk-boundary-int-fig}.
Define the signed tree $\epsilon_p\cdot t_p$
associated to such a $p$ by:
\begin{equation}\label{eq:boundary-int-tree}
\epsilon_p\cdot t_p:=-\epsilon \delta \cdot(a^\epsilon,b^\delta)
\end{equation}
where $\epsilon_p,\epsilon,\delta\in\{+,-\}\widehat{=}\{+1,-1\}$. The reason for the minus sign in Equation~(\ref{eq:boundary-int-tree}) will be made clear in Figure~\ref{fig:cancelling-boundary-push} below. 

\begin{figure}[h]
\centerline{\includegraphics[scale=.35]{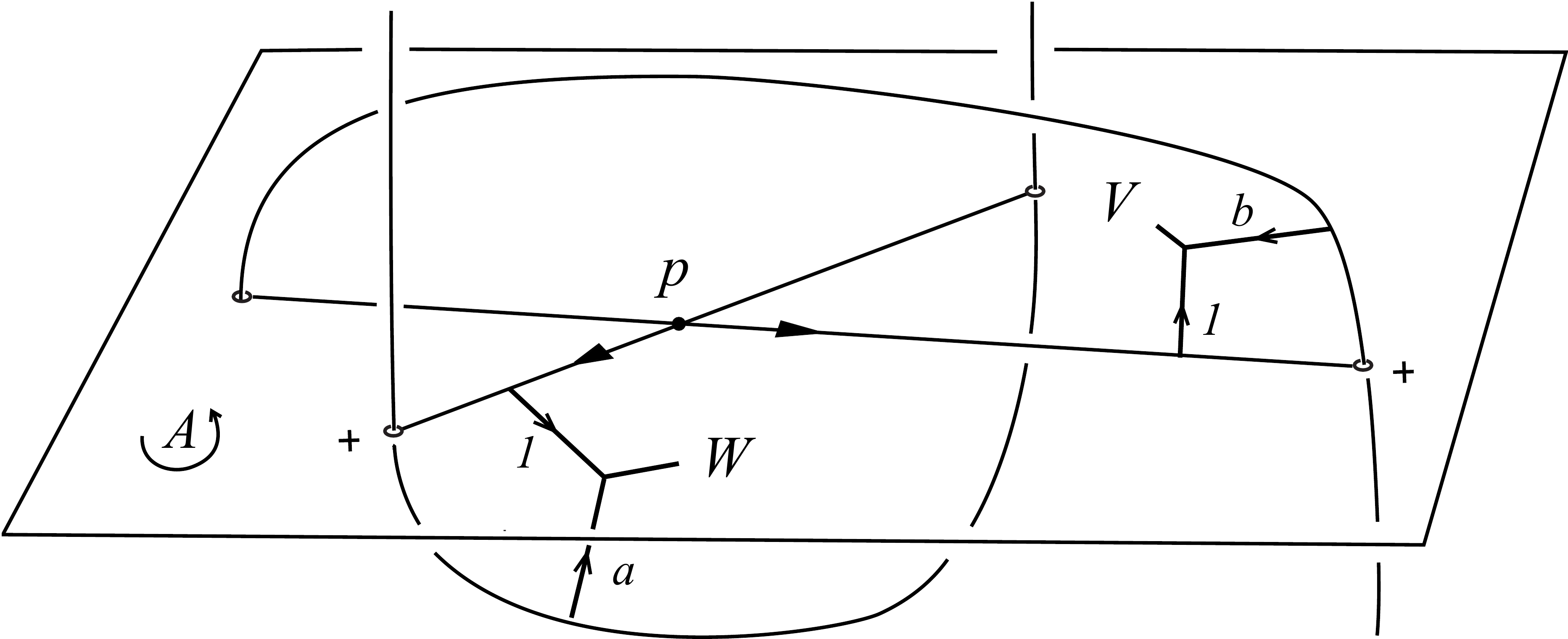}}
         \caption{The case $\epsilon=+=\delta$ in Equation~(\ref{eq:boundary-int-tree}): The illustrated intersection point
         $p\in\partial_{\epsilon}W \cap\partial_{\delta}V$
         between Whitney disk boundaries is assigned the signed tree $\epsilon_p\cdot t_p=-\epsilon\delta(a^\epsilon,b^\delta)=-(a,b)$.}
         \label{W-disk-boundary-int-fig}
\end{figure}

One can check that this definition of $t_p$ does not depend on the
choices made (Exercise~\ref{ex:boundary-int-def-well-defined}). The extended version of $\tau_1(A)$ is defined by
including all such $t_p$ in the sum. Since all boundary intersections
can be eliminated by finger moves which create interior
intersections having the exact same trees (Figure~\ref{W-disk-boundary-int-push-off-fig}), this extended definition can always be reduced to the
original one.
\begin{figure}[h]
\centerline{\includegraphics[scale=.35]{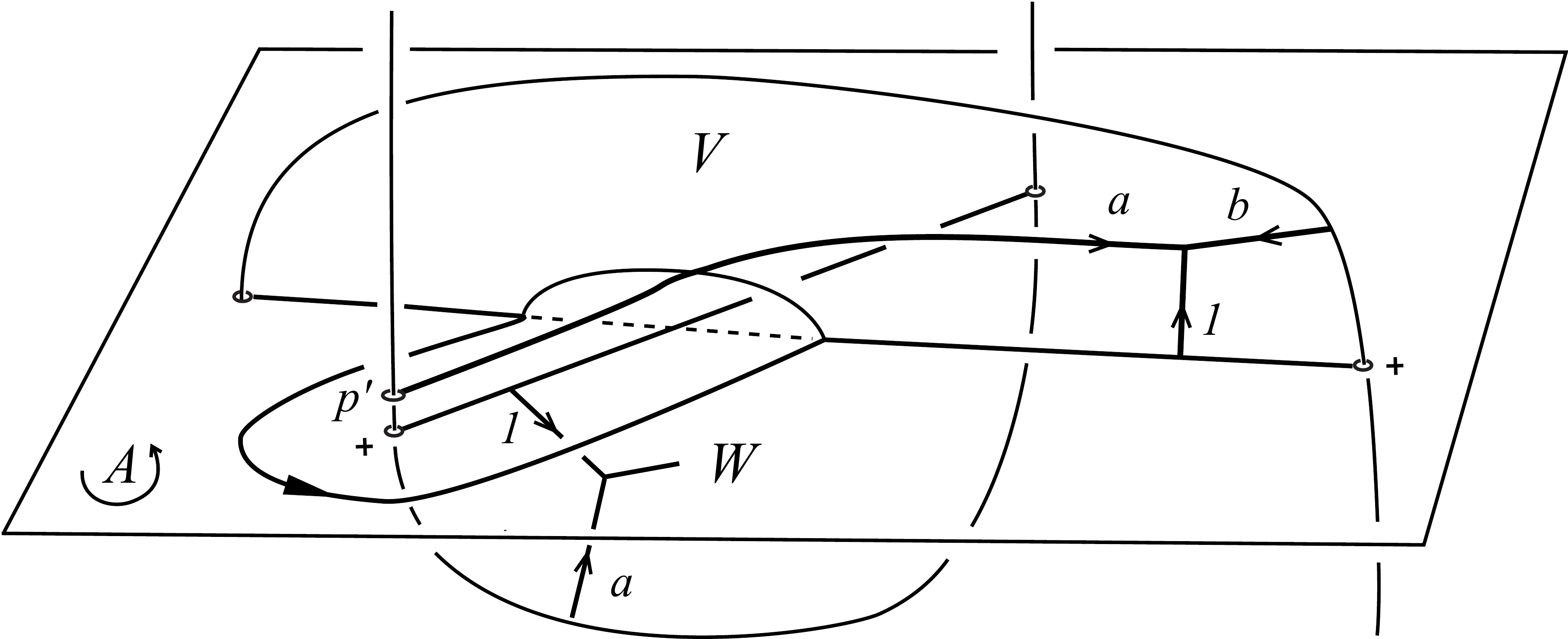}}
         \caption{Eliminating
         $p\in\partial_+W \cap\partial_+V$ from Figure~\ref{W-disk-boundary-int-fig} by pushing across the positive self-intersection paired by $W$
          creates a positive intersection $p'\in A\pitchfork V$ with signed tree $\epsilon_{p'}\cdot t_{p'}=+(b,a)=-(a,b)=\epsilon_p\cdot t_p$.}
         \label{W-disk-boundary-int-push-off-fig}
\end{figure}

Properly interpreted, the formula assigning $t_p$ to
$p\in\partial W\cap\partial V$ also works when $W=V$. For instance, for $p\in\partial_-W\cap\partial_+W$
such that the orientation of $A$ is equal to
$(\overrightarrow{\partial_{-} W},\overrightarrow{\partial_{+}
W})_p$, then $\epsilon_p\cdot t_p=-(a,a^{-1})$, where the
rooted tree associated to $W$ has decorations $1$ and $a$ on the edges
dual to the $\partial_+W$ and $\partial_-W$ boundary arcs respectively. (Exercise~\ref{ex:boundary-int-def-well-defined}.) 

The proof of independence of Whitney disk boundaries now goes as
follows. For a fixed choice of pairings of self-intersections induced by a given collection of Whitney disks, any
other configuration of Whitney disk boundaries can be achieved by a
regular homotopy of (collars of) the given Whitney disk boundaries,
fixing the self-intersection points of $A$ (Clarification: we mean here
that this regular homotopy is induced by a regular homotopy of the
inverse images of the Whitney disk boundaries in the domain of
$A$, and extends to a regular homotopy of collars of the Whitney
disks in $X$). During such a homotopy, $\tau_1$ does not change
since boundary intersections come and go in canceling pairs, or
accompanied by a canceling interior intersection when pushing
over a self-intersection point of $A$, see Figure~\ref{fig:cancelling-boundary-push} for one case and Exercise~\ref{ex:cancelling-boundary-push-cases} for the others.
(This step uses the fact that the domain of $A$ is simply connected.)
\begin{figure}[h]
\centerline{\includegraphics[scale=.35]{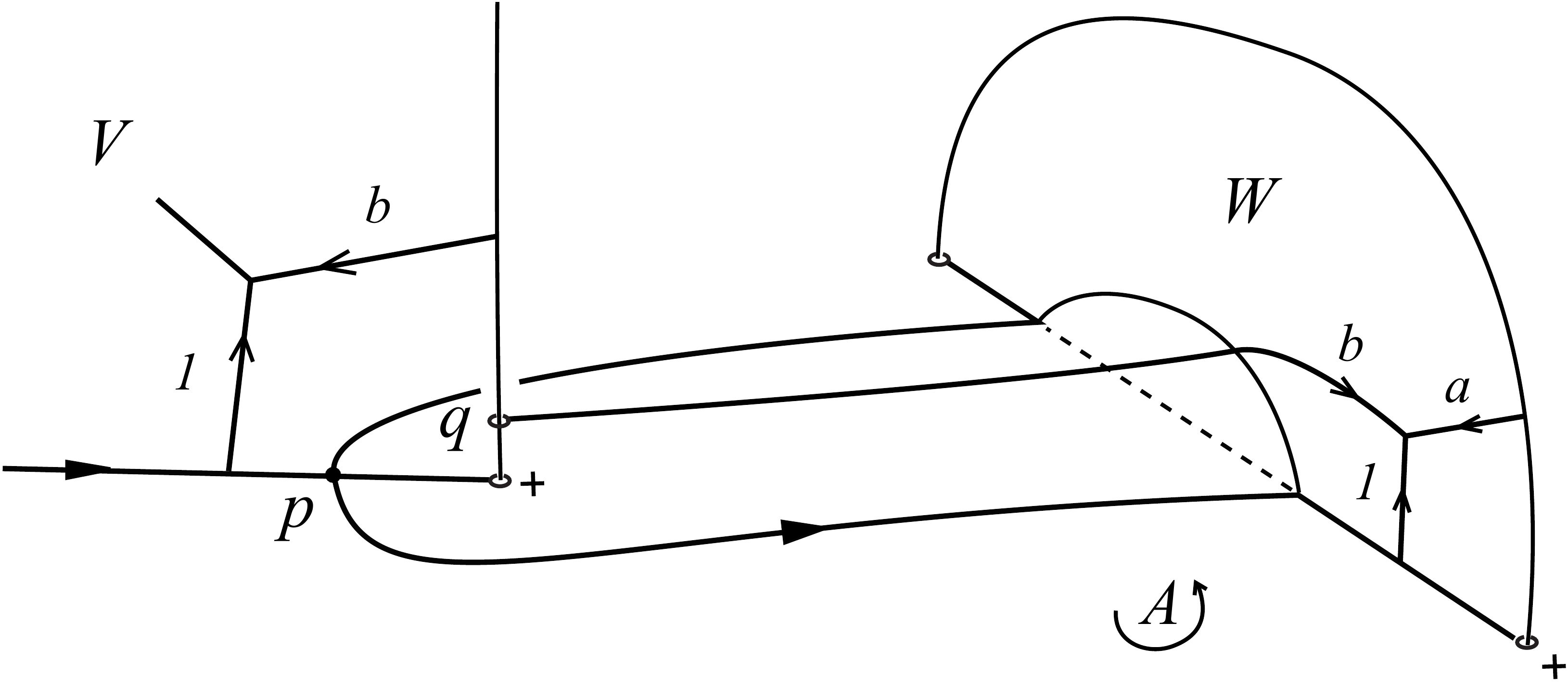}}
         \caption{Pushing $\partial_+W$ into $\partial_+V$ across the positive self-intersection of $A$ paired by $V$ creates
         $p\in\partial_+W \cap\partial_+V$ and $q\in A\pitchfork W$ with algebraically cancelling signed trees $\epsilon_{p}\cdot t_{p}=-(a,b)=-\epsilon_q\cdot t_q$.}
         \label{fig:cancelling-boundary-push}
\end{figure}

\subsubsection{Independence of pairings of self-intersections:}\label{subsubsec:pairings} 
Let $W$ be a Whitney disk in $\cW$ pairing self-intersections $p$ and $q$, and let $W'$ be another Whitney disk in $\cW$ pairing self-intersections $p'$ and $q'$, such that all these self-intersections determine the same element of $\pi_1X$.
Then there exists a framed Whitney disk $V'$ pairing $p'$ and $q$ such that $\partial V'$ is disjoint from all singularities in $\cW$ (including Whitney disk boundaries, cf.~Figure~\ref{W-disk-boundary-int-push-off-fig}). Now a framed Whitney disk $V$ pairing $p$ and $q'$ can be constructed from $W$ and $W'$ (minus small corners) and a parallel of $V'$ (with the orientation reversed) as illustrated in Figure~\ref{pairing-choice-change-fig}. 

Replacing $W$ and $W'$ with $V$ and $V'$ in $\cW$ does not change $\tau_1(A)$ since the 
order~1 decorated trees corresponding to the intersections $A\pitchfork V'$ cancel with those corresponding to the oppositely-signed parallel copy of $V'$ in $V$
(Exercise~\ref{ex:pairing-choice}). Any choice of pairings can be achieved by iterating this construction.

\begin{figure}[h]
\centerline{\includegraphics[scale=.25]{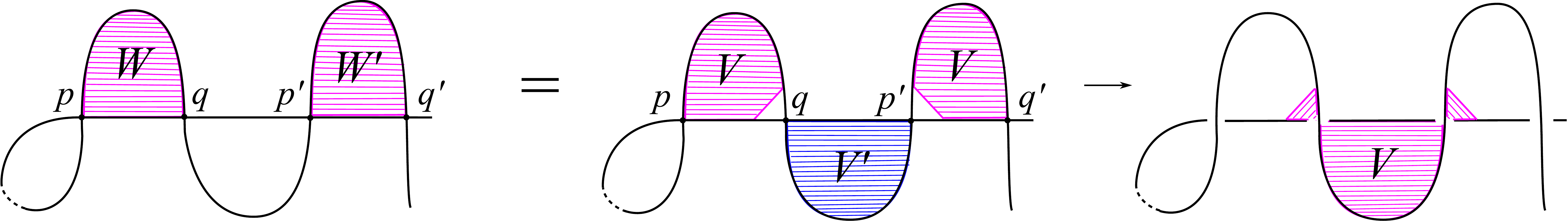}}
         \caption{Transverse intersections between Whitney disks and $A$ are not shown. This picture is otherwise accurate up to diffeomorphism since $W$, $W'$ and $V'$ can be assumed to be framed and embedded. The dotted sub-arc of the doublepoint loop for $p$ is understood to extend outside the $4$--ball of the figure.}
         \label{pairing-choice-change-fig}
\end{figure}

\subsubsection{The INT relation for $\RP^2$s}\label{subsubsec:RP2-INT}
For use in the next step of the proof, we clarify here the INT relation $(a,\lambda_0(A, R) +\omega_2(R)\cdot 1 )=0\in\tilde{\cT}_1$ for the case where $a^2=1$ and $R:\RP^2\imra X$.

The pairing $\lambda_0(A,R)$ sums the signed group elements $\epsilon_p\cdot g_p$ associated to each $p\in A\pitchfork R$ as in section~\ref{subsec:classical-mu-lambda}, except that because $\RP^2$ is neither simply connected nor orientable now both $\epsilon_p$ and $g_p$ depend on the choice of sheet-changing loop through $p$. Let $d$ be a sheet-changing loop through $p$, and let $d_R$ be the path $d\cap R$. Then $g_p=[d]\in\pi_1X$ is the group element associated to $p$ using $d$, and by definition the sign $\epsilon_p$ is gotten by transporting a local orientation of $R$ from the basepoint of $R$ back along $d_R$ to $p$, where it is paired with the orientation of $A$ at $p$ for comparison with the orientation of $X$. Any different choice of $d_R$ changes $g_p$ by right multiplication by $a^n$ and changes $\epsilon_p$ by multiplication by $(-1)^n$ (Exercise~\ref{ex:RP2-lambda0-indeterminacy}), so $\lambda_0(A,R)\in\Z[\pi_1X]$ is only well-defined up to the relations $g=-ga$, for all $g\in\pi_1X$.

In the setting of the INT relations, the local orientation of $R$ comes from the orientation of the Whitney disk that $R$ has been tubed into,
and we have $(a,g)=-(a,ga)$ by the HOL and AS relations. So $(a,\lambda_0(A, R))$ is well-defined in $\tilde{\cT}_1$.

As for $2$-spheres, the second Stiefel--Whitney number $\omega_2(R)\in\Z_2$ is computed as $|R\pitchfork R'|$ mod $2$, where $R'$ is a parallel copy of $R$. (Recall that $(a,\omega_2(R)\cdot 1)$ is 2-torsion in $\tilde{\cT}_1$.)


\subsubsection{Independence of sheet choices:}\label{subsubsec:sheet-choices} 
For $A:S^2\imra X$, let $p$ and $q$ be a positive and a negative transverse self-intersection of $A$, and
denote the preimages by $A^{-1}(p)=\{x,x'\}\subset S^2$, and $A^{-1}(q)=\{y,y'\}\subset S^2$. 

If $p$ and $q$ have common group element $a$,
then any Whitney disk $W$ pairing $p$ and $q$ induces a pairing of $\{x,x'\}$ with $\{y,y'\}$ since each arc of $\partial W$ runs between a sheet of $A$ around $p$ and a sheet of $A$ around $q$. 

Framed Whitney disks exist for \emph{both} of the two pairing choices $x\leftrightarrow y, x'\leftrightarrow y'$ and $x\leftrightarrow y', x'\leftrightarrow y$ if and only if $a^2=1\in\pi_1X$ (Exercise~\ref{ex:sheet-choice-w-disks-exist}).
So we need to show that for any $a$ such that $a^2=1$, $\tau_1(A)$ does not depend on these \emph{choices of preimage pairings}, also called \emph{choices of sheets} (\cite[Sec.4]{ST4}). 

\begin{figure}[h]
\centerline{\includegraphics[scale=.2]{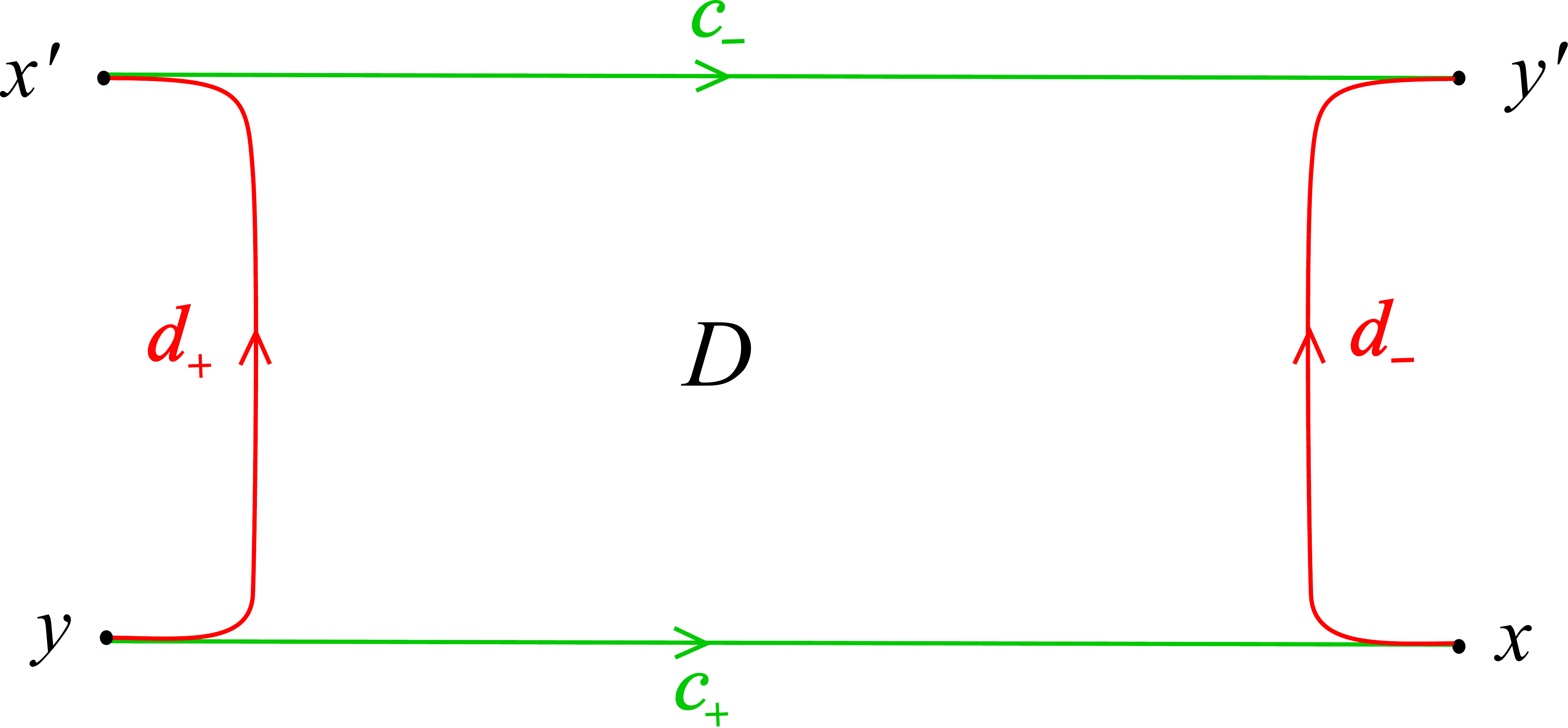}}
         \caption{In $S^2$: The quadrilateral $r=c_+\cup d_-\cup c_-\cup d_+$ bounds the 2-cell $D$.}
         \label{fig:sheet-choice-preimage-arcs}
\end{figure}

Let $W$ and $V$ be framed Whitney disks corresponding to the two
ways of pairing the preimages of such a pair $p$ and $q$ having group element $a$ with $a^2=1$. 
The union of the preimages of the Whitney disk boundary arcs $A^{-1}(\partial W)\cup A^{-1}(\partial V)$ is a quadrilateral $r$
in $S^2$ which is the union of the two pairs of arcs
$c_{\pm}:=A^{-1}(\partial_{\pm}W$) and
$d_{\pm}:=A^{-1}(\partial_{\pm}V)$
(see Figure~\ref{fig:sheet-choice-preimage-arcs}). Since $\tau_1(A)$ does not depend on the choices of boundaries of $W$ and $V$,
we may arrange that $r$ is embedded
and bounds a 2-cell $D$ in $S^2$ such that $A$
restricts to an embedding on $D$.
(For later use we have chosen $\partial V$ to coincide with $\partial W$ near $p$ and $q$, as indicated in Figure~\ref{fig:sheet-choice-preimage-arcs}.)

The union $A(D)\cup W\cup V$ defines the image of a map $R:\RP^2\to X$ representing $a$ (Exercise~\ref{ex:RP2}).  
This $R$ contains the transverse intersections $A\pitchfork W$ and $A\pitchfork V$, but $R$ is not transverse to $A$ along $A(D)$, and $R$ is not smooth along $\rho:=A(r)$. 
We will describe a perturbation of $R$ to a smoothly immersed $R':\RP^2\imra X$ transverse to $A$ such that the perturbation creates
exactly new intersections $R'\pitchfork A$ which correspond to $\omega_2(R')$. Then then the change in $\tau_1(A)$ due to replacing $W$ by $V$ will be given by $(a,\lambda_0(A, R') +\omega_2(R')\cdot 1 )$.

The perturbation of $R$ to $R'$ uses the following lemma, which will be proved below:
\begin{lem}\label{lem:rho-prime}
There exists a push-off $\rho'$ of $\rho$ such that $\rho'$ is normal to $A$ and $\rho'$ restricts to Whitney parallels of $\partial W$ and $\partial V$. 
\end{lem}
See Remark~\ref{subsec:general-whitney-section} for an explanation of ``Whitney parallel''. 

Given $\rho'$ as in Lemma~\ref{lem:rho-prime}, we define $R':\RP^2\imra X$ as follows.
Extend $\rho'$ across $W$ and $V$ to get Whitney parallels $W'$ and $V'$.
Extend $\rho'$ generically across $A(D)$ to get $A(D)'$ which is transverse to $A(D)$.

This defines the image $W'\cup V'\cup A(D)'$ of a map $\RP^2\to X$.
Now smooth the corners of $W'\cup V'\cup A(D)'$ near $\rho'$ by a perturbation supported away from $A$
to get $R':\RP^2\imra X$.

We have $A\pitchfork R'=(A\pitchfork W')\cup (A\pitchfork V')\cup (A\pitchfork A(D)')$.

By construction, the intersections $A\pitchfork W'$ are all parallel to the intersections $A\pitchfork W$, and  
the intersections $A\pitchfork V'$ are all parallel to the intersections $A\pitchfork V$.
Also $A\pitchfork A(D)'=A(D)\pitchfork A(D)'$, since $A$ restricts to an embedding on $D$.

We claim that $|A(D)\pitchfork A(D)'|\equiv\omega_2(R')$ mod $2$. To check this claim, we can use a smooth perturbation of the topological $R:\RP^2\to X$ above as a parallel of $R'$, and compute
$\omega_2(R')=|R\pitchfork R'|$ mod $2$, as in section~\ref{subsubsec:RP2-INT}. 
Recall that $R=A(D)\cup W\cup V$, and since we now are not concerned with controlling intersections with $A$ we can take any small perturbation near $\rho$ to make $R$ smooth.   

Since $W$ and
$V$ are framed, and $R'$ restricts to Whitney parallels of $\partial W$ and $\partial V$,
we have $W\pitchfork W'\equiv 0\equiv V\pitchfork V'$ mod 2.
So the only possible contributions to $|R\pitchfork R'|$ mod $2$ come from $A(D)\pitchfork A(D)'$ as claimed.

The group element associated to any point in $A(D)\pitchfork A(D)'=A\pitchfork A(D)'\subset A\pitchfork R'$ is $1$ or $a$, so the change in $\tau_1(A)$ due to replacing $W$ by $V$ is
$(a,\lambda_0(A,R')+\omega_2(R')\cdot 1)$, since $(a,1)=(a,a)$ is $2$-torsion.

To show independence of sheet choice it just remains to prove Lemma~\ref{lem:rho-prime}.
\begin{proof}[Proof of Lemma~\ref{lem:rho-prime}] 

We want to define a push-off $\rho'$ of $\rho=\partial_+W\cup \partial_-V\cup-\partial_-W\cup-\partial_+V$ which is normal to $A$ and restricts to Whitney parallels of $\partial W$ and $\partial V$.

Consider the normal disk-bundle $\nu_XW|_{\partial W}$ which is the restriction to $\partial W$ of the normal disk-bundle $\nu_XW$ of $W$.
Denote by $A_\pm\subset \nu_XW|_{\partial W}$ the sheets corresponding to $A$ with $\partial_-W\subset A_-$ and $\partial_+W\subset A_+$.
Figure~\ref{W-section-over-int-pts-fig} shows an embedding of $\nu_XW|_{\partial W}\cong S^1\times D^2$ into $3$-space, as in Figure~\ref{W-subspaces}.

At $p$ and at $q$, the two sheets $A_{\pm}$ split each of $\nu_XW|_{\partial W}(q)$ and $\nu_XW|_{\partial W}(p)$ into four quadrants. 
We define a preferred quadrant at $p$ and at $q$ as follows. The orientation of $A$ induces orientations of the $A_{\pm}$, which in turn induce orientations of 
$\partial A_-$ and $\partial A_+$ via the usual convention. At the positive self-intersection $p$ the preferred quadrant is bounded by the vectors $\overrightarrow{\partial A_+}$ and $\overrightarrow{\partial A_-}$, and at the negative self-intersection $q$ the preferred quadrant is bounded by the vectors $-\overrightarrow{\partial A_+}$ and $-\overrightarrow{\partial A_-}$. 

Define $\rho'(p)$ and $\rho'(q)$ by choosing a vector in each of these preferred quadrants (see left of Figure~\ref{W-section-over-int-pts-fig}).

Extend $\rho'(p)$ and $\rho'(q)$ to a section $\rho'(\partial W)$ of $\nu_XW|_{\partial W}$ over $\partial W$ such that $\rho'(\partial W)$ is normal to both sheets of $A$ along $\partial W$, as in the right of Figure~\ref{W-section-over-int-pts-fig}. Note that this section exists by our choices of preferred quadrants.

\begin{figure}[h]
\centerline{\includegraphics[scale=.275]{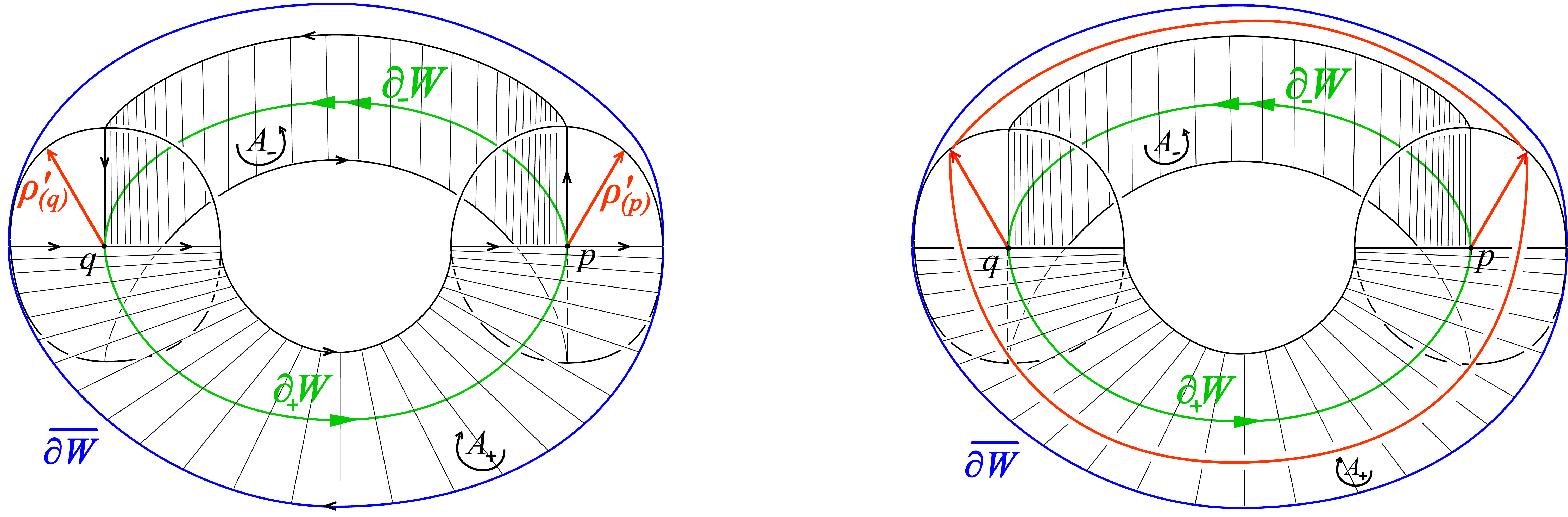}}
         \caption{In $\nu_XW|_{\partial W}$: On the left, $\rho'(q)$ and $\rho'(p)$ in the preferred quadrants, and on the right, extended to a red Whitney section $\rho'(\partial W)$ which is normal to $A$. The blue `standard' Whitney section $\overline{\partial W}$ is shown as a reference with Figure~\ref{W-subspaces}.}
         \label{W-section-over-int-pts-fig}
\end{figure}

To define $\rho'$ over $\partial V$, note that
we may assume that $V=W$ near both $p$ and $q$ (Exercise~\ref{ex:V=W-near-p-and-q-for-rho-prime-lemma}),
so  $\nu_XV|_{\partial V}$ coincides with $\nu_XW|_{\partial W}$ near $p$ and $q$.
This means that the preferred quadrants also coincide, and the previously defined $\rho'(p)$ and $\rho'(q)$ can be extended along $\partial V$ to a section $\rho'(\partial V)$
which is normal to the sheets of $A$.

So along $\rho=\partial_+W\cup \partial_-V\cup-\partial_-W\cup-\partial_+V$ we have the push-off $\rho'$ assembled from the push-offs along the given oriented subintervals starting at $\rho(q)$: The section $\rho'(\partial_+W)$ ending at $\rho'(p)$ which is the restriction to $\partial_+W$ of $\rho'(\partial W)$, followed by the section $\rho'(\partial_-V)$ which is the restriction to $\partial_-V$ of $\rho'(\partial V)$ running from $\rho'(p)$ back to $\rho'(q)$, 
followed by the section $\rho'(-\partial_-W)$ which is the restriction to $\partial_-W$ of $\rho'(\partial W)$ but running from $\rho'(q)$ back to $\rho'(p)$,
followed by the section $\rho'(-\partial_+V)$ which is the restriction to $\partial_+V$ of $\rho'(\partial V)$ but running from $\rho'(p)$ back to $\rho'(q)$.
\end{proof}

We have so far shown that $\tau_1(A)\in\widetilde{\cT}_1/\textrm{INT}(A)$ is independent of the choice of order~1 framed Whitney tower for a fixed immersion $A$.



\subsubsection{Invariance of $\tau_1(A)$ under homotopies of $A$}\label{subsubsec:htpy-invariance-tau1}
It suffices to check invariance under regular homopy by \cite[Thm.1.2]{PRT}.
Suppose that $\tau_0(A)$ vanishes, and $A$ is regularly homotopic to $A'$. By the homotopy invariance of $\tau_0$ it follows that $A'$ also supports an order~1 framed Whitney tower.
Recall that up to isotopy, any generic regular homotopy from $A$ to $A'$ can be realized as a sequence of finitely many finger moves followed by finitely many Whitney moves. 

An isotopy from $A$ to $A'$ extends to any Whitney tower $\cW$ on $A$ to yield a Whitney tower $\cW'$ on $A'$ with identical intersection forests 
$t(\cW)=t(\cW')$, so $\tau_1$ is invariant under isotopy. 

Since any Whitney move has a finger move as an ``inverse'', there exists $A''$ which differs from each of $A$ and $A'$ by only finger moves (up to isotopy). Since a finger move is supported near an arc, it can be made disjoint from the Whitney disks in any pre-existing Whitney tower by a small isotopy, and the pair of intersections created by a finger move admit a local clean Whitney disk disjoint from any other Whitney disks. So any Whitney tower on $A$ or $A'$ gives rise to a Whitney tower on $A''$ yielding $\tau_1(A)=\tau_1(A'')=\tau_1(A')$, since $\tau_1$ does not depend on the choice of Whitney tower on $A''$.
Here we are also using that $\textrm{INT}(A)=\textrm{INT}(A')=\textrm{INT}(A'')$, since $A$, $A'$ and $A''$ are all homotopic.

\subsection{$\tau_1(A)$ vanishes if and only if $A$ supports an order~2 framed Whitney tower}\label{subsec:tau1-vanishes-equals-order-2-w-tower}
Here we show the equivalence of statements~(\ref{thm-item:lambda1-vanishes}) and (\ref{thm-item:A-admits-order2-tower}) in Theorem~\ref{thm:tau1-vanishes}.

If $A$ supports an order~2 framed Whitney tower $\cV$, then $\tau_1(A)$ vanishes since $\cV$ is also an order~1 framed Whitney tower, and 
$\cV$ contains no unpaired order~1 intersections.

So assume that $\tau_1(A)$ vanishes, and let $\cW$ be an order~1 framed Whitney tower on $A$.
The vanishing of $\tau_1(A)=[t(\cW)]\in\widetilde{\cT_1}/\mbox{INT}$ means that the intersection forest $t(\cW)=\sum \epsilon_p\cdot t_p$ lies in the span of the AS, HOL, FR and INT relators.
Here we are considering $t(\cW)$ as a word in the span of decorated order~1 trees.

The construction of an order~2 framed Whitney tower on $A$ will involve two main steps: 
First, geometric realizations of the relators will be used to modify $\cW$ so that $t(\cW)$ consists of pairs $\pm t_p$ of oppositely-signed isomorphic trees
corresponding to \emph{algebraically cancelling pairs} of intersections. 
Then, further controlled modifications of $\cW$ will convert these algebraically cancelling pairs into \emph{geometrically cancelling pairs} of intersections admitting Whitney disks.  
The issue here is that an algebraically cancelling pair of intersections may lie in different Whitney disks (Figure~\ref{fig:transfer-move-Before-1}), so achieving geometric cancellation will require ``transferring'' intersections from one Whitney disk to another.

In the following $\cW$ will not be renamed as controlled modifications are made.

\subsubsection{}\label{subsubsection:towards-alg-cancellation} 
\textbf{Towards algebraic cancellation:} Perform a finger move on $A$ guided by a circle representing any $a\in\pi_1X$, and then tube the resulting clean local framed Whitney disk into a $2$-sphere $S$. The resulting Whitney disk $W$ has twisting $\omega(W)\equiv \omega_2(S)$ mod $2$. If $\omega_2(S)=0$, then after performing some interior twists on $W$ it can be arranged that $\omega(W)=0\in\Z$. If $\omega_2(S)=1$, then after performing one boundary-twist and some interior twists on $W$ it can be arranged that $\omega(W)=0\in\Z$.
So after recovering the framing on $W$ it follows that $[t(\cW)]\in\widetilde{\cT_1}$ is changed exactly by adding the INT relator $(a,\lambda_0(A,S)+\omega_2(S)\cdot 1)$. 

In the case where $a^2=1$, a similar but more complicated procedure changes $[t(\cW)]\in\widetilde{\cT_1}$ exactly by adding $(a,\lambda_0(A,R)+\omega_2(R)\cdot 1)$, where $R:\RP^2\imra X$ represents $a$. This will be shown below in Lemma~\ref{lem:realize-RP2-INT}.

By realizing INT relations in these ways we can arrange that $t(\cW)$ lies in the span of AS, HOL and FR relators.
To realize any FR relator $(1,a)+(a,a)$ perform a finger move guided by a circle representing $a$ to get a clean local framed Whitney disk $W$.
Then performing two opposite boundary twists around each boundary arc of $W$ changes $t(\cW)$ exactly by $(1,a)+(a,a)$.

So by realizing INT and FR relators we can arrange that $t(\cW)$ lies in the span of AS and HOL relators. By re-choosing whiskers
on the trivalent vertices of trees to realize HOL relations we can further arrange that $t(\cW)$ lies in the span of AS relators.

At this point it is convenient to \emph{split} $\cW$ so that each Whitney disk in $\cW$ contains exactly 
one interior intersection with $A$. This splitting is accomplished by finger moves guided by arcs the Whitney disks running from one boundary arc to the other, and is the easiest case of Lemma~\ref{lem:split-w-tower} (see also Figure~\ref{split-w-tower-with-trees-fig}).
Now any tree $(a,b)\in t(\cW)$ can be changed to $-(a,b)\in t(\cW)$ simply by re-choosing the orientation of the corresponding Whitney disk.

As a result of these constructions we can assume that now the un-paired intersections in $\cW$ consist of algebraically cancelling pairs.

\begin{figure}[h]
\includegraphics[scale=.4]{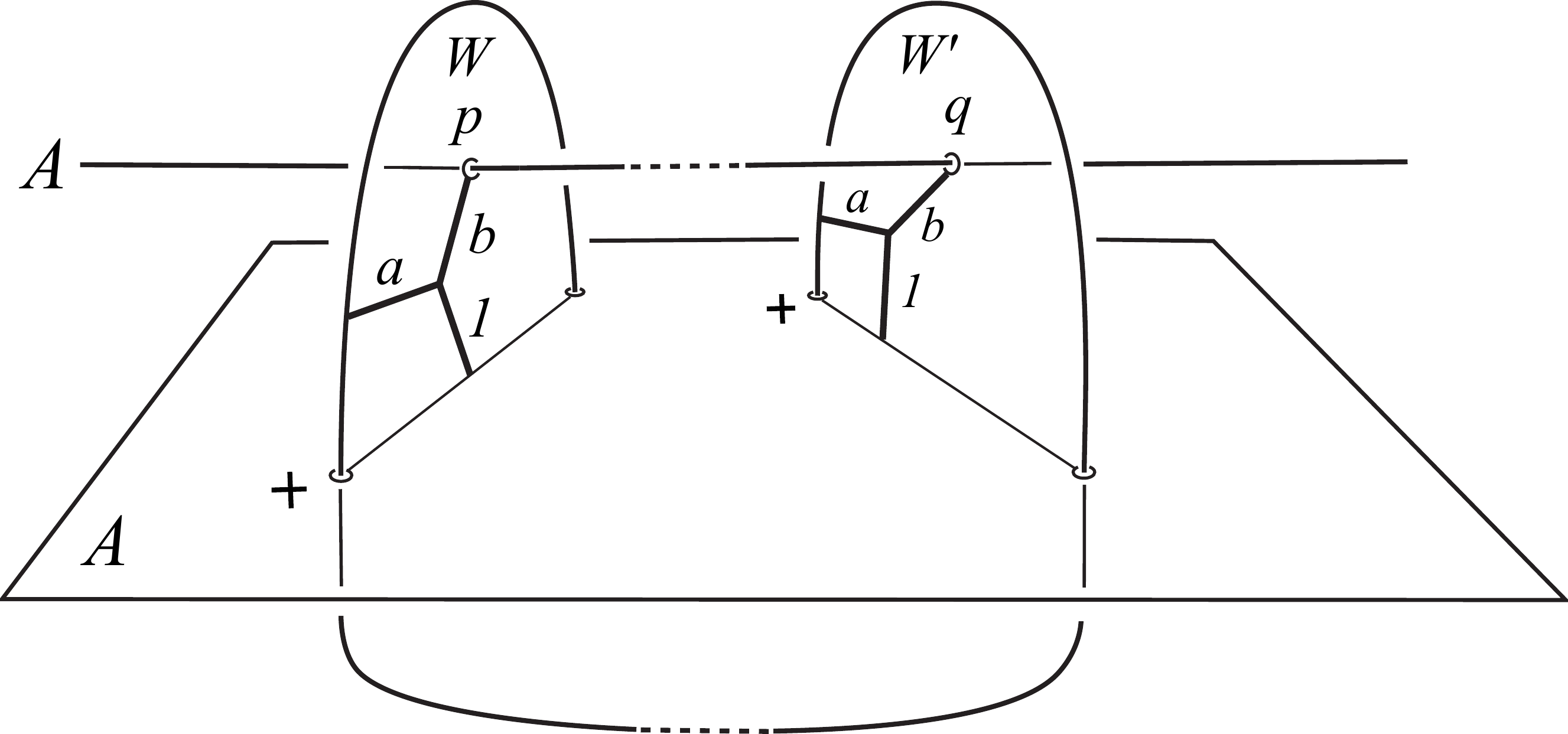}
\caption{}\label{fig:transfer-move-Before-1}
\end{figure}
\subsubsection{}\label{subsubsection:towards-geo-cancellation}
\textbf{Towards geometric cancellation:} 
Consider an algebraically cancelling pair of intersections $p=A\pitchfork W$ and $q=A\pitchfork W'$, with $\epsilon\cdot t_p=(a,b)$ and $\epsilon\cdot t_q=-(a,b)$, where the $b$-decorated edge changes sheets at $p$ and $q$ respectively. See Figure~\ref{fig:transfer-move-Before-1}, where the short dashed sub-arcs indicate where sheets extend outside the $4$-ball shown in the figure. 

We will describe a controlled modification of $\cW$ that ``transfers'' $p$ over to $W'$ so that $p$ and $q$ admit a framed order~2 Whitney disk. This ``transfer move'' will not create any new unpaired order~1 intersections, and can be iteratively applied to convert all algebraically cancelling pairs into geometrically cancelling pairs, yielding the desired order~2 framed Whitney tower on $A$. Description of this modification will be accompanied by Figures~\ref{fig:transfer-move-Before-1} through Figure~\ref{fig:transfer-move-2}, with some details left to the exercises.

\begin{figure}[h]
\includegraphics[scale=.4]{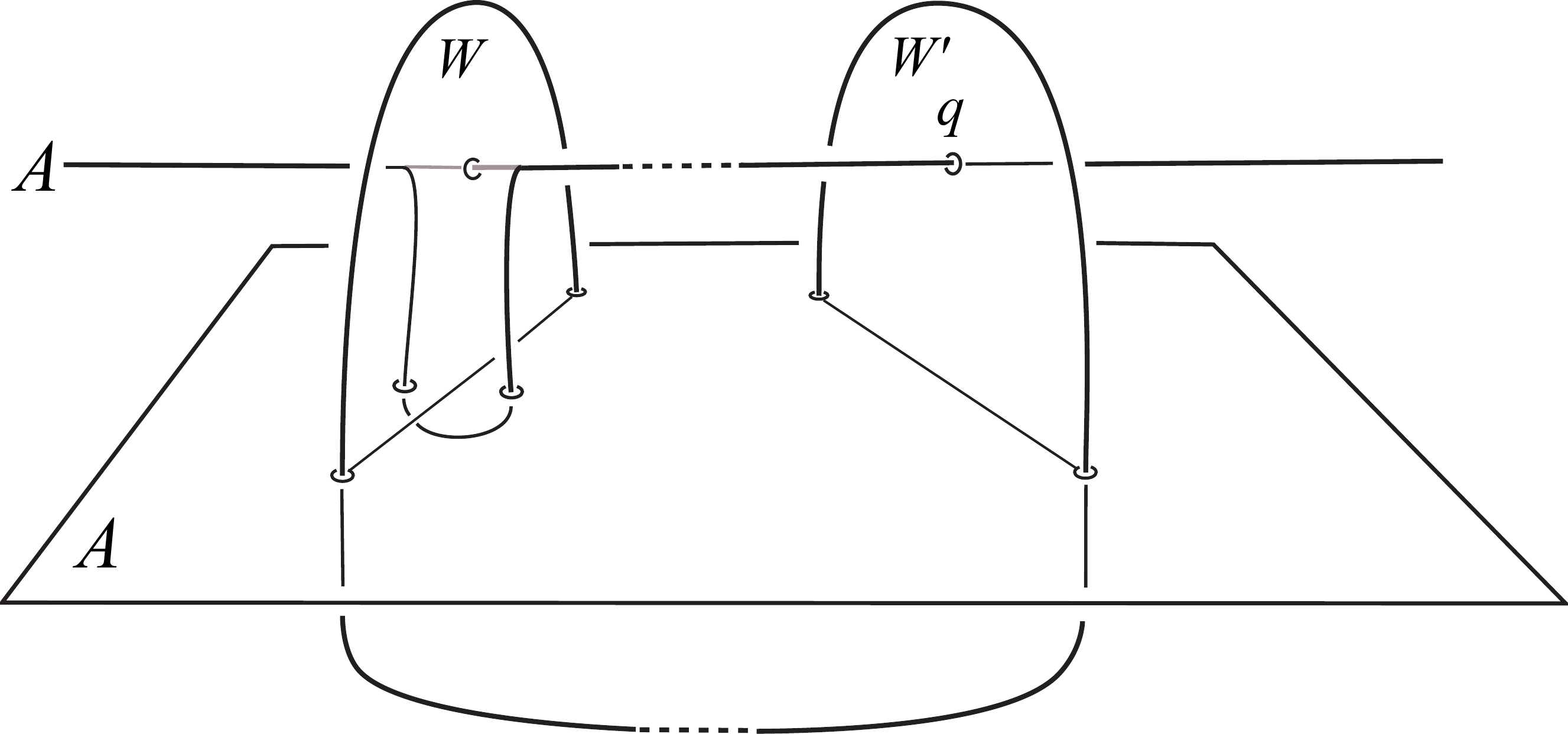}
\caption{}\label{fig:transfer-move-first-finger-move}
\end{figure}
The transfer move starts by pushing $p$ off of $W$ by a finger move into $A$ as shown in Figure~\ref{fig:transfer-move-first-finger-move}. This finger move is guided by an arc in $W$ along the $b$-labeled and $1$-labeled edges of $t_p$ from $p$ to $\partial_+W$.
\begin{figure}[h]
\includegraphics[scale=.4]{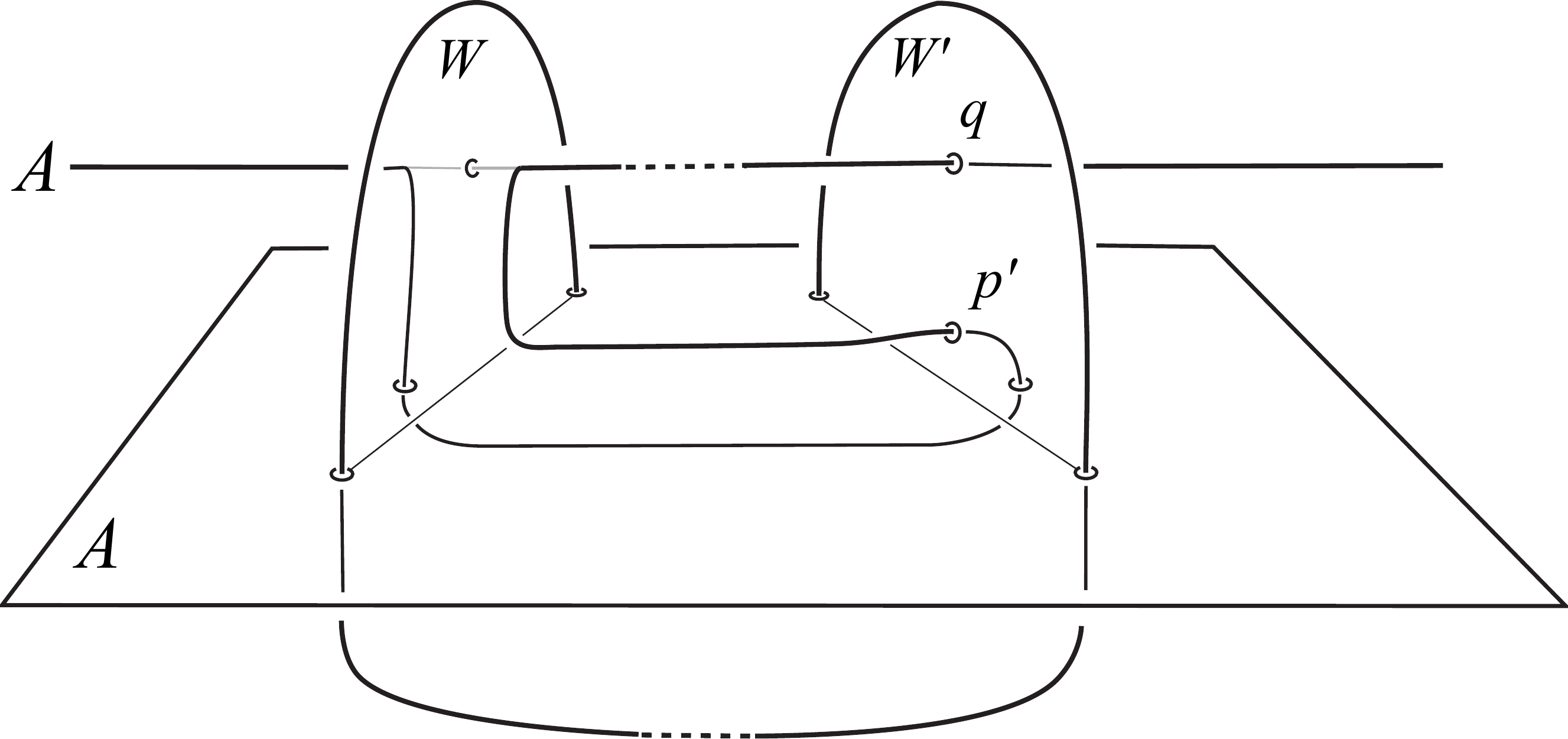}
\caption{}\label{fig:transfer-move-1}
\end{figure}
\begin{figure}[h]
\includegraphics[scale=.4]{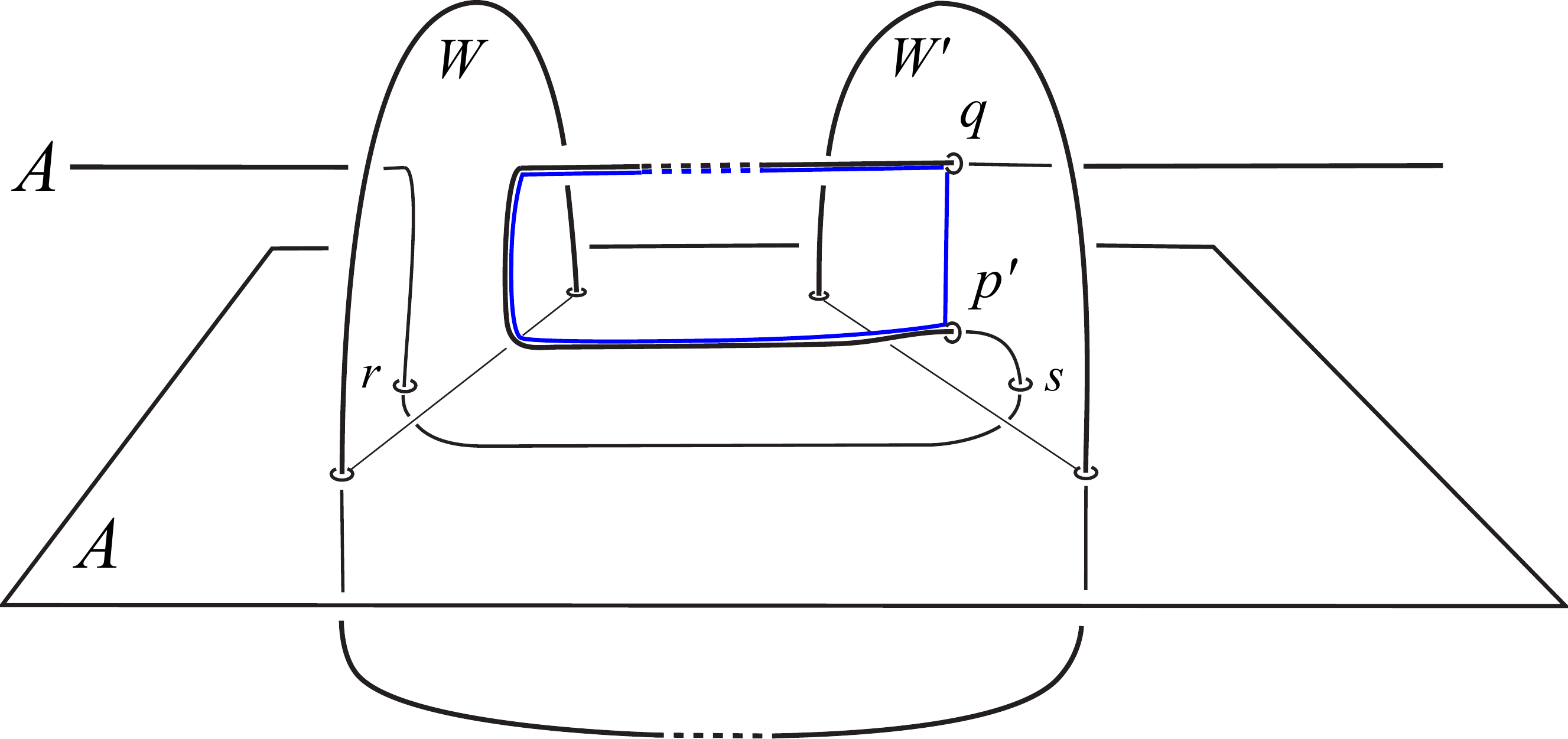}
\caption{}
\label{fig:transfer-move-1-with-order-2-boundary}
\end{figure}
\begin{figure}[h]
\includegraphics[scale=.45]{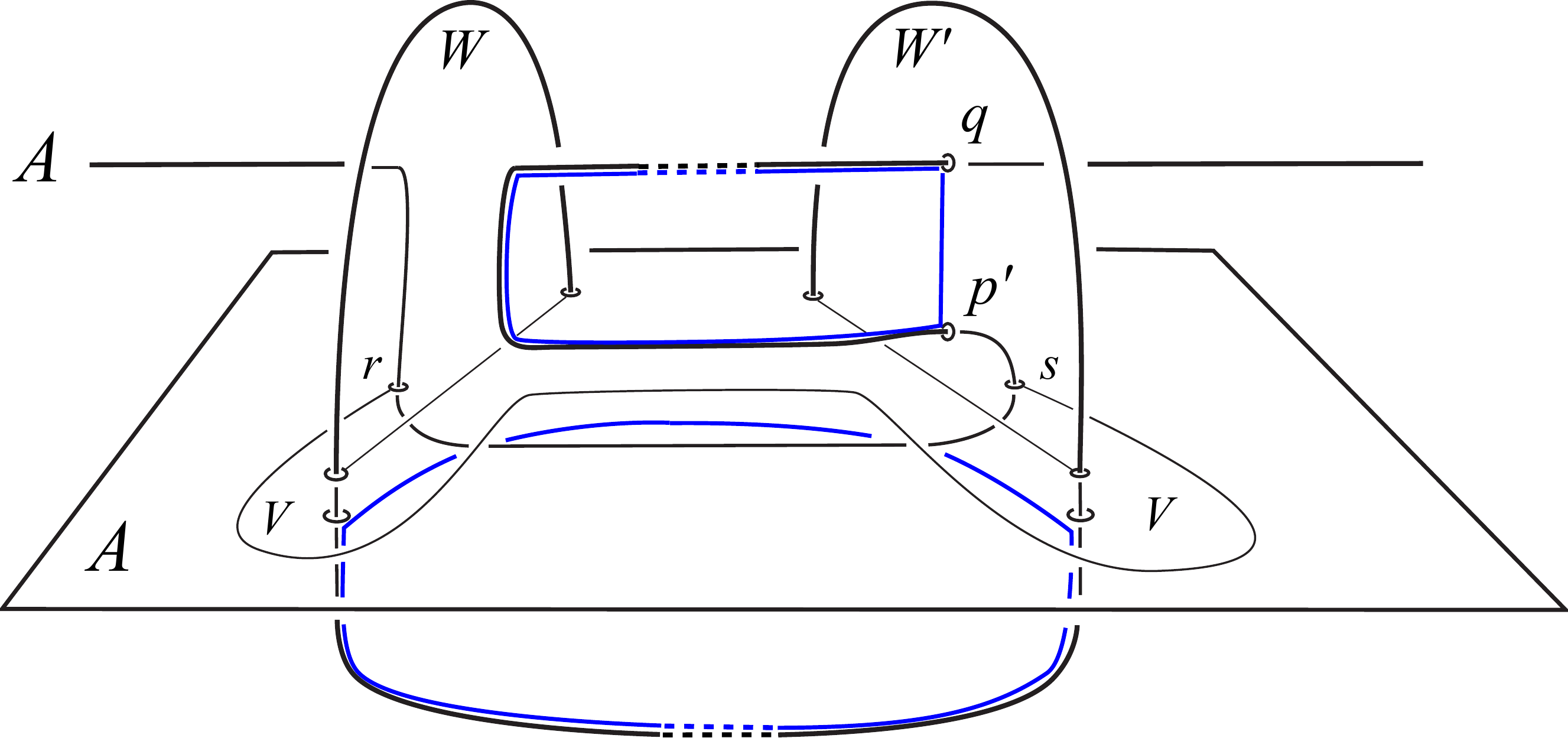}
\caption{}
\label{fig:transfer-move-2}
\end{figure}

The next step in the transfer move is to push one of the self-intersections of $A$ created by the first finger move across $\partial_+W'$ by a finger move as in Figure~\ref{fig:transfer-move-1}. This finger move is guided by an arc in $A$ from $\partial_+W$ to $\partial_+W'$, and creates a new order~1 intersection $p'\in A\pitchfork W'$. (The $4$-ball shown in Figures~\ref{fig:transfer-move-Before-1}--\ref{fig:transfer-move-2} is a neighborhood of the union of this guiding arc together with $W$ and $W'$.) 
We may ensure that $p'$ has the same sign as $p$ did by choosing this guiding arc to approach from the correct side $\partial_+W'$ in $A$ (Exercise~\ref{ex:transfer-push-sign}), and in fact $t_{p'}=t_p$ by Exercise~\ref{ex:transfer-push-tree}. 
As illustrated by the blue Whitney circle in Figure~\ref{fig:transfer-move-1-with-order-2-boundary},
$p'$ and $q$ admit an order~2 framed Whitney disk by Exercise~\ref{ex:transfer-order-2-w-disk}. 

The construction has also created order~1 intersections $r,s\in A\pitchfork A$ which admit a framed embedded Whitney disk $V$, appearing ``underneath'' the horizontal sheet in Figure~\ref{fig:transfer-move-2}.
This $V$ has a pair of oppositely-signed intersections with $A$ that admit a framed order~2 Whitney disk whose boundary is indicated in blue in the figure
(Exercise~\ref{ex:order-2-disk-underneath}).

To see that this transfer move can be iterated to convert all algebraically cancelling pairs of intersections into geometrically cancelling pairs, observe that disjointly embedded guiding arcs in $A$ can be found for all pairs of Whitney disks, and the new order~1 $V$-Whitney disks are supported near these arcs and the original Whitney disks' boundaries.  
The order~2 Whitney disks created in the construction do not need to be controlled since they can only create new intersections of order $\geq 2$. 

To complete the proof that the vanishing of $\tau_1(A)$ implies that $A$ supports an order~2 framed Whitney tower it just remains to prove the following lemma describing the realization of INT relations for $\RP^2$s, which was used to achieve all algebraically cancelling pairs:
\begin{lem}\label{lem:realize-RP2-INT}
Let $\cW$ be an order~1 framed Whitney tower on $A$, and let $R:\RP^2\imra X$ represent $a\in\pi_1X$ with $a^2=1$. Then, after one finger move, $A$ supports an order~1 framed Whitney tower $\cW'$ such that $[t(\cW')]=[t(\cW)+(a,\lambda_0(A,R)+\omega_2(R))]\in\widetilde{\cT_1}$.
\end{lem}
\begin{proof}
Perform a finger move on $A$ guided by a circle which is isotopic to $R(\RP^1)$ representing $a$.
We may assume that this circle is disjoint from all Whitney disks in $\cW$.
The new pair of self-intersections created by this finger move admit a clean local framed Whitney disk $W$ whose boundary is indicated in green in Figure~\ref{finger-move-green-w-arcs}.
\begin{figure}[h]
\centerline{\includegraphics[scale=.35]{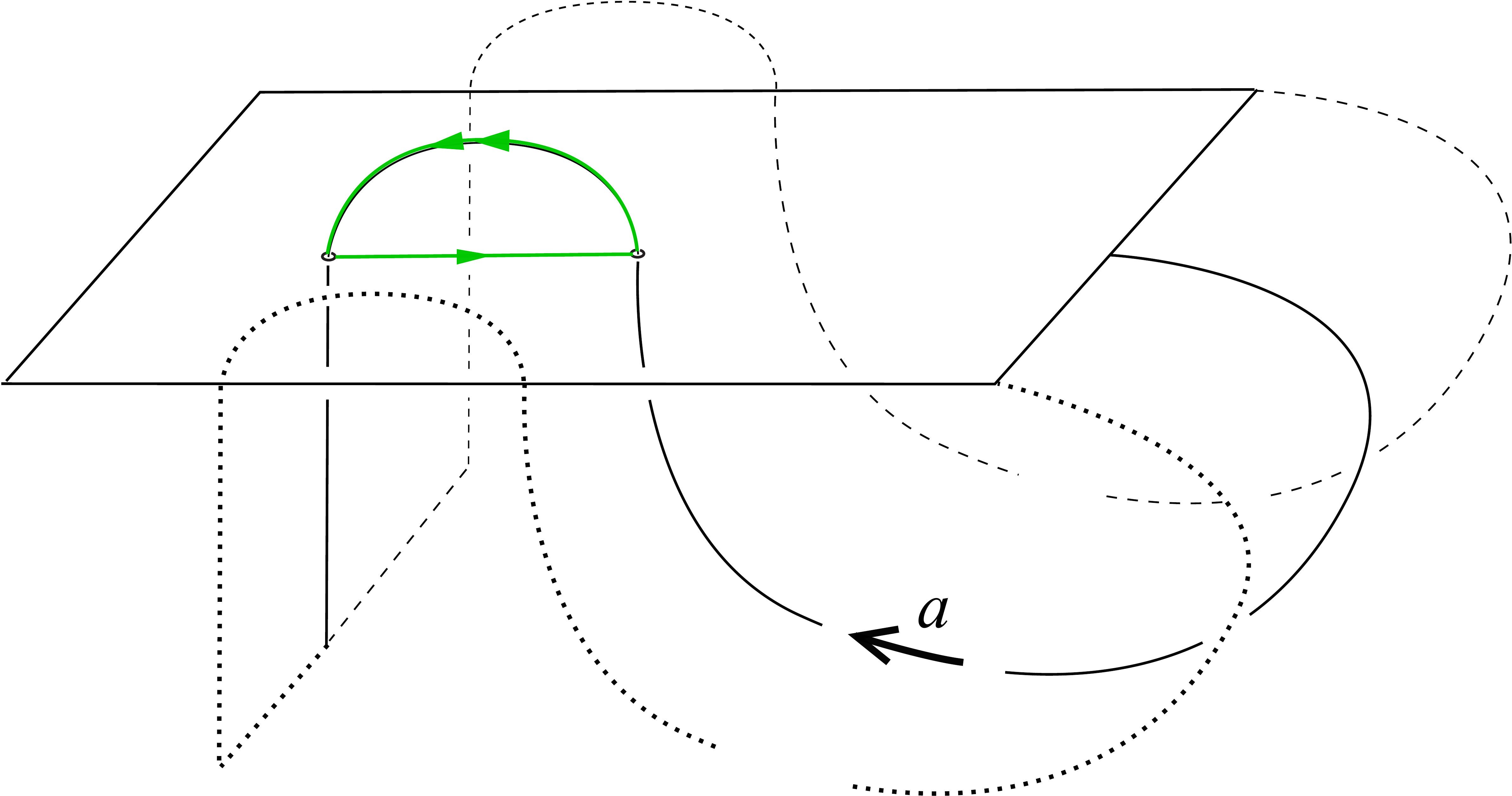}}
         \caption{After the finger move: A neighborhood of the circle representing $a$. The thick arrow indicates where the circle extends outside the illustrated local coordinates.}
         \label{finger-move-green-w-arcs}
\end{figure}

Figure~\ref{sheet-choice-preimage-arcs-realization} shows the preimages in red of a different choice of Whitney arcs that induce the opposite sheet choice as $W$. (Compare with Figure~\ref{fig:sheet-choice-preimage-arcs}, but note that here preimages of the same point are aligned vertically, as opposed to diagonally in Figure~\ref{fig:sheet-choice-preimage-arcs}.) 
\begin{figure}[h]
\centerline{\includegraphics[scale=.2]{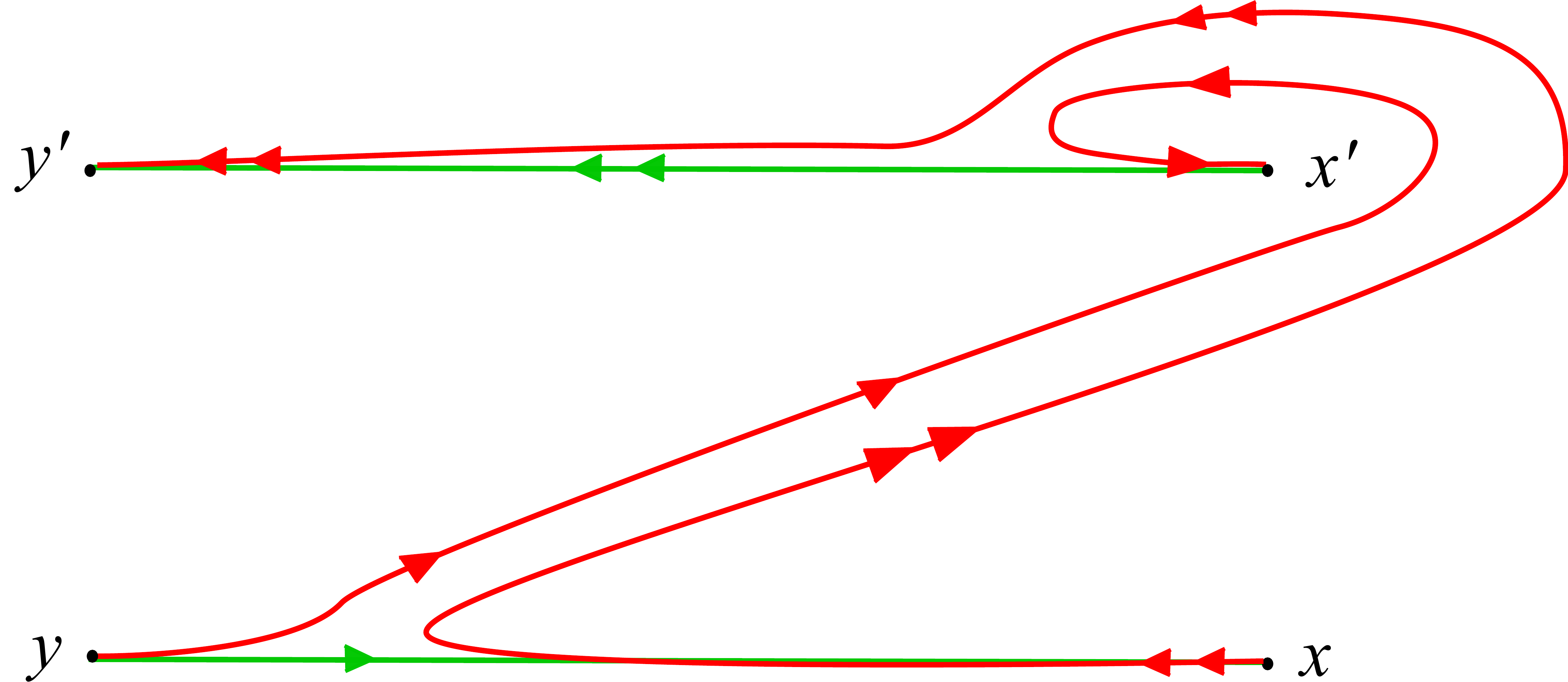}}
         \caption{}
         \label{sheet-choice-preimage-arcs-realization}
\end{figure}

Decomposing $\RP^2$ as the union $\textrm{Mb}\cup_{\partial \textrm{Mb}=\partial D^2}D^2$ of a M\"{o}bius band neighborhood Mb around 
$\RP^1$ and a disk $D^2$, want to use $R(D^2)$ as a Whitney disk $W'$ whose boundary is the image of the red arcs (see Figure~\ref{RP2-w-disk-with-double-point-loop}). To do this first observe that we can assume that $R$ restricts to an embedding on Mb. Now extend an isotopy between 
$R(\RP^1)$ and the finger move circle to a homotopy of $R(\RP^2)$ which restricts to an isotopy of $R(\textrm{Mb})$. To use the $R(D^2)$ as $W'$ we just need the image of the red arcs to align with $\partial R(\textrm{Mb})$.
This can be accomplished by appropriately choosing the twisting of $A$ in $S^1\times B^3$ along the finger move circle, which is the core of $R(\textrm{Mb})$.

By general position we may assume that $A\cap R(\RP^2)\subset R(D^2)$, so $A\pitchfork W'=A\pitchfork R(\RP^2)$. After arranging that $\omega(W')=0$ by boundary-twists, if needed, it follows that changing $\cW$ to $\cW'$ by replacing $W$ by $W'$ yields the change $[t(\cW')]=[t(\cW)+(a,\lambda_0(A,R)+\omega_2(R))]\in\widetilde{\cT_1}$. Here the $(a,\omega_2(R))$ term comes from boundary-twisting (and interior-twisting) as needed to make $W'$ framed
(as in section~\ref{subsubsection:towards-alg-cancellation}, any interior twists contribute trivially in $\widetilde{\cT_1}$).
\end{proof}

\begin{figure}[h]
\centerline{\includegraphics[scale=.35]{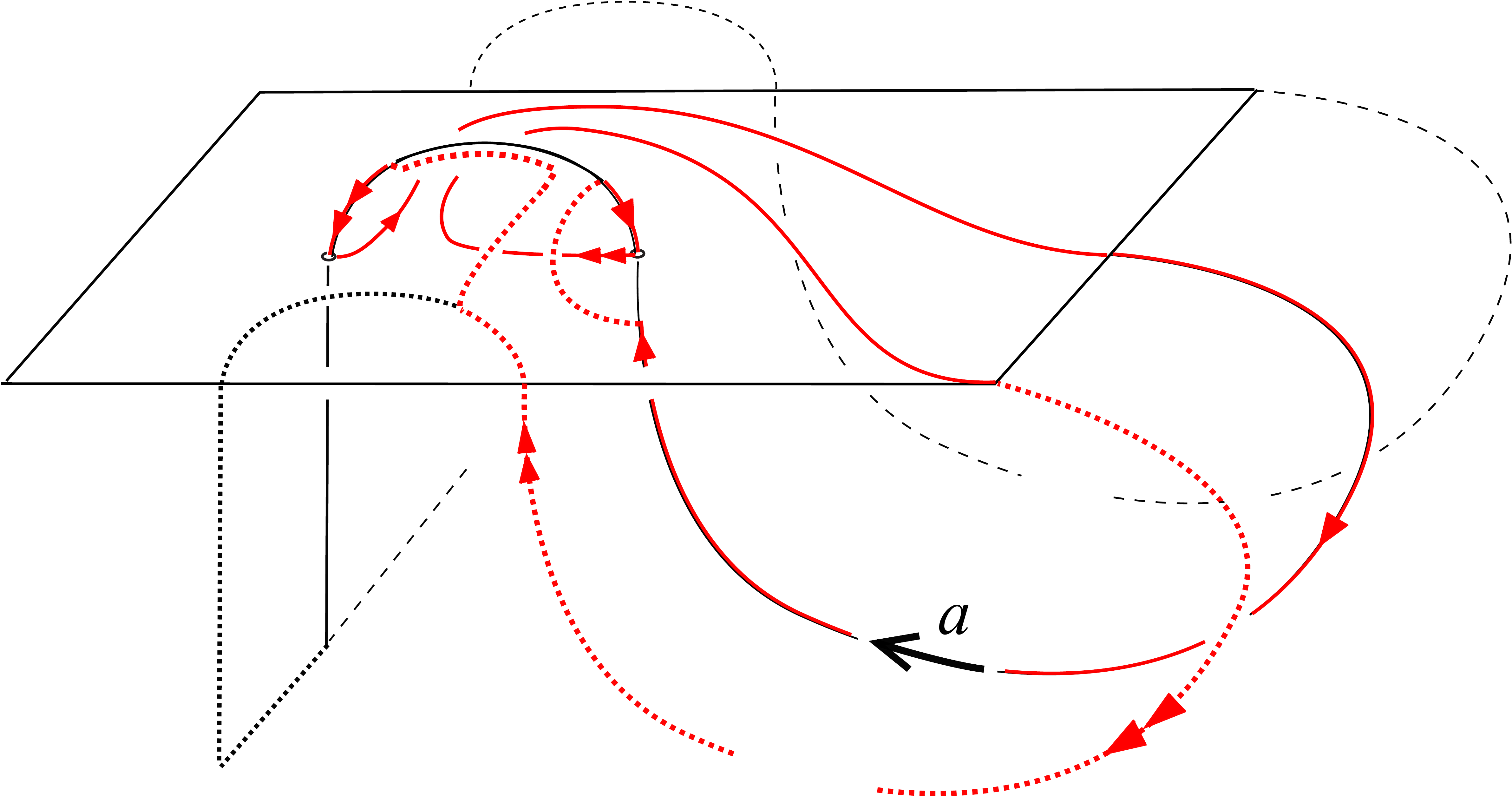}}
         \caption{The image under $A$ of the red arcs in Figure~\ref{sheet-choice-preimage-arcs-realization}. Solid arcs are in the present, dotted arcs are in the future, and dashed arcs are in the past.}
         \label{RP2-w-disk-with-double-point-loop}
\end{figure}

\subsection{Splitting twisted Whitney towers}\label{subsec:split-w-towers}
A framed Whitney tower is \emph{split} if the set of singularities in the interior
of any Whitney disk consists of either a single point, or a single boundary arc of a Whitney disk, or is empty.
This can always be arranged, as observed in Lemma~13 of \cite{ST2} (Lemma~3.5 of \cite{S1}), by performing finger moves along Whitney disks guided by arcs
connecting the Whitney disk boundary arcs (see Figure~\ref{split-w-tower-with-trees-fig}). Implicit in this construction is that the finger moves preserve the Whitney disk twistings
(by not twisting relative to the Whitney disk that is being split -- see Figure~\ref{twist-split-Wdisk-fig}).
A Whitney disk $W$ is \emph{clean} if the interior of $W$ is embedded and disjoint from the rest of the Whitney tower.
In the setting of twisted Whitney towers, it
simplifies the combinatorics of controlled manipulations to use ``twisted'' finger moves to similarly split twisted Whitney disks 
into $\pm 1$-twisted
clean Whitney disks.

We call a twisted Whitney tower \emph{split} if all of its non-trivially twisted Whitney disks are clean and have twisting 
$\pm 1$, and all of its framed Whitney disks are split in the usual sense (as for framed Whitney towers).

Recall from section~\ref{subsec:trees-for-w-disks-and-ints} our convention of embedding into a Whitney tower the trees associated to unpaired intersections and Whitney disks.
\begin{lem}[\cite{CST1,ST2}]\label{lem:split-w-tower}
If $A$ supports an order $n$ twisted Whitney tower $\cW$, then $A$ is homotopic (rel $\partial$) to 
$A'$ which supports a split order $n$ twisted
Whitney tower $\cW'$, such that:
\begin{enumerate}
\item\label{lem-item:framed-split-trees} The sub-multiset of signed framed trees $\sum_p \ \epsilon_p \cdot  t_p \subset t(\cW)$ is isomorphic to 
the sub-multiset of signed framed trees $\sum_{p'} \ \epsilon_{p'} \cdot  t_{p'} \subset t(\cW')$.
\item\label{lem-item:twisted-split-trees} Each $\omega(W_J)\cdot  J^\iinfty$ in $t(\cW)$ gives rise to exactly $|\omega(W_J) |$-many $\pm 1\cdot J^\iinfty$ in $t(\cW')$, 
where each twisting coefficient $\pm 1$ of the $J^\iinfty$ in $t(\cW')$
has the same sign as the twisting $\omega(W_J)$ of the original $W_J\subset\cW$.
\end{enumerate}
\end{lem}

\begin{proof}
\begin{figure}
\centerline{\includegraphics[scale=.6]{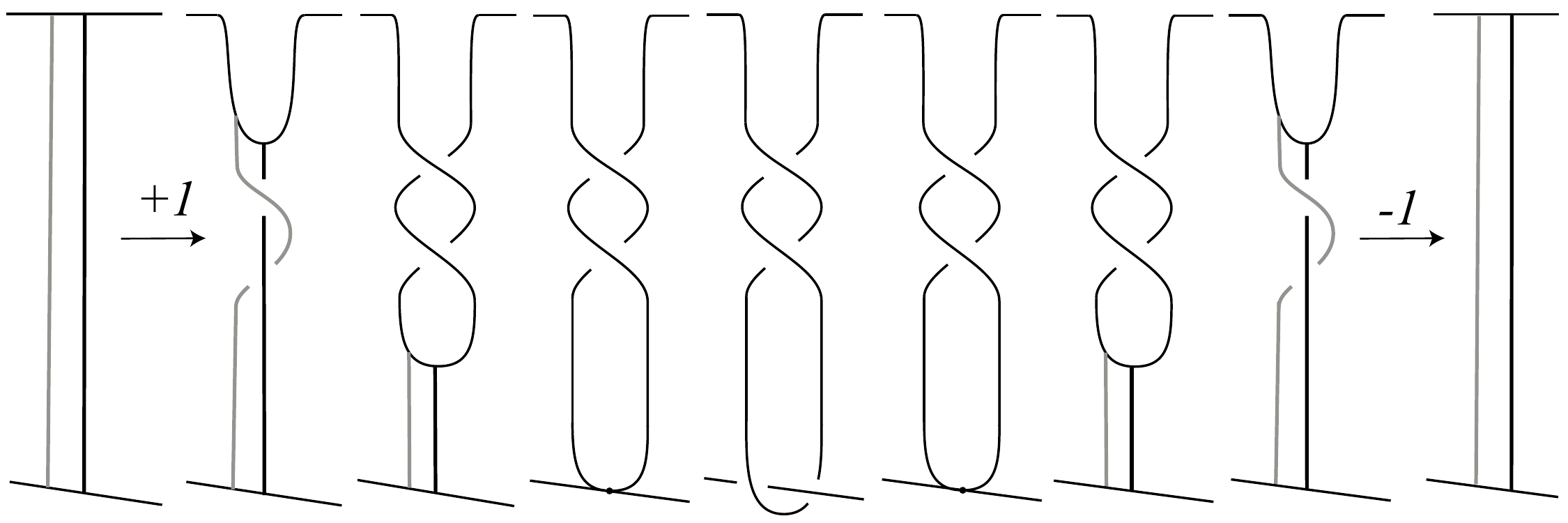}}
         \caption{A neighborhood of a twisted finger move which splits a Whitney disk into two Whitney disks. 
         The vertical black arcs are slices of
         the new Whitney disks, and the grey arcs are slices of extensions of the Whitney sections.
         The finger-move is supported in a neighborhood of an arc in the original Whitney disk running 
         from a point in the Whitney disk boundary on the ``upper'' surface sheet to a point 
         in the Whitney disk boundary on the ``lower'' surface sheet. (Before the finger-move this guiding arc would 
         have been visible in the middle picture as a vertical black arc-slice of the original
         Whitney disk.)}
         \label{twist-split-Wdisk-fig}
\end{figure}
Illustrated in
Figure~\ref{twist-split-Wdisk-fig} is a local picture of a twisted finger move, which splits one Whitney disk into two, while also changing twistings.
If the original Whitney disk in Figure~\ref{twist-split-Wdisk-fig} was framed, then the two new Whitney disks will have twistings $+1$ and $-1$, respectively. In general, if the arc guiding the finger move splits the twisting of the original Whitney disk into $\omega_1$ and $\omega_2$ zeros of the extended Whitney section, then the two new Whitney disks will have twistings $\omega_1+1$ and
$\omega_2-1$, respectively. Thus, by repeatedly splitting off framed corners into $\pm 1$-twisted Whitney disks, any 
$\omega$-twisted Whitney disk ($\omega \in\Z$) can be split into $|\omega |$-many $+1$-twisted or $-1$-twisted clean Whitney disks, together with split framed Whitney disks containing any interior intersections in the original twisted Whitney disk. Combining this with the untwisted splitting \cite[Lem.13]{ST2} of the framed Whitney disks illustrated in Figure~\ref{split-w-tower-with-trees-fig} gives the result.
\end{proof}


\subsection{The Whitney move IHX relation}\label{subsec:w-move-IHX}
Suppose that $\cW$ is a split Whitney tower on $A$. Let $p$ be an unpaired intersection with $t_p\subset \cW$.
Define the \emph{split subtower} $\cW_p\subset\cW$ to be the union of the Whitney disks containing the trivalent vertices of $t_p$ together with sheets of $A$ around the boundary arcs of each of these Whitney disks whose boundary lies in a sheet of $A$. These sheets inherit the indices of the order~0 surfaces containing them, so the univalent labels of $t_p$ are unchanged. By construction each Whitney disk in $\cW_p$ is framed and embedded.

\begin{figure}[ht!]
        \centerline{\includegraphics[scale=.7]{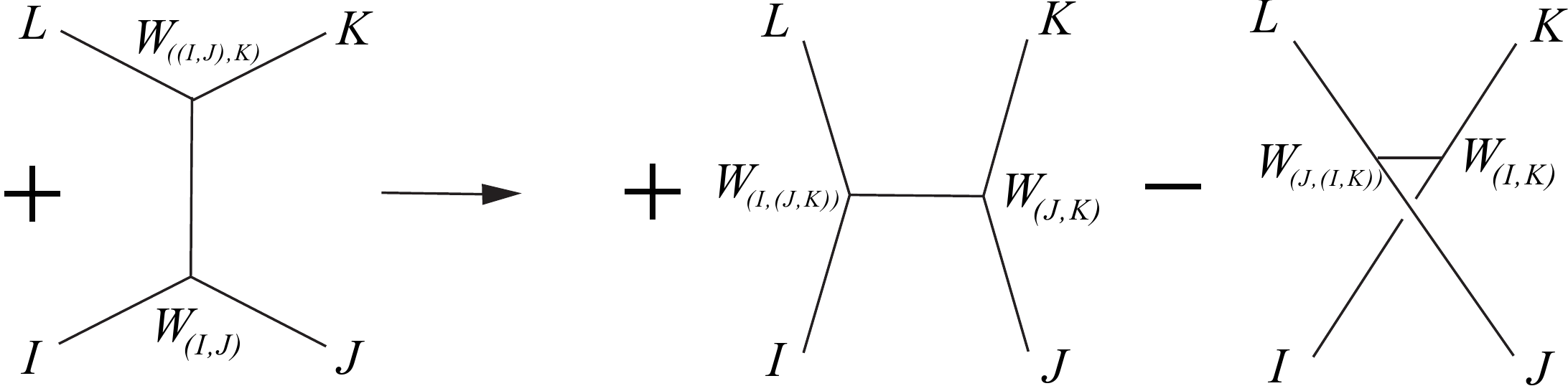}}
        \caption{The \emph{framed Whitney move} IHX \emph{relation} 
        replaces a split subtower whose signed tree looks locally like the one on the left
        with a pair of nearby disjoint split subtowers whose signed trees
        look locally like the trees on the right.}
        \label{IHX-trees-fig}

\end{figure}
\begin{lem}[Whitney move IHX relation]\label{lemma:w-move-IHX}
Let $\cW_p$ be a split subtower in a split Whitney tower $\cW$,
and let $W_{((I,J),K)}$ be a Whitney disk in $\cW_p$ so that
$t_p$ looks locally like the leftmost tree in
Figure~\ref{IHX-trees-fig} near $W_{((I,J),K)}$. Then $\cW$ can be modified in a
regular neighborhood $\nu(\cW_p)$ of $\cW_p$ yielding a split
Whitney tower on the same order 0 surface sheets, with $\cW_p$
replaced by disjoint split subtowers $\cW_{p'}$ and $\cW_{p''}$
contained in $\nu(\cW_p)$ such that the signed trees $t_{p'}$ and
$t_{p''}$ are as pictured on the right hand side of
Figure~\ref{IHX-trees-fig}.
\end{lem}

\begin{figure}[h]
        \centerline{\includegraphics[scale=.7]{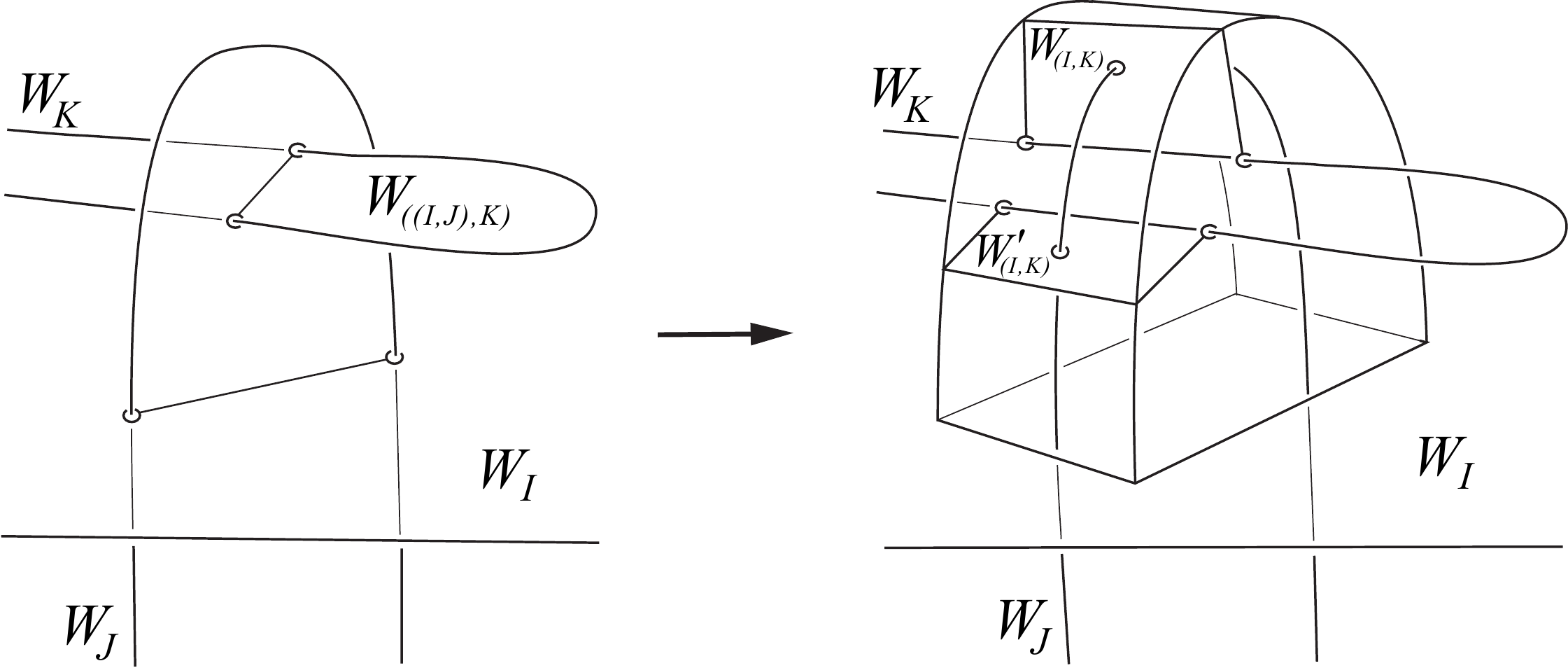}}
        \caption{The Whitney move IHX construction starts with a $W_{(I,J)}$ Whitney move
        on $W_I$. Note that the intersection $p=W_{((I,J),K)}\pitchfork W_L$ is {\em not} shown in this
        figure (and is suppressed in subsequent figures as well).}
        \label{IHX-W-move1and2-fig}
\end{figure}
\begin{figure}[h]
        \centerline{\includegraphics[scale=.65]{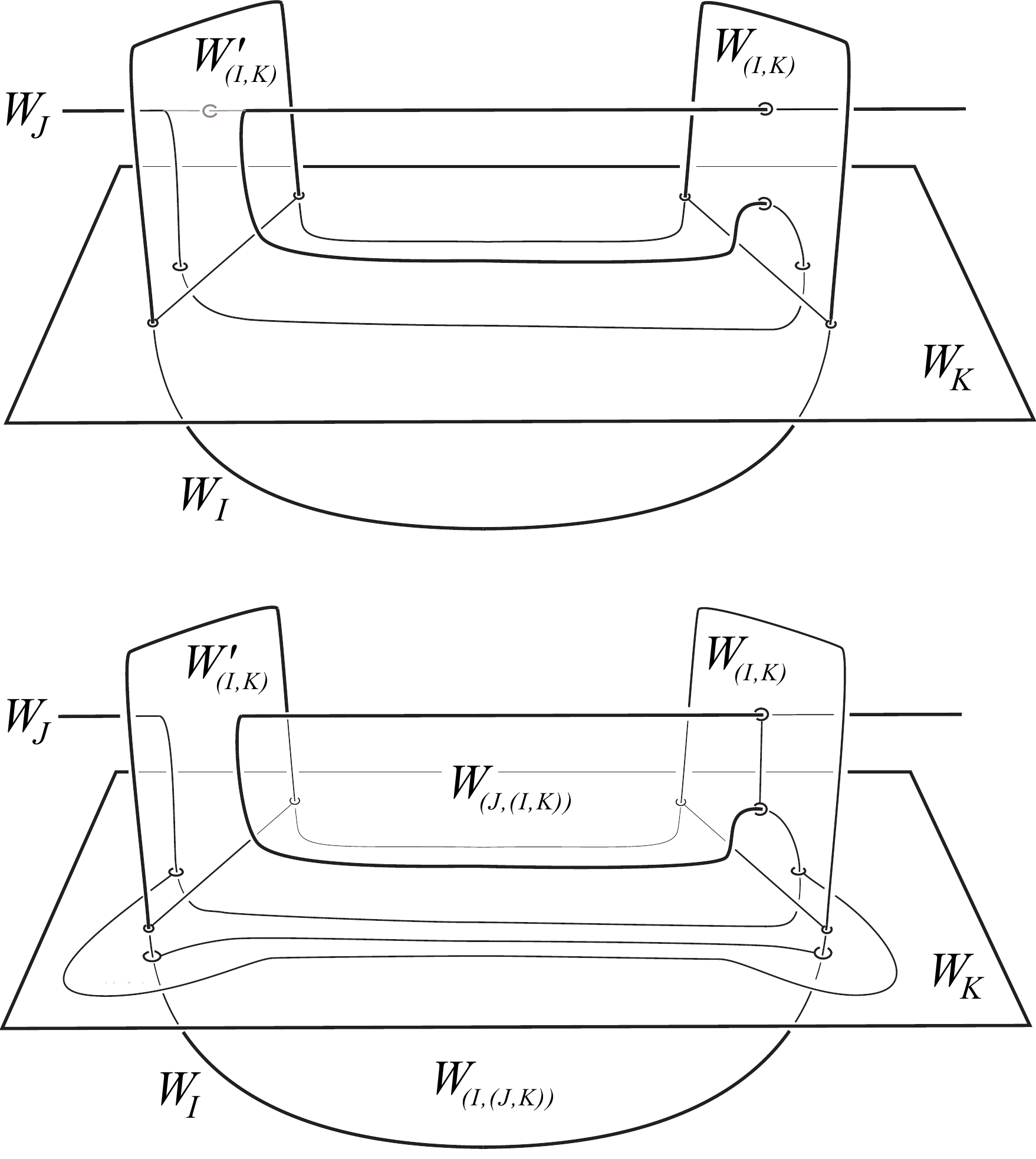}}
        \caption{The intersection point $W_J\cap W_{(I,K)}'$ is
        `transferred' via a finger move (top)
        to create a cancelling pair $W_J\cap
W_{(I,K)}$ paired by $W_{(J,(I,K))}$ at the cost of also creating $W_J\cap W_K$ paired by
$W_{(J,K)}$ and $W_I\cap W_{(J,K)}$ paired by $W_{(I,(J,K))}$ (bottom).}
        \label{IHX-transfer-move1and1A-fig}
\end{figure}

\begin{proof}
As a preliminary step, observe that the labelling of the trivalent vertices in the left-hand tree in Figure~\ref{IHX-trees-fig} indicates that the unpaired intersection $p\in\cW_p$ corresponds to an edge in the $L$-subtree. By applying the move of Figure~\ref{fig:move-int-puncture-in-tree} in section~\ref{subsec:moving-unpaired-int-edge} we may assume that in fact $p=W_{((I,J),K)}\pitchfork W_L$, so that the edge corresponding to $p$ is the $L$-labeled edge in the lefthand tree. This will simplify later steps in the construction. For visual clarity $p$ will be suppressed from view in the figures.

Now do the $W_{(I,J)}$
Whitney move on $W_I$ (see Figure~\ref{IHX-W-move1and2-fig}). This eliminates the cancelling
pair of intersections between $W_I$ and $W_J$ at the cost of
creating two cancelling pairs of intersections between $W_I$ and
$W_K$ which we pair by Whitney disks $W_{(I,K)}$ and $W_{(I,K)}'$ which are meridional disks to $W_J$
as illustrated in Figure~\ref{IHX-W-move1and2-fig}.

The new
Whitney disks $W_{(I,K)}$ and $W_{(I,K)}'$ each have a single
interior intersection with $W_J$ and the next step is to
``transfer'' (as illustrated in the upper part of
Figure~\ref{IHX-transfer-move1and1A-fig}) the intersection point
$W_J\cap W_{(I,K)}'$ to create a cancelling pair $W_J\cap
W_{(I,K)}$ paired by $W_{(J,(I,K))}$ at the cost of also creating
$W_J\cap W_K$ paired by $W_{(J,K)}$ and $W_I\cap W_{(J,K)}$ paired
by $W_{(I,(J,K))}$ (as illustrated in the lower part of
Figure~\ref{IHX-transfer-move1and1A-fig}). Note that
Figure~\ref{IHX-transfer-move1and1A-fig} differs from
Figure~\ref{IHX-W-move1and2-fig} by a change of coordinates
which brings the sheet of $W_K$ into the ``present'' slice of
3--space.

\begin{figure}[h]
        \centerline{\includegraphics[scale=.65]{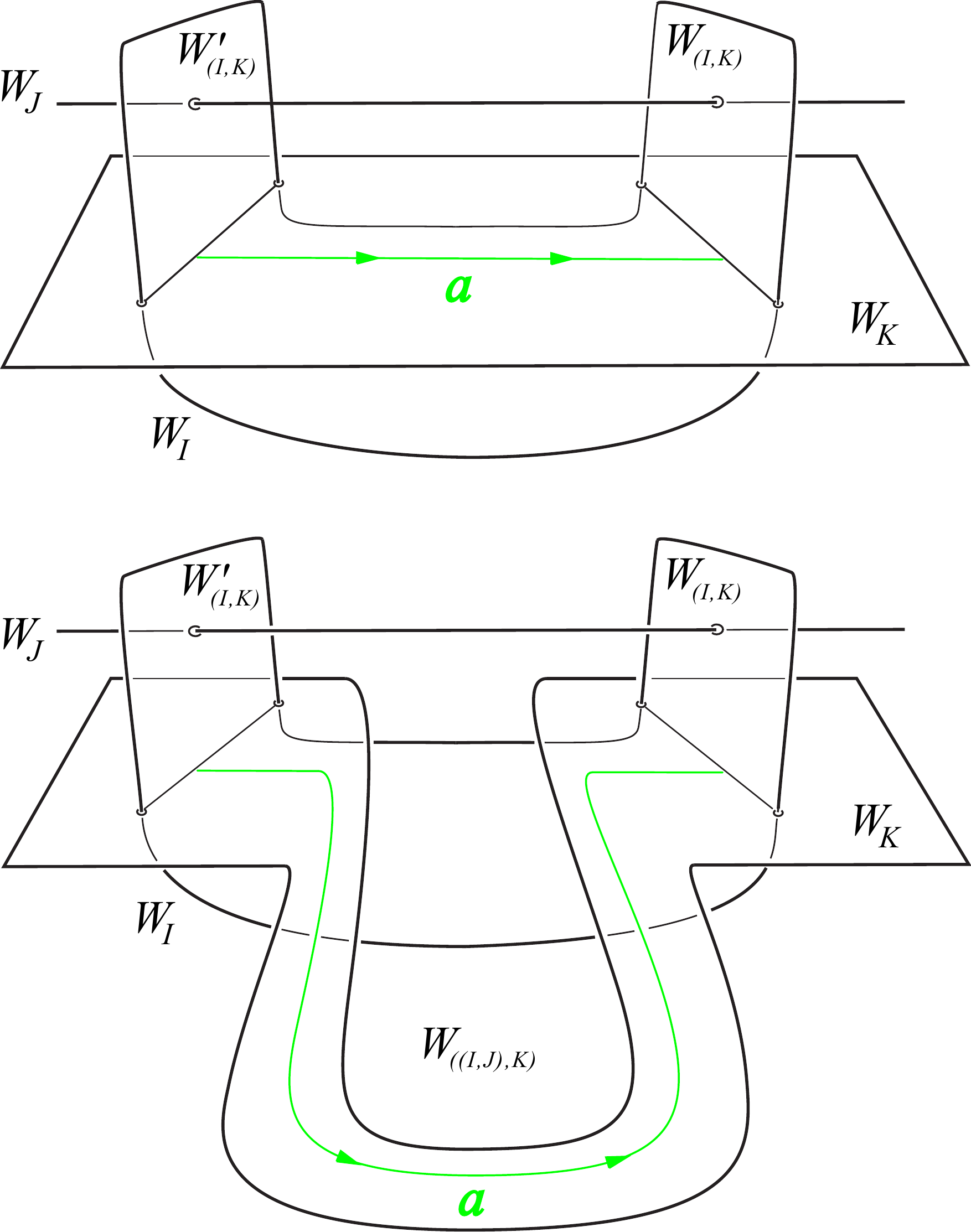}}
        \caption{The transferring finger move is guided by an arc $a$ (shown in green)
        which can be taken to run along what used to be the part of the boundary arc of
        $W_{((I,J),K)}$ lying in $W_K$. The bottom picture is the same as the top picture but indicating
         where $W_{((I,J),K)}$ used to be.}
        \label{IHX-transfer-move2and3-fig}

\end{figure}

\begin{figure}[h]
        \centerline{\includegraphics[scale=.65]{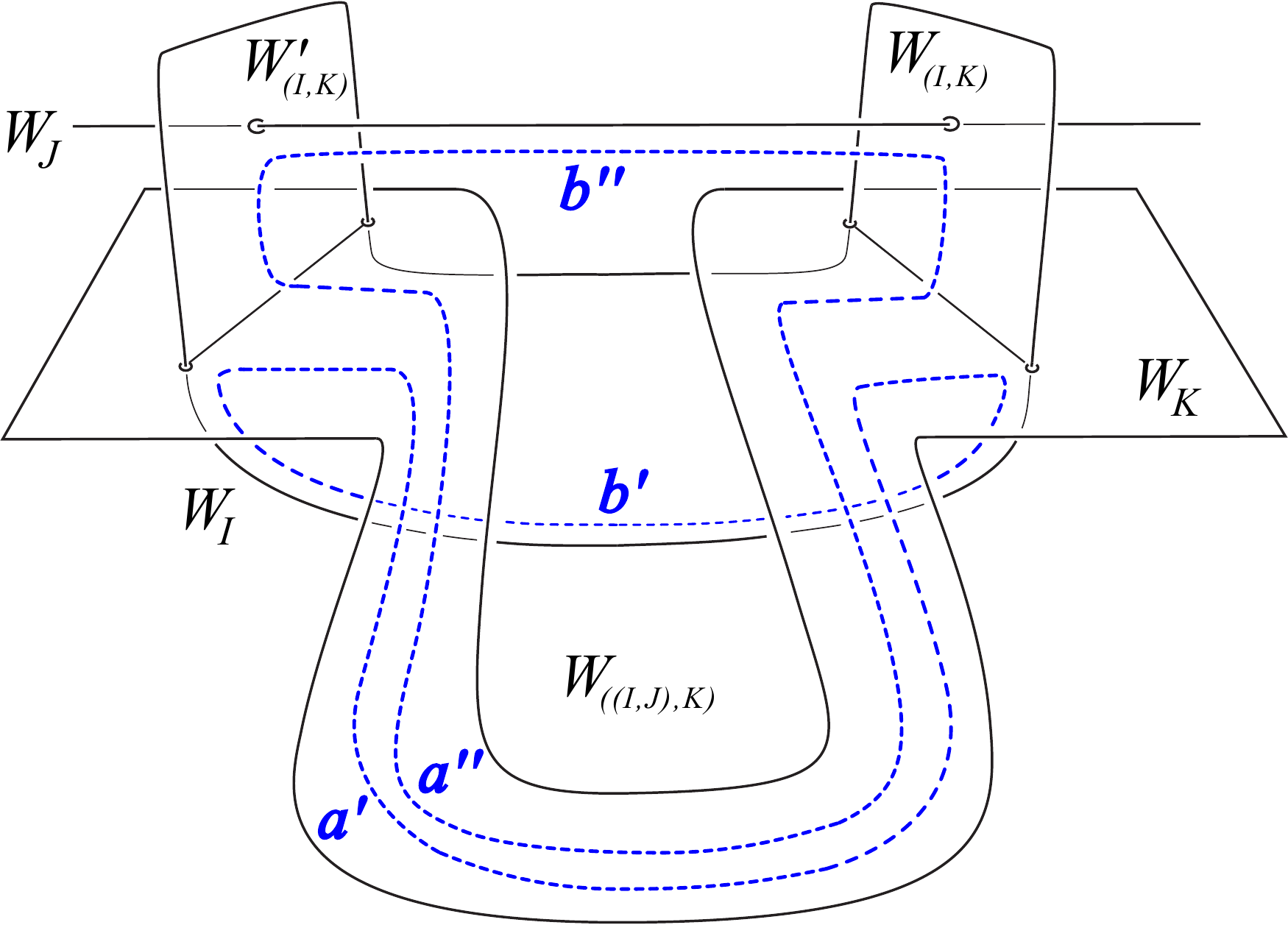}}
        \caption{Before the transfer move: New Whitney disks $W_{(I,(J,K))}$
        and $W_{(J,(I,K))}$, whose boundaries are the unions of arcs $a'\cup b'$
        and $a''\cup b''$ (see also Figure
        ~\ref{IHX-W-move1and2B-fig}),
         will be created from parallel copies of the old $W_{((I,J),K)}$.}
        \label{IHX-transfer-move3B-fig}
\end{figure}

This transfer move is combinatorially the same as in section~\ref{subsubsection:towards-geo-cancellation} but here applied to higher-order sheets. The important thing to note here is that the finger move is
guided by an arc $a$ (see Figure~\ref{IHX-transfer-move2and3-fig})
from $\partial W_{(I,K)}'$ to $\partial W_{(I,K)}$ in $W_K$ and we
can take this arc to run along what used to be the part of
$\partial W_{((I,J),K)}$ lying in $W_K$. This is illustrated in
the lower part of Figure~\ref{IHX-transfer-move2and3-fig} which
gives a better picture of the situation before the finger move is
applied.
\begin{figure}[h]
        \centerline{\includegraphics[scale=.55]{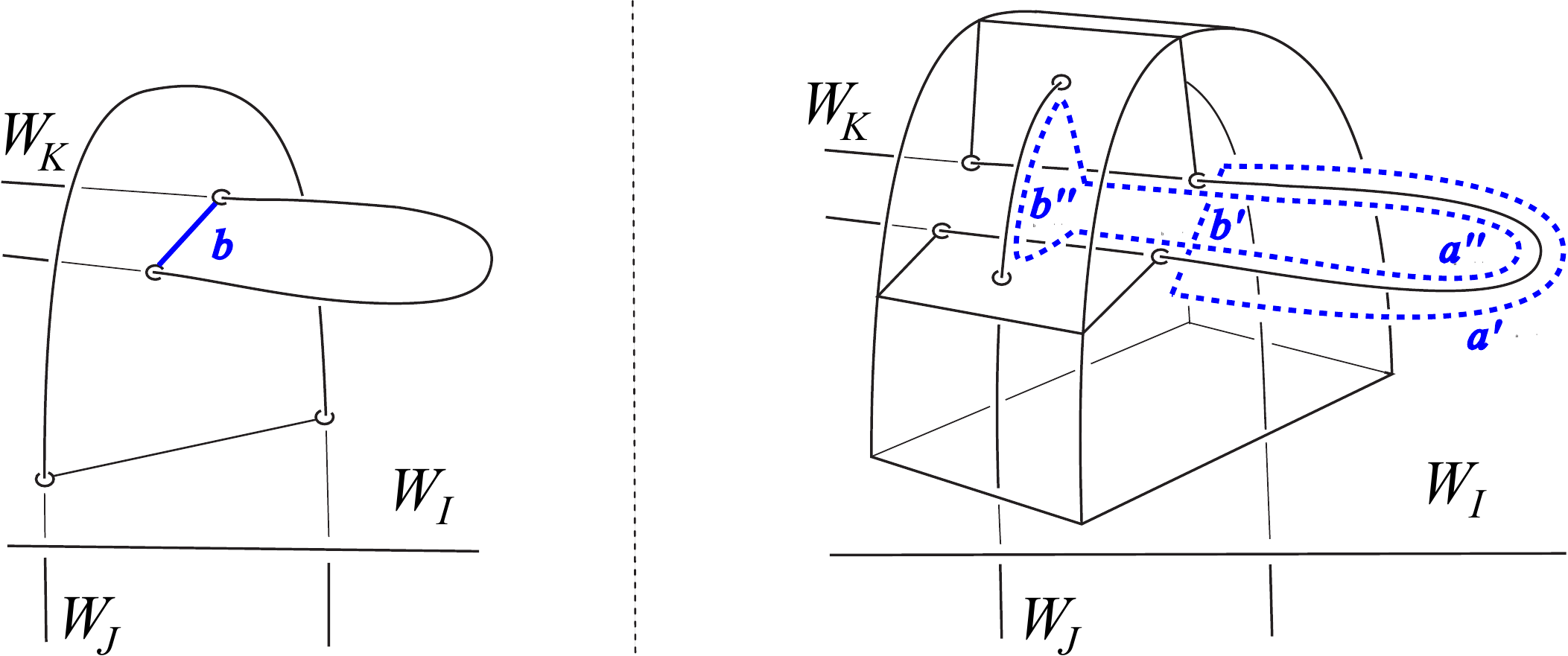}}
        \caption{Applying the transfer move to the right-hand side will create
        new Whitney disks $W_{(I,(J,K))}$
        and $W_{(J,(I,K))}$, whose boundaries are the unions of arcs $a'\cup b'$
        and $a''\cup b''$ (see also Figure
        ~\ref{IHX-transfer-move3B-fig}),
         from parallel copies of the old $W_{((I,J),K)}$ shown on the left.}
        \label{IHX-W-move1and2B-fig}

\end{figure}
The Whitney disks $W_{(I,(J,K))}$ and $W_{(J,(I,K))}$ are taken
to be parallel copies of the old $W_{((I,J),K)}$ as follows: The
boundary of $W_{(I,(J,K))}$ (resp. $W_{(J,(I,K))}$) consists of
arcs $a'$ and $b'$ (resp. $a''$ and $b''$), where $a'$ and $a''$
are tangential push-offs of $a$ in $W_K$ and $b'$ and $b''$ are
normal push-offs of what was the boundary arc $b$ of
$W_{((I,J),K)}$ in $W_{(I,J)}$. This is shown in both
Figure~\ref{IHX-transfer-move3B-fig} and
Figure~\ref{IHX-W-move1and2B-fig}, where again it is easier to
picture things before the transferring finger move. Since
$W_{((I,J),K)}$ was framed and embedded, $W_{(I,(J,K))}$ and
$W_{(J,(I,K))}$ can be formed from two disjoint parallel copies of
$W_{((I,J),K)}$ which each intersect $W_L$ in a single point as $W_{((I,J),K)}$ did.

Thus exactly two new unpaired intersection
points $p'$ and $p''$ have been created (near where $p$ was) with corresponding trees
$t(p')$ and $t(p'')$ as shown locally in the right hand side of
Figure~\ref{IHX-trees-fig}. After the transferring finger move,
the $W'_{(I,K)}$ Whitney move can be done (on either sheet)
without affecting anything else. Finally, $W_I$, $W_J$ and $W_K$
will need to be split since they now each contain two boundary
arcs of Whitney disks. Splitting $W_I$, $W_J$ and $W_K$ down into
the lower order Whitney disks (as in Figure~\ref{split-w-tower-with-trees-fig} and Lemma~\ref{lem:split-w-tower}(\ref{lem-item:framed-split-trees}))
yields the two split subtowers $\cW_{p'}$ and $\cW_{p''}$.
\end{proof}

\subsection{The Whitney move twisted IHX relation}\label{subsec:w-move-twistedIHX}
Recall the twisted IHX/Jacobi relation in the even order twisted tree groups (section~\ref{subsec:twisted-tree-groups}):
$$\textrm{I}^\iinfty=\textrm{H}^\iinfty+\textrm{X}^\iinfty-\langle \textrm{H},\textrm{X}\rangle$$
where I, H and X denote rooted trees which correspond to the three terms of a Jacobi identity, ie.~they differ locally as in Figure~\ref{IHX-trees-fig}, where we now are interpreting the trees in the figure as being rooted trees with each of the $L$-subtrees containing a root univalent vertex. Here we are using Roman font in the twisted IHX relation rather than the italic font used in section~\ref{subsec:twisted-tree-groups} to clarify the distinction between the lefthand term ``I'' of the Jacobi identity and the sub-tree ``$I$'' in each of the three terms in Figure~\ref{IHX-trees-fig}.

We will modify the proof of the above Whitney move IHX relation (Lemma~\ref{lemma:w-move-IHX}) to show how to locally replace a twisted Whitney disk whose associated signed $\iinfty$-tree is I$^\iinfty$ by two twisted Whitney disks whose associated signed $\iinfty$-trees are H$^\iinfty$ and X$^\iinfty$, respectively, together with a single unpaired intersection $p$ having signed tree $-\langle \textrm{H},\textrm{X}\rangle$.

Suppose that $\cW$ is a split Whitney tower on $A$. Let $W$ be a Whitney disk in $\cW$ with rooted tree $J\subset \cW$.
Then the \emph{split subtower} $\cW_W\subset\cW$ is defined to be the union of the Whitney disks containing the trivalent vertices of $J$ together with sheets of $A$ around the boundary arcs of each of these Whitney disks whose boundary lies in a sheet of $A$. Note that $W\subset\cW_W$. If $W$ is twisted, then by construction all the other Whitney disks in $\cW_W$ are framed and the twisting of $W$ is $\pm 1$ by the definition of a split Whitney tower, and $W_J$ contains the $\iinfty$-labeled root vertex of $J^\iinfty\subset\cW_W$.

\begin{lem}\label{lem:w-move-twistedIHX}
Let $\cW_W$ be a split subtower for an $\epsilon$-twisted Whitney disk $W$ in a split Whitney tower $\cW$, 
with
the $\iinfty$-tree I$^\iinfty$ associated to $W$ 
looking locally like the leftmost tree in
Figure~\ref{IHX-trees-fig} (with the $L$-subtree containing a root univalent vertex), and denote by H$^\iinfty$ and X$^\iinfty$ the $\iinfty$-trees which only differ locally from I$^\iinfty$ as in the first two trees in the right side of the equation in Figure~\ref{IHX-trees-fig}. 

Then $\cW$ can be modified in a
regular neighborhood $\nu(\cW_W)$ of $\cW_W$ yielding a split
Whitney tower on the same order 0 surfaces, with $\cW_W$
replaced by disjoint split subtowers $\cW_{W'}$, $\cW_{W''}$ and $\cW_p$
contained in $\nu(\cW_W)$ such that:
\begin{enumerate}
\item
 $W'$ is an $\epsilon$-twisted Whitney disk with associated $\iinfty$-tree H$^\iinfty$, and

 \item
 $W''$  is an $\epsilon$-twisted Whitney disk with associated $\iinfty$-tree X$^\iinfty$, and

\item
the unpaired intersection $p$ has sign $-\epsilon$ and associated tree $\langle \textrm{H},\emph{X}\rangle$.

\end{enumerate}

\end{lem}

\begin{proof}

We first consider the case where the $L$-labeled sub-tree is order zero, which means that $L$ is just the $\iinfty$-label, and the upper trivalent vertex of the I$^\iinfty$-tree in Figure~\ref{IHX-trees-fig} corresponds to the clean $\epsilon$-twisted $W_{((I,J),K)}$, with $\iinfty$-tree $\textrm{I}^\iinfty=((I,J),K)^\iinfty$. Then the construction in the proof of Lemma~\ref{lemma:w-move-IHX}, which starts by performing a Whitney move on the framed Whitney disk
$W_{(I,J)}$ corresponding to the lower trivalent vertex of the I-tree, exchanges $W_{((I,J),K)}$ for two Whitney-parallel Whitney disks $W_{(I,(J,K))}$ and $W_{(J,(I,K))}$ (Figures~\ref{IHX-transfer-move3B-fig} and~\ref{IHX-W-move1and2B-fig}). In our current setting these two new Whitney disks $W'$ and $W''$ inherit the $\epsilon$-twisting of $W$, and have associated twisted trees H$^\iinfty$ and X$^\iinfty$. 
And because of the twisting $\epsilon=\pm 1$, there is now a single new intersection $p=W'\cap W''$.
In the construction of Lemma~\ref{lemma:w-move-IHX} $W'$ inherits the orientation of $W$, and $W''$ inherits the opposite orientation, but in our current setting the signs of the twisted trees associated to $W$ and $W''$ are given by their twisting $\epsilon=\omega(W')=\omega(W'')$ which does not depend on orientations.
So we are free to choose orientations on $W'$ and $W''$ so that the sign of $p$ is $-\epsilon$.
Finally, splitting yields the desired split subtowers $\cW_{W'}$, $\cW_{W''}$ and $\cW_p$.

Now we consider the case where the $L$-labeled sub-tree has positive order. This means that the Whitney disk $W_{((I,J),K)}$ is framed and contains a boundary arc of some higher-order Whitney disk $V$ corresponding to a trivalent vertex of $L$ which is adjacent to the trivalent vertex for $W_{((I,J),K)}$ in the I$^\iinfty$ tree. 
As above, the construction of Lemma~\ref{lemma:w-move-IHX} exchanges $W_{((I,J),K)}$ for two parallel Whitney disks $W_{(I,(J,K))}$ and $W_{(J,(I,K))}$, but now these three Whitney disks are all framed. The pair of intersections between $W_{((I,J),K)}$ and some sheet $W_{L_1}$ that was paired by $V$ gives rise to two pairs of intersections $W_{L_1}\pitchfork W_{(I,(J,K))}$ and $W_{L_1}\pitchfork W_{(J,(I,K))}$.
To continue with the construction we need to find Whitney disks for these new intersection pairs. This will be accomplished using parallels of $V$, and we state this as a lemma for future reference:
\begin{lem}\label{lem:parallel-V-disks}
Let $V$ be an embedded Whitney disk for a pair of intersections between embedded sheets $A$ and $B$, and let $A'$ and $A''$ be pairwise disjoint parallel copies of $A$. 
\begin{enumerate}

\item\label{item:lem-parallel-framed-V-disks}
If $V$ is framed, then there exist pairwise disjointly embedded framed Whitney disks $V'$ pairing $A'\pitchfork B$, and $V''$ pairing $A''\pitchfork B$, which are parallel to $V$.

\item\label{item:lem-parallel-framed-V-disks}
If $V$ is $\epsilon$-twisted ($\epsilon=\pm 1$), then there exist embedded $\epsilon$-twisted Whitney disks $V'$ pairing $A'\pitchfork B$, and $V''$ pairing $A''\pitchfork B$, which are parallel to $V$, such that $V'\pitchfork V''$ is a single point.

\end{enumerate}
\end{lem} 
\begin{figure}[h]
        \centerline{\includegraphics[scale=.4]{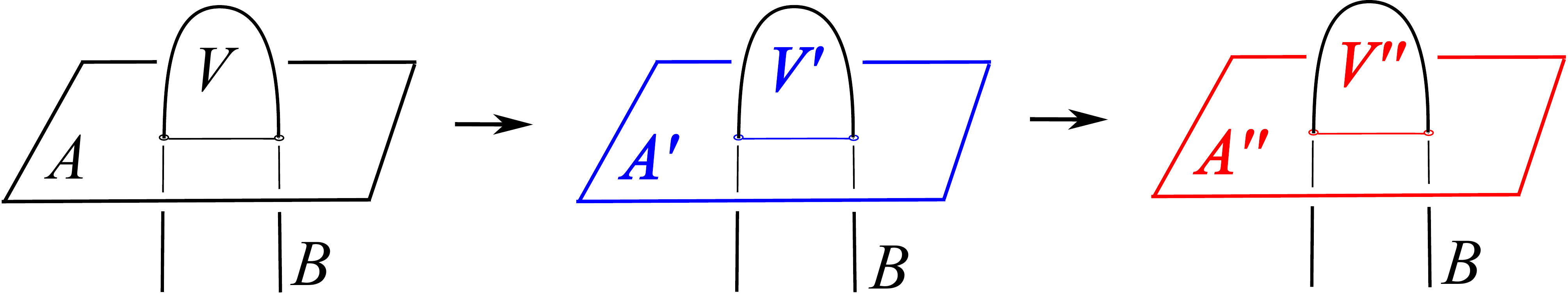}}
        \caption{A neighborhood of the framed embedded Whitney disk $V$ pairing sheets $A$ and $B$ contains (in nearby time-slices) parallel Whitney disks $V'$ and $V''$ pairing the intersections between $B$ and parallels $A'$ and $A''$ of $A$. Each of $V'$ and $V''$ has one boundary arc which is tangent to a boundary arc of $V$, and the other boundary arc normal to the boundary arc of $V$ in $A$. Only $B$ extends into all time-slices; the three Whitney disks and other three sheets are each contained in single time slices.}
        \label{fig:parallel-V-disks}

\end{figure}

\begin{proof}
Figure~\ref{fig:parallel-V-disks} shows the construction of $V'$ and $V''$ which are Whitney-parallel to $V$ in the case that $V$ is framed.
If $V$ is $(\pm 1)$-twisted, then Figure~\ref{fig:parallel-V-disks} is only accurate near the Whitney disk boundaries, but since $V'$ and $V''$ are Whitney parallel it follows that $V'$ and $V''$ intersect in a single point since they inherit the $(\pm 1)$-twisting of $V$.
\end{proof}

Returning to the twisted IHX construction, and simplifying notation by writing $A=W_{((I,J),K)}$, and $B=W_{L_1}$, and
$A'=W_{(I,(J,K))}$, and $A''=W_{(J,(I,K))}$, Lemma~\ref{lem:parallel-V-disks} gives us $V'$ pairing $A'\pitchfork B$, and $V''$ pairing $A''\pitchfork B$.

If $V$ was the $\epsilon$-twisted $W$, then we take $W'=V'$, $W''=V''$, and $p=W'\pitchfork W''$. Splitting then yields the desired split subtowers $\cW_{W'}$, $\cW_{W''}$ and $\cW_p$. 

If $V$ was not the $\epsilon$-twisted $W$, then $V$ intersected some sheet $W_{L_2}$ in a pair of intersections paired by some Whitney disk $V_2$, and hence the parallels $V'$ and $V''$ also each intersect $W_{L_2}$ in a pair of intersections. Again applying Lemma~\ref{lem:parallel-V-disks} (but now with $A:=V_2$ and $B:=W_{L_2}$), yields $V'_2$ and $V''_2$ pairing $V'\pitchfork W_{L_2}$ and $V''\pitchfork W_{L_2}$, with $V'_2$ and $V''_2$ inheriting the twisting or the intersections that 
$V_2$ had. 

We continue to apply Lemma~\ref{lem:parallel-V-disks} in this way as needed until reaching the $\epsilon$-twisted $W=V_n$, which gives rise to $W'=V'_n$, $W''=V''_n$ and $p=W'\pitchfork W''$. Splitting then yields the desired split subtowers $\cW_{W'}$, $\cW_{W''}$ and $\cW_p$.
\end{proof}


\subsection{Outline of twisted order-raising obstruction theory proof}\label{subsec:order-raising-proof-sketch}
Recall from sections~\ref{subsec:twisted-tree-groups},~\ref{subsec:geometry-of-relations} and~\ref{subsec:twisted-order-n-invariant} the statement of Theorem~\ref{thm:twisted-order-raising}: A link $L\subset S^3$ bounds an order $n$ twisted $\cW\subset B^4$ with $\tau_n^\iinfty(\cW)=0\in\cT_n^\iinfty$ if and only if $L$ bounds an order $n+1$ twisted Whitney tower.

Theorem~\ref{thm:twisted-order-raising} is essential to the classification of order~$n$ twisted Whitney towers in the 4-ball discussed in Section~\ref{sec:twisted-order-n-classification-arf-conj}, and this section gives a brief outline of the proof, as given in \cite[Thm.1.9]{CST1}.
This proof, which depends on the above Whitney move IHX Lemmas~\ref{lemma:w-move-IHX} and \ref{lem:w-move-twistedIHX}, generalizes part of the construction of an order~2 framed Whitney tower in section~\ref{subsec:tau1-vanishes-equals-order-2-w-tower}.

The ``if'' direction of the theorem holds since by definition any order $n+1$ twisted Whitney tower is also an order $n$ twisted Whitney tower with no unpaired order $n$ intersections or order $n/2$ twisted Whitney disks.
For the ``only if'' direction, we will sketch how the realization of the relations in $\cT_n^\iinfty$ by geometric constructions can be used to arrange that all unpaired intersections and twisted Whitney disks occur in ``algebraically cancelling'' pairs (representing inverse elements in $\cT_n^\iinfty$), which can then be exchanged for ``geometrically cancelling'' intersection pairs (admitting Whitney disks) and framed Whitney disks of order $n/2$. 
The strategy is analogous to the order~2 construction in section~\ref{subsec:tau1-vanishes-equals-order-2-w-tower}, but now we are dealing with twisted Whitney disks and higher-order intersections, and hence more complicated trees (although here in $B^4$ we do not have to keep track of edge decorations and INT relations).

We can assume that $\cW$ is split (Lemma~\ref{lem:split-w-tower}), and that $t(\cW)$ contains no framed trees of order $>n$, and no twisted trees of order $>n/2$ (Exercise~\ref{ex:eliminate-higher-order-trees}).

The condition ${\tau_n}^\iinfty(\cW)=0\in{\cT_n}^\iinfty$
means that in the free abelian group on order~$n$ framed trees and order $n/2$ twisted trees $t(\cW)$ lies in the span of the relators in section~\ref{subsec:twisted-tree-groups} which define ${\cT_n}^\iinfty$.
As usual we consider the trees $t(\cW)$ to be embedded in $\cW$.

As described in \cite[Sec.4.1]{CST1}, using also \cite[Sec.4]{ST2} and \cite{CST}, there are three main steps to constructing an order~$n+1$ twisted Whitney tower from $\cW$: 

First, controlled modifications of $\cW$ realizing the relators, as discussed in section~\ref{subsec:geometry-of-relations}, are used to arrange that the order $n$ trees and order $n/2$ $\iinfty$-trees in $t(\cW)$ all occur in isomorphic oppositely-signed algebraically canceling pairs (see the start of section~4 of \cite{CST1}). 

Secondly, using the above Whitney move IHX Lemmas~\ref{lemma:w-move-IHX} and \ref{lem:w-move-twistedIHX}, all these paired trees are converted into pairs of ``simple'' (right- or left-normed) trees by IHX constructions. These simple trees are characterized by the property that every trivalent vertex is adjacent to a univalent vertex. This corresponds to every Whitney disk having a boundary arc on an order zero disk. (The reason for this step will be explained momentarily in the description of the third step.) After this step the all trees still occur in algebraically cancelling pairs.

The final third step uses an iterated higher-order variation of the ``transfer move'' used in section~\ref{subsubsection:towards-geo-cancellation} to achieve geometric cancellation for algebraically canceling order~1 pairs. This step is described in detail for algebraically cancelling pairs of framed trees in \cite[Lem.15]{ST2}, and
for twisted order~$n/2$ Whitney disks in \cite[Sec.4.1]{CST1}, and
requires that all cancelling pairs are simple, as arranged in the second step. 
The reason for this requirement has to do with the connectivity of sheets that is needed for the transfer move. This can be seen by observing that the construction of the two order~2 Whitney disks in Figure~\ref{fig:transfer-move-2} depends on all three sheets of $A$ being connected. In higher orders it turns out that having just two of these sheets connected suffices to iterate the move finitely many times until eventually terminating with the desired result, provided one starts at an ``end'' of a simple tree.

After this third step the new layer of order $n+1$ Whitney disks have uncontrolled intersections, but all of these new intersections are of order $\geq n+1$. And the construction combining the twisted Whitney disk pairs into framed Whitney disks (Figures 21--22 in \cite[Sec.4.1]{CST1}) creates only new twisted Whitney disks of order $>n/2$, which are supported near the original twisted Whitney disk pairs, along with intersections of order $\geq n$ among these new twisted Whitney disks. 
Hence an order~$n+1$ twisted Whitney tower has been created.


Analogous order-raising intersection-obstruction theories are described in \cite[Sec.4.4]{CST1} for order $n$ framed Whitney towers, in \cite[Thm.6]{ST2} for non-repeating Whitney towers, and in \cite[Thm.6.17]{CST6} for ``$k$-repeating'' Whitney towers.

%
%
%
%


\subsection{Whitney disk orientations and the AS relation}\label{subsec:w-disk-orientation-choices-AS}
This section explains why the signs associated to trees for unpaired intersections in a Whitney tower only depend on the orientations of the underlying order~0 surface and the ambient $4$-manifold modulo the AS antisymmetry relations.

Let $\cW$ be a Whitney tower in an oriented $4$-manifold $X$.
The order~0 surface supporting $\cW$ comes with a fixed orientation, and arbitrary orientations on the Whitney disks in $\cW$ are chosen and fixed.
Fix a choice of either the positive or negative corner convention described in section~\ref{subsec:w-tower-tree-orientations}. Here we will refer to this fixed convention choice as ``our corner convention''.

For $p\in W_I\pitchfork W_J$ an unpaired intersection between Whitney disks $W_I$ and $W_J$ in $\cW$, 
the corresponding tree $t_p=\langle I,J\rangle$ from section~\ref{subsec:trees-for-w-disks-and-ints} is embedded in $\cW$ in a way that satisfies our corner convention, and the orientations of the trivalent vertices of $t_p$ are taken to be induced by the Whitney disk orientations.

We will consider here the effect on the signed oriented tree $\epsilon_p\cdot t_p$ of switching the fixed orientation choice on any of the Whitney disks corresponding to the trivalent vertices of $t_p$. The conclusion will be that any such orientation switch corresponds to an AS relation (Figure~\ref{fig:ASandIHXtree-relations}), so that modulo AS relations $\epsilon_p\cdot t_p$ only depends on the orientation of the underlying order~0 surface.  

First of all,
the sign $\epsilon_p=\pm 1$ is determined by comparing the concatenated orientations of $W_I$ and $W_J$ with the orientation of $X$ at $p$. 
So in the case that $W_I\neq W_J$, then switching the orientation of one of either $W_I$ or $W_J$ will switch the trivalent orientation of the corresponding trivalent vertex of $t_p$, and will switch the sign $\epsilon_p$, meaning that $\epsilon_p\cdot t_p$ is changed by an AS relation. 

In the case that $W_I=W_J$, then switching the orientation of $W_I=W_J$ will switch the trivalent orientation of $t_p$ at both of the trivalent vertices adjacent to the edge passing through $p$, and will not change the sign $\epsilon_p$, meaning that $\epsilon_p\cdot t_p$ is changed by two AS relations.

Now consider the effect of switching the orientation on a Whitney disk $W_{K_1}$ other than $W_I$ or $W_J$, so that $W_{K_1}$ contains a boundary arc of a higher-order Whitney disk $W_{(K_1,K_2)}$ which pairs intersections $q$ and $r$ between $W_{K_1}$ and some $W_{K_2}$, as in the left side of Figure~\ref{fig:W-disk-orientation-switch}.

The orientation switch on $W_{K_1}$ changes the cyclic orientation at the trivalent vertex in $W_{K_1}$, and it also switches the signs of $q$ and $r$.
This switching of the signs of $q$ and $r$ means that the embedding of $t_p$ needs to changed near the trivalent vertex $v$ of $W_{(K_1,K_2)}$
in order to preserve our corner convention. The effect of this convention-preserving change in the embedding of $t_p$ near $v$ is to switch the cyclic orientation of $t_p$ at $v$ as shown in the right side of Figure~\ref{fig:W-disk-orientation-switch}.
Note that the orientation of the Whitney disk $W_{(K_1,K_2)}$ does not change.


\begin{figure}[h]
        \centerline{\includegraphics[scale=.3]{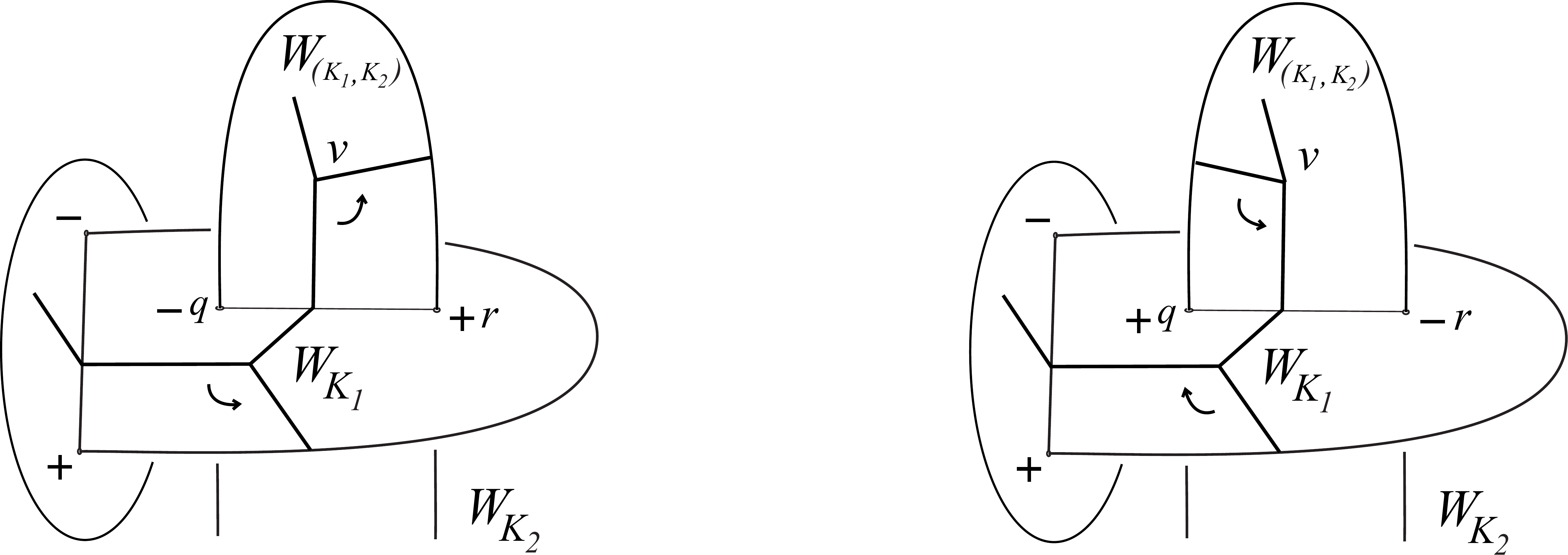}}
        \caption{Before (left) and after (right) the orientation switch on $W_{K_1}$, using the positive corner convention. Any interior intersection(s) in $W_{(K_1,K_2)}$ are not shown.}
        \label{fig:W-disk-orientation-switch}

\end{figure}

Since the orientation of $W_{(K_1,K_2)}$ is not changed, the sign $\epsilon_p$ of $p\in W_I\pitchfork W_K$ is unchanged.
(It is possible that $W_{(K_1,K_2)}=W_I$ or $W_{(K_1,K_2)}=W_J$.)
So the result of switching the orientation of any one $W_{K_1}$ (with $W_{K_1}\neq W_I$, and $W_{K_1}\neq W_J$) while preserving our corner convention is the same as applying \emph{two} AS relations to $t_p$, one at the trivalent vertex in $W_{K_1}$, and one at $v\in W_{(K_1,K_2)}$.

\subsection{Section~\ref{sec:appendix} Exercises}\label{subsec:appendix-exercises}

\subsubsection{Exercise} \label{ex:independence-of-interior-htpy}
In the proof of independence of $\tau_1(A)$ on Whitney disk interiors given in section~\ref{subsubsec:w-disk-interiors} show that the sphere $S=(W\setminus w)\cup (V'\setminus v')$ is homotopic to $S'=W\cup V'$ by a homotopy supported near $\gamma$.
HINT: Consider a homotopy from $S'$ to $S$ which pulls the common collars $w=v'$ slightly apart while perturbing $S'$ to be smooth, and then shrinks the union of the collars away from $\gamma$.

\subsubsection{Exercise} \label{ex:boundary-int-def-well-defined}
Check that the signed tree $t_p$ associated to 
$p\in
\partial_{\epsilon}W\cap\partial_{\delta}V$
in Equation~\ref{eq:boundary-int-tree}
near Figure~\ref{W-disk-boundary-int-fig} is well-defined, including the case that $W=V$.

\subsubsection{Exercise} \label{ex:cancelling-boundary-push-cases}
Check that in each of the other cases of Figure~\ref{fig:cancelling-boundary-push}
pushing $\partial_\delta W$ into $\partial_\epsilon V$ across the $\pm$ self-intersection of $A$ paired by $V$ creates
an algebraically canceling pair of signed trees.

\subsubsection{Exercise} \label{ex:pairing-choice}
In the setting of section~\ref{subsubsec:pairings} 
and Figure~\ref{pairing-choice-change-fig}, 
check that the order~1 decorated trees corresponding to the intersections $A\pitchfork V'$ cancel with those corresponding to the oppositely-signed parallel copy of $V'$ in $V$.

\subsubsection{Exercise} \label{ex:RP2-lambda0-indeterminacy}
In the discussion in section~\ref{subsubsec:RP2-INT} of the pairing $\lambda_0(A,R)$ for $A:S^2\imra X$ and $R:\RP^2\imra X$ with the generator of $\pi_1\RP^2$ mapping to $a\in\pi_1X$, show that changing the choice of sheet-changing path through $p\in A\pitchfork R$ changes $g_p$ by right multiplication by $a^n$ and changes $\epsilon_p$ by multiplication by $(-1)^n$ for some integer $n$.

\subsubsection{Exercise} \label{ex:sheet-choice-w-disks-exist}
From section~\ref{subsubsec:sheet-choices}: 
For $A:S^2\imra X$, let $p$ and $q$ be a positive and a negative transverse self-intersection of $A$, and
denote the preimages by $A^{-1}(p)=\{x,x'\}\subset S^2$, and $A^{-1}(q)=\{y,y'\}\subset S^2$. 

If $p$ and $q$ have common group element $g_p=a=g_q$,
then any Whitney disk $W$ pairing $p$ and $q$ induces a pairing of $\{x,x'\}$ with $\{y,y'\}$ since each arc of $\partial W$ runs between a sheet of $A$ around $p$ and a sheet of $A$ around $q$. 

Check that Whitney disks exist for \emph{both} of the two pairing choices $x\leftrightarrow y, x'\leftrightarrow y'$ and $x\leftrightarrow y', x'\leftrightarrow y$ if and only if $a^2=1\in\pi_1X$.

\subsubsection{Exercise} \label{ex:RP2}
From section~\ref{subsubsec:sheet-choices} and Figure~\ref{fig:sheet-choice-preimage-arcs}: 
Check that the union $A(D)\cup W\cup V$ defines the image of a map $R:\RP^2\to X$ which sends the generator of $\pi_1\RP^2$ to $a\in\pi_1X$.

\subsubsection{Exercise} \label{ex:V=W-near-p-and-q-for-rho-prime-lemma}
Let $W$ and $V$ be Whitney disks pairing the same intersections $p$ and $q$ between surfaces in a 4-manifold,
and suppose that $\partial W=\partial V$ near $p$ and near $q$. 
Convince yourself that by an isotopy it can be arranged that $W=V$ near $p$ and near $q$ without creating any
new intersections.

\subsubsection{Exercise} \label{ex:transfer-push-sign}
Check that in the setting of Figures~\ref{fig:transfer-move-Before-1},~\ref{fig:transfer-move-first-finger-move}
and~\ref{fig:transfer-move-1} it can be arranged that $p'$ has the same sign as $p$.

\subsubsection{Exercise} \label{ex:transfer-push-tree}
Check that in the setting of Figures~\ref{fig:transfer-move-Before-1},~\ref{fig:transfer-move-first-finger-move}
and~\ref{fig:transfer-move-1} it can be arranged that $p'$ has the same decorated tree as $p$.

\subsubsection{Exercise} \label{ex:transfer-order-2-w-disk}
Use the previous two exercises to check that $p'$ and $q$ in Figure~\ref{fig:transfer-move-1-with-order-2-boundary} admit an order~2 Whitney disk.

\subsubsection{Exercise} \label{ex:order-2-disk-underneath}
The construction has also created order~1 intersections $r,s\in A\pitchfork A$ which admit a framed embedded Whitney disk $V$, shown ``underneath'' the horizontal sheet in Figure~\ref{fig:transfer-move-2}.
Check that the pair of transverse intersections between $A$ and the embedded order~1 Whitney disk $V$ in Figure~\ref{fig:transfer-move-2} admits a framed order~2 Whitney disk (whose boundary is indicated in blue in the figure).

\subsubsection{Exercise} \label{ex:sheet-choice-whiskers}
In the paragraph before the proof of Lemma~\ref{lem:rho-prime}, check that the group element associated to any point in $A(D)\pitchfork A(D)'$ is $1$ or $a$.

\subsubsection{Exercise} \label{ex:splitting-preserves-edge-decorations}
Check that the operation of splitting a Whitney tower preserves edge decorations on trees.
Do this first for decorated order~1 trees (as in section~\ref{subsec:decorated-order-1-trees}), then generalize to higher-order trees.

\subsubsection{Exercise} \label{ex:IHX-signs}
Check that the signs of the trees created by the IHX construction in the proof of Lemma~\ref{lemma:w-move-IHX}
have the correct signs, as shown in the right side of Figure~\ref{IHX-trees-fig}.

\subsubsection{Exercise} \label{ex:IHX-preserves-edge-decorations}
Check that the IHX construction in the proof of Lemma~\ref{lemma:w-move-IHX}
also works for decorated trees. 

\subsubsection{Exercise}\label{ex:derive-non-repeating-order-1-thm-from-tau-1-proof}
Prove that $\lambda_1(A_1,A_2,A_3)$ from Theorem~\ref{thm:lambda1-vanishes} is a well-defined homotopy invariant which vanishes if and only if $A_1\cup A_2\cup A_3$ admits an order~2 non-repeating Whitney tower by adapting to this easier non-repeating setting the proof given in sections~\ref{subsec:tau1-well-defined} and~\ref{subsec:tau1-vanishes-equals-order-2-w-tower}
of the analogous statements for $\tau_1(A)$.

HINT: First show that $\lambda_1(A_1,A_2,A_3)$ is independent of the choice of order~1 non-repeating Whitney tower, following sections~\ref{subsubsec:w-disk-interiors}, \ref{subsubsec:w-disk-boundaries} and \ref{subsubsec:pairings} (but without worrying about Whitney disk twistings in section~\ref{subsubsec:w-disk-interiors}). Then one gets homotopy invariance exactly as in section~\ref{subsubsec:htpy-invariance-tau1}.
To get an order~2 non-repeating Whitney tower proceed as in
section~\ref{subsec:tau1-vanishes-equals-order-2-w-tower} to first achieve algebraic cancellation and then geometric cancellation of all order~1 intersections intersections.

\end{document}